\documentclass[a4paper,12pt,intlimits]{amsart}
\usepackage{amssymb}

\usepackage{graphicx}  \usepackage{epstopdf}

\usepackage[english]{babel}

\usepackage{amsmath}
\usepackage{comment}

\numberwithin{equation}{section}
\newtheorem{de}{Definition}[section]
\newtheorem{thm}{Theorem}[section]
\newtheorem{rem}[thm]{Remark}
\newtheorem{cor}[thm]{Corollary}
\newtheorem{prop}[thm]{Proposition}
\newtheorem{lem}[thm]{Lemma}

\renewcommand{\dim}{\noindent\textbf{Proof.} }

\newcommand{\finedim}{{\unskip\nobreak\hfil\penalty50
   \hskip2em\hbox{}\nobreak\hfil\mbox{$\Box$ \qquad}
   \parfillskip=0pt \finalhyphendemerits=0\par\medskip}}
\newcommand{\R}{\mathbb{R}}

\newcommand{\Z}{\mathbb{Z}}
\newcommand{\Om}{\Omega}

\newcommand{\al}{\alpha}

\newcommand{\xs}{\overline{x}}

\newcommand{\ts}{\overline{t}}

\newcommand{\p}{\partial}

\renewcommand{\theta}{\vartheta}
\renewcommand{\epsilon}{\varepsilon}
\newcommand{\ep}{\epsilon}
\newcommand{\I}{\mathcal{I}_1}

\newcommand{\us}{\overline{u}}

\renewcommand{\leq}{\leqslant}
\renewcommand{\le}{\leqslant}
\renewcommand{\geq}{\geqslant}
\renewcommand{\ge}{\geqslant}

\newcommand{\beq}{\begin{equation}}
\newcommand{\eeq}{\end{equation}}
\newcommand{\beqs}{\begin{equation*}}
\newcommand{\eeqs}{\end{equation*}}
\newcommand{\beqa}{\begin{eqnarray}}
\newcommand{\eeqa}{\end{eqnarray}}
\newcommand{\beqas}{\begin{eqnarray*}}
\newcommand{\eeqas}{\end{eqnarray*}}

\title[]{From  the Peierls-Nabarro model to  the equation of motion of the dislocation continuum}

\author{Stefania Patrizi and Tharathep Sangsawang}

\address[Stefania Patrizi and Tharathep Sangsawang]{
Department of Mathematics,
University of Texas at Austin,
2515 Speedway Stop C1200,
Austin, Texas 78712-1202, USA}

\email{spatrizi@math.utexas.edu} 
\email{tsangsaw@math.utexas.edu}

\subjclass[2010]{82D25, 35R09, 74E15, 35R11, 47G20.}
\keywords{Peierls-Nabarro model, nonlocal integro-differential equations,
dislocation dynamics, fractional Allen Cahn, homogenization.}

\begin{document}

\begin{abstract}
We consider a semi-linear integro-differential equation  in dimension one associated to the half Laplacian
whose solution represents the atom dislocation in a crystal.
The equation  comprises the evolutive version of the classical
Peierls-Nabarro model. 
We show that for  a large number of dislocations, the solution, properly rescaled, converges to the  solution of a 
well known equation called by Head \cite{H}
 ``the equation of motion of the dislocation continuum".  The limit equation is a model for the macroscopic 
 crystal plasticity  with density of dislocations. 
In particular, we recover the  so called Orowan's law which states that  dislocations move at a velocity proportional to the effective stress.
\end{abstract}
\maketitle

\section{Introduction}

In this paper we are  interested in studying  the behavior as $\ep\to0$ of the solution $u^\ep$ of 
the following 
 integro-differential equation:
\beq\label{uepeq}\begin{cases}
\delta\partial_t u^\ep=\I [u^\ep]-\displaystyle\frac{1}{\delta}W'\left(\frac{u^\ep}{\ep}\right)&\text{in }(0,+\infty)\times\R\\
u^\ep(0,\cdot)=u_0(\cdot)&\text{on }\R
\end{cases}\eeq
where $\ep,\,\delta>0$ are small scale parameters and  $\delta=\delta_\ep\to 0$ as $\ep\to0$, $W$ is a periodic potential and we denote by  $\I$ is the so-called fractional Laplacian 
 of order 1, $-(-\Delta)^\frac{1}{2}$, 
defined on the Schwartz class $\mathcal{S}(\R)$ by
\beq\label{halflapfourier}\widehat{{(-\Delta )^\frac{1}{2}v}\
}(\xi)=|\xi|\ \widehat{v}(\xi),\eeq where $\widehat{v}$ is the
Fourier transform of $v$.  It is well known, see e.g. \cite{stein},
 that $\I$  may be also represented as $$ \I[v](x)=\frac{1}{\pi}PV
\displaystyle\int_{\R}\displaystyle\frac{v(y)-v(x)}{(y-x)^{2}}dy,$$ where PV stands for principal value. 
 See also~\cite{s} or~\cite{dnpv} for a basic introduction
to the fractional Laplace operator. 

We assume that~$W$ is a multi-well potential with nondegenerate
minima at integer points.
More precisely, we suppose that
\begin{equation}\label{Wass}
\begin{cases}W\in C^{2,\beta}(\R)& \text{for some }0<\beta<1\\
W(u+1)=W(u)& \text{for any } u\in\R\\
W=0& \text{on }\Z\\
W>0 & \text{on }\R\setminus\Z\\
W''(0)>0.\\
\end{cases}
\end{equation}
On the function $u_0$ we assume
\begin{equation}\label{u^0ass}
\begin{cases}
u_0\in C^{1,1}(\R)\\
u_0 \text{ non-decreasing}.
\end{cases}
\end{equation}

Equation \eqref{uepeq} is a rescaled version of the so called Peierls-Nabarro model, which is a phase field model describing
dislocations. Dislocations are line defects in crystals. Their typical length is
of the order of  $10^{-6}m$ and their thickness of order of
$10^{-9}m$. When the material is submitted to shear stress, these
lines can move in the crystallographic planes (slip planes) and their dynamics
is one of the main explanation of the plastic behavior of metals. We refer the reader to the book \cite{hl} for a tour in the theory of dislocations.
Dislocations  can be described at several scales by different models:
\begin{itemize}
\item atomic scale (Frenkel-Kontorova model), 
\item microscopic scale (Peierls-Nabarro model), 
\item mesoscopic scale (Discrete dislocation dynamics), 
\item macroscopic scale (Elasto-visco-plasticity with density of dislocations). 
\end{itemize}

 Our goal in this paper is to understand the large scale limit of the Peierls-Nabarro model  for  a  {\em large} number of
 {\em 
  parallel straight edge dislocation lines in the same slip plane with the same Burgers' vector, moving with self-interactions}. The number of dislocations is of order $1/\ep$, while the distance between neighboring dislocations is (at microscopic scale) of order $1/\delta$. 
 Rescaling the Peierls-Nabarro model leads to  equation \eqref{uepeq}.  The model is explained in further details in Section \ref{nabarrrosec}.

We show that at macroscopic scale the density of dislocations is governed by the following evolution law:
\beq
\label{ubareq}\begin{cases}
\partial_t u=c_0\partial_x u\,\I[ u]&\text{in }(0,+\infty)\times\R\\
u(0,\cdot)=u_0&\text{on }\R
\end{cases}\eeq
where  $c_0>0$ is defined in the forthcoming \eqref{c0}. Under  assumption \eqref{u^0ass}, there exists a unique non-decreasing in $x$ viscosity solution $\us$ of 
\eqref{ubareq} (see Section \ref{prelimsec}).
Our main result is the following:
\begin{thm}\label{mainthm}
Assume \eqref{Wass} and \eqref{u^0ass}.   Let $u^\ep$ be the viscosity solution of~\eqref{uepeq}.
Then, $u^\ep$ converges locally uniformly in $(0,+\infty)\times\R$ to the viscosity solution $\us$ of~\eqref{ubareq}, as $\ep\to0$. 
\end{thm}
\begin{rem}
We do not assume any assumption  about how $\delta$ goes to 0 when $\ep\to0$.
\end{rem}
The limit equation \eqref{ubareq} represents the plastic flow rule for the   macroscopic crystal plasticity  with density of dislocations.
The theorem says  that in this regime, the plastic strain velocity $\partial_t u$ in (\ref{ubareq})
is proportional to the dislocation density $u_x$ times the effective stress $\I[u]$. This physical law is  known as  Orowan's equation, see e.g. \cite{sed} p.~3739.
 Equation 
\beq\label{limiteq}
\partial_t u=c_0\partial_x u\,\I[ u]
\eeq
is an integrated form of a model studied by Head \cite{H} for the self-dynamics of a dislocation density represented by $u_x$.
 Indeed, denoting $f=u_x$, differentiating \eqref{limiteq}, we see that,  at least formally, $f$ solves
\beq\label{transporteq}\partial_t f=c_0\partial_x(f\mathcal{H}[f])\eeq
where $\mathcal{H} $ is Hilbert transform defined in Fourier variables  by
 \beqs\widehat{{\mathcal{H}[v]}\
}(\xi)=i\,\text{sgn}(\xi)\ \widehat{v}(\xi),\eeqs for $v\in \mathcal{S}(\R)$. The Hilbert transform has the following representation formula, see e.g. \cite{stein}, 
$$\mathcal{H}[v](x)=\frac{1}{\pi}PV\int_\R\frac{v(y)}{y-x}dy$$
and if $u\in C^{1,\alpha}(\R)$ and $u_x\in L^p(\R)$ with $1<p<+\infty$, then 
\beq\label{halflaplhilbert}\I[u]=\mathcal{H}[u_x].\eeq
Identity \eqref{halflaplhilbert} can be easily proven by performing an integration by parts or using Fourier variables.
The conservation of mass  satisfied by  the positive integrable solutions of \eqref{transporteq} reflects the fact that if  $f=u_x$ is the density of dislocations, no dislocations are created or annihilated. 

Equation \eqref{transporteq} was also proposed by Constantin, et al. \cite{constantin} as a simplified  one dimensional version of the 2-D quasi-geostrophic  model.
In \cite{cc}, Castro and C\`{o}rdoba  show that given an initial datum $f_0(x)$ which belongs to $C^\alpha(\R)\cap L^2(\R)$ and  is strictly positive, then there exists a smooth (analytic in $x$) global (for all times) solution of \eqref{transporteq} that at time 0 is equal to $f_0(x)$.
If $f_0(x)$ is non-negative and 0 at some point, the authors show  the existence of  a  local   solution that  blows up  in finite time.
On the other hand, Carrillo, Ferreira and Precioso \cite{carrillo} apply transportation methods and show that the solution can be obtained as a gradient flow in the space  of probability measures with bounded second moment.
Finally, we mention that  equation  \eqref{transporteq} is a particular case of the fractional porous medium equation
\beqs \partial_t u=\nabla\cdot(u^{m-1}\nabla (-\Delta)^{-s}u)\eeqs
recently studied in \cite{caffavaz,caffavaz2,caffavazsoria}. Indeed, it corresponds to the case $s=1/2$ and $m=2$ in dimension 1. 
Self-similar solutions and decay estimates  for  equation \eqref{limiteq}   have been 
studied in \cite{bkm}.

From a mathematical point of view,  as $\delta$ and $\ep$ go to 0 simultaneously,  \eqref{uepeq} is both a homogenization problem (even though there is no a cell problem and the limit equation is explicit) and a non-local Allen-Cahn type equation. As for an  Allen-Cahn type problem,  the solution gets closer and closer  to the stable minima of the potential, that for the rescaled potential $W(\cdot/\ep)$, by  \eqref{Wass}, are the points of the set $\ep\Z$, and converges 
 to a continuous function, the solution of \eqref{ubareq},  when $\ep$ goes to 0.  
 To prove Theorem \ref{mainthm}, the idea is to approximate the dislocation particles with points $x_i(t)$ where the limit function $u$ attains the value $\ep i$ at time $t$. We then provide  a discrete  approximation formula  for the operator $\I$ with uniform  error estimates over $\R$, which holds true for any $C^{1,1}$ function,  and we use it to show that 
 $$\dot{x}_i=-\frac{\partial u_t(t,x_i(t))}{\partial_x u(t,x_i(t))}\simeq -c_0\I[u(t,\cdot)](x_i(t)).$$
 The strategy and the heuristic of the proofs are explained in Section \ref{strategysec}.

\subsection{The 1-D  Peierls-Nabarro}\label{nabarrrosec}
The Peierls-Nabarro model \cite{n,p,nabarro} is a phase field model for dislocation
dynamics incorporating atomic features into continuum framework.
In a phase field approach, the dislocations are represented by
transition of a continuous field.
We briefly review the model in the case of an {\em edge straight} dislocation in a
crystal with simple cubic lattice.  In a Cartesian system of
coordinates $x_1x_2x_3$, we assume that the straight dislocation is located
in the slip plane $x_1x_3$ (where the dislocation can move) and perpendicular to the axis $x_1$.
In the case of an edge dislocation  the Burgers' vector  (i.e. a fixed vector associated to the
dislocation) is perpendicular to the dislocation line, thus in the direction of  the   axis $x_1$. We write this
Burgers' vector as $be_1$ for a real $b$. 
After a section of the three-dimensional crystal with the plane $x_1x_2$, the dislocation line can be identified with a point  on the $x_1$ axis.
The disregistry of the
upper half crystal $\{x_2>0\}$ relative to the lower half
$\{x_2<0\}$ in the direction of the Burgers' vector is
$\phi(x_1)$, where $\phi$ is a phase parameter between $0$ and
$b$. Then the dislocation point can be for instance localized by
the level set $\phi=b/2$. 
In the Peierls-Nabarro model, the total energy is given by
\beq\label{energy}
\mathcal{E}=\mathcal{E}^{el}+\mathcal{E}^{mis}.
\eeq
In \eqref{energy}, 
$\mathcal{E}^{el}$ is the elastic energy  induced by the dislocation. In the isotropic case   and for a straight dislocation line
it takes the form $$\mathcal{E}^{el}=\frac{1}{2}\int_{\R\times\R^+}|\nabla U|^2\,dx_1\,dx_2,$$ where $U:\R\times\R^+\to\R$ represents the displacement which is such that $U(x_1,0)=\phi(x_1)$. 
$\mathcal{E}^{mis}$ is the so called misfit energy due to
the nonlinear atomic interaction across the slip plane, \beq\label{misfitenergy}
\mathcal{E}^{mis}(\phi)=\int_{\R}W(U(x_1,0))\, dx_1=\int_{\R}W(\phi(x_1))\ dx_1,\eeq
where $W(\phi)$ is the interplanar potential. In a  general model, one can consider a potential $W$ satisfying assumptions \eqref{Wass}.
The periodicity of $W$ reflects the periodicity of the crystal,
while the minimum property is consistent with the fact that the
perfect crystal is assumed to minimize the energy.
The equilibrium configuration of the edge dislocation is obtained by
minimizing the total energy with respect to $U$, under the constraint that far from the dislocation core, the function
 $\phi$ tends to $0$ in one half line and to $b$ in the other half line. The corresponding Euler-Lagrange equation can be written in terms of  the phase transition $\phi$ as
 $$\I[\phi]=W'(\phi).$$
 Assume for simplicity $b=1$, if we fix the value of $\phi$ at the origin to be $1/2$, then  for  $x=x_1$   the 1-D phase transition is solution to:
 \begin{equation}\label{phi}
\begin{cases}
\I[\phi]=W'(\phi)&\text{ in } \R\\
\phi'>0&\text{ in}\quad \R\\
\displaystyle\lim_{x\to-\infty}\displaystyle\phi(x)=0,\quad\lim_{x\to+\infty}\displaystyle\phi(x)=1, \quad \phi(0)=\displaystyle\frac{1}{2}.
\end{cases}
\end{equation}
Existence of a unique solution of \eqref{phi} has been proven in \cite{csm}.
In the classical Peierls-Nabarro model the potential is given by 
$
W(u)=\frac{
b^2}{4\pi^2d}\left(1-\cos\left(\frac{2\pi u}{b}\right)\right),$ where $d$ is the lattice spacing perpendicular to the slip plane, 
and the 1-D phase transition,  found by Nabarro \cite{n}, is explicit: 
$\phi(x)=\frac{b}{2}+ \frac{b}{\pi}\arctan\left(\frac{2x}{d}\right)$.

\medskip

In the face cubic structured (FCC) observed in many metals and
alloys, dislocations move at low temperature on the slip plane.
The dynamics for a collection of straight dislocations lines with the same
Burgers' vector and all contained in a single slip plane, moving with self-interactions (no exterior forces)
 is then
described by the evolutive version  of the Peierls-Nabarro model
 (see for instance \cite{MBW} and \cite{Denoual}):
\begin{equation}\label{nabarroevolutintro}
\p_{t} u=\I[u(t,\cdot)]-W'\left(u\right)\quad\text{in}\quad \R^+\times\R.\\
\end{equation}
In this paper we consider equation \eqref{nabarroevolutintro} when the  number of dislocations is of order $1/\ep$ and  neighboring dislocations are at distance at microscopic scale of order  $1/\delta$.  
This can be represented by the following initial condition
\beqs u(0,x)=\sum_{i=1}^{N_\ep}\phi\left(x-\frac{y_i}{\delta}\right),\eeqs
where $\phi$ is the solution of \eqref{phi}, $N_\ep\sim1/\ep$ and 
$$0\le y_{i+1}-y_i\sim 1.$$
We want to identify at {\em large (macroscopic) scale} the evolution model for
the dynamics of a density of dislocations. We consider the
following rescaling
$$u^\ep(t,x)=\ep u\left(\frac{t}{\ep \delta^2},\frac{x}{\ep\delta}\right),$$
then we see that $u^\ep$ is solution of \eqref{uepeq} with initial datum 
\beq\label{u0introphi}u^\ep(0,x)=\sum_{i=1}^{N_\ep}\ep \phi\left(\frac{x-\ep y_i}{\ep\delta}\right).\eeq
Here $\epsilon$ 
describes the ratio between the microscopic scale and the macroscopic scale.
After the rescaling we see that the distance between neighboring dislocations is of order $\ep\sim1/N_\ep$.
Every dislocation point is described by a phase transition $\ep\phi\left(\frac{x-\ep y_i}{\ep\delta}\right)$ whose derivative is of order 
$1/\delta$.  

More in general, we consider an initial datum $u_0$ satisfying \eqref{u^0ass}.
One can actually prove (see Proposition \ref{approxpropfinal}) that any function satisfying \eqref{u^0ass}, normalized such that the infimum is 0, can be approximated by a function of the form  \eqref{u0introphi}. The monotonicity of $u_0$ reflects the fact that the dislocations have all the same orientation so that no annihilations occur. 

\subsection{The  discrete dislocation dynamics ($\delta=0$)}
When $\ep=1$,   \eqref{uepeq} is a non-local Allen-Cahn equation. In \cite{gonzalezmonneau}, Gonz\'{a}lez and Monneau, show that 
the solution converges as $\delta\to 0$ to the stable minima of the potential $W$, that is integers.
More precisely, if the initial datum is well prepared,  the solution converges to a sum of Heaviside functions of the form 
$\sum_{i=1}^N H(x-y_i(t))$, where 
the interface points $y_i(t)$,  $i=1,\ldots,N$ evolve in time   driven by  the following system of  ODE's:
\beq\label{DDD}\begin{cases}\dot{y}_i=\displaystyle \frac{c_0}{\pi}\displaystyle\sum_{j\neq i}\frac{1}{y_i-y_j}&\text{in }(0,+\infty)\\
y_i(0)=y_i^0.
\end{cases}
\eeq
Here the points $y_i^0$,  $i=1,\ldots,N$,  are given in the initial condition and 
\beq\label{c0}c_0=\left(\int_\R (\phi')^2\right)^{-1},\eeq
with $\phi$ the solution of \eqref{phi}.
System \eqref{DDD} corresponds to the classical  discrete dislocation dynamics 
in the particular case of parallel straight edge dislocation lines in the same slip plane with the same Burgers' vector
and describe the dynamics of dislocation particles  at mesoscopic scale.

\subsection{Brief review of the literature}
When $\delta=1$, \eqref{uepeq} is an homogenization problem and the convergence of the solution when $\ep\to0$ have been studied by Monneau and  the first author in \cite{mp} in any dimension. In this case it is proven that $u^\ep$ converges to the solution of an homogenized equation of type
$\partial_tu=\overline{H}(\nabla u, \I[u])$, where the effective Hamiltonian $\overline{H}$ is implicitly defined through a cell problem. In 
\cite{mp2} it is proven that in dimension 1 $\overline{H}(\delta p,\delta L)\simeq c_0\delta^2 |p|L$ as $\delta\to0$.  See also \cite{pv} for  fractional operators of any order $s\in(0,2)$. 
The proofs of \cite{mp,mp2} cannot be adapted here as the errors obtained blow up when $\ep$ and  $\delta$ converge to 0 simultaneously. 
For more results about homogenization of local and nonlocal first order operators with $u^\ep/\ep$ dependence we refer to 
\cite{im, imr, barles}.
Collisions of dislocation particles and/or long time behavior for the solution of \eqref{uepeq} with $\ep=1$ have been studied  in \cite{pv2, pv3,pv4,cddp}.
In \cite{gm} and \cite{gm2}, Garroni
and Muller study a variational model for dislocations that is the
variational formulation of the stationary Peierls-Nabarro
equation in dimension 2, and  they derive a line tension model.

The passage from  discrete  models of type \eqref{DDD}  ($\delta=0$)  to   continuum models has been stu\-died in several papers. In \cite{fim},  Forcadel, Imbert and Monneau prove that the function $\sum_{i=1}^{N_\ep} H(x-y_i(t))$, where $y_i$, $i=1,\ldots,N_\ep$  solve \eqref{DDD}, properly rescaled,
 converges to the continuous viscosity solution of an homogenized equation, which is \eqref{limiteq} when the forcing term is 0. 
In \cite{MeMo}, van Meurs and Morandotti present a  discrete-to-continuum limit passage for a system of dislocation particles  with a regularized potential, which includes annihilation. 
Convergence of evolving interacting particle systems in dimension 2 have been studied in \cite{gmps}. For further related results we refer the reader to 
\cite{glp,mm,mmp,mora,sppg} and references therein. 

 

\subsection{Organization of the paper}
The paper is organized as follows. In Section \ref{strategysec} we present the strategy and the heuristic of the proof of Theorem \ref{mainthm}.
In Section \ref{prelimsec} we recall some general auxiliary results that  will be used in the rest of the paper.
 In Section \ref{discreteIsec} we prove a discrete approximation formula for  the operator $\I$. Section \ref{mainthmsec} is devoted to the proof of our main result, Theorem \ref{mainthm}.
The main comparison result used in the proof of the theorem is shown in Section \ref{comparisonu+u-sec}. Finally the proofs of some auxiliary lemmas are given in Section \ref{lemmatasec}.
\subsection{Notations} 
 We denote by $B_r(x)$ the ball of radius
$r$ centered at $x$. The cylinder $(t-\tau,t+\tau)\times B_r(x)$
is denoted by $Q_{\tau,r}(t,x)$.
$\lfloor x \rfloor$ and $\lceil x\rceil$ denote respectively the
floor and the ceil integer parts of a real number $x$.

For $r>0$, we denote 
\beq\label{i1v}\I^{1,r}[v](x)=\frac{1}{\pi}PV
\displaystyle\int_{|y-x|\leq r}\displaystyle\frac{v(y)-v(x)}{(y-x)^{2}}dy,\eeq
and 
\beq\label{i2v}\I^{2,r}[v](x)=\frac{1}{\pi}
\displaystyle\int_{|y-x|>r}\displaystyle\frac{v(y)-v(x)}{(y-x)^{2}}dy.\eeq
Then we can write
$$\I[v](x)=\I^{1,r}[v](x)+\I^{2,r}[v](x).$$
We denote by $USC_b((0,+\infty)\times\R)$ (resp.,
$LSC_b((0,+\infty)\times\R)$) the set of upper (resp., lower)
semicontinuous functions on $(0,+\infty)\times\R$ which are bounded on
$(0,T)\times\R$ for any $T>0$ and we set
$C_b((0,+\infty)\times\R):=USC_b((0,+\infty)\times\R)\cap
LSC_b((0,+\infty)\times\R)$.
We denote by $C_b^2((0,+\infty)\times\R)$ the subset of functions of $C_b((0,+\infty)\times\R)$ with continuous second derivatives.
Finally,  $C^{1,1}(\R)$ is the set of functions with bounded $C^{1,1}$ norm over~$\R$.

Given a sequence $\{u^\ep\}$ we denote
$${\limsup_{\ep\to0}}^*u^\ep(t,x)=\sup\Big\{\limsup_{\ep\to0} u^\ep(x_\ep)\,|\,x_\ep\to x\Big\},$$
and 
$${\liminf_{\ep\to0}}_*u^\ep(t,x)=\inf\Big\{\liminf_{\ep\to0} u^\ep(x_\ep)\,|\,x_\ep\to x\Big\}.$$

Given a quantity   $E=E(x)$, we write  $E=O(A)$ is there exists a constant $C>0$ such that, for all $x$, 
$$|E|\le C A.$$
We write $E=o_\ep(1)$ if 
$$\lim_{\ep \to0} E=0,$$
uniformly in $x$. 

\section{Strategy and heuristic of the proofs}\label{strategysec}
In this section we  explain the steps that we will follow  to prove Theorem \ref{mainthm}
 and the heuristic of the main proofs. 
 \subsection{Approximation of $\I$}
 The first result is a discrete  approximation formula for the fractional Laplace $\I$ of non-decreasing $C^{1,1}$  functions (Proposition \ref{apprIcor} and Proposition \ref{apprIcorall}, see also Remark \ref{remgamma0}). Let $v\in C^{1,1}(\R)$. Assume for simplicity that $v$ is strictly increasing.  Let $\ep>0$ be a  small parameter. Let us define the  points $x_i$ as follows,
 \beq\label{xiintro}v(x_i)=\ep i,\quad i=M_\ep,\ldots, N_\ep\eeq
where $M_\ep:= \left\lceil\frac{\inf_\R v+\ep}{\ep}\right\rceil$ and $N_\ep=\left \lfloor\frac{\sup_\R v-\ep}{\ep} \right \rfloor.$
By the monotonicity of $v$ the points $x_i$ are ordered,
$$x_{i}<x_{i+1}\quad\text{for all }i.$$
 Then, we show that 
\beq\label{heuristicI1xi0}\I[v](x_{i_0})\simeq\frac{1}{\pi}\sum_{i\not=i_0} \dfrac{\ep}{x_i-x_{i_0}},\eeq
where  the error goes to 0 when $\ep\to0$. To show \eqref{heuristicI1xi0}, we consider a small radius $r=r_\ep$ such that $r\to0$ as $\ep\to0$ and we split 
\beqs\sum_{i\not=i_0} \dfrac{\ep}{x_i-x_{i_0}}=\sum_{\stackrel{i\not=i_0}{|x_i-x_{i_0}|\leq r} }\dfrac{\ep}{x_i-x_{i_0}}
+\sum_{\stackrel{}{|x_i-x_{i_0}|> r}} \dfrac{\ep}{x_i-x_{i_0}}.\eeqs
Then, we have 
\beqs\begin{split}\frac{1}{\pi}\sum_{\stackrel{}{|x_i-x_{i_0}|> r}} \dfrac{\ep}{x_i-x_{i_0}}&=
\frac{1}{\pi}\sum_{\stackrel{}{|x_i-x_{i_0}|> r}} \dfrac{v(x_{i+1})-v(x_i)}{x_i-x_{i_0}}\\&
\simeq \frac{1}{\pi}\sum_{\stackrel{}{|x_i-x_{i_0}|> r}} \dfrac{v_x(x_i)(x_{i+1}-x_i)}{x_i-x_{i_0}}\\&
\simeq\frac{1}{\pi} \int_{|x-x_{i_0}|>r}\dfrac{v_x(x)}{x-x_{i_0}}dx\\&
= \frac{1}{\pi}\int_{|x-x_{i_0}|>r}\dfrac{v(x)-v(x_{i_0})}{(x-x_{i_0})^2}dx-\frac{1}{\pi}\frac{v(x_{i_0}+r)+v(x_{i_0}-r)-2v(x_{i_0})}{r}\\
&\simeq \I[v](x_{i_0}),
\end{split}
\eeqs
where we have performed an integration by parts in the fourth equality. 
We can control  the error produced in the approximation by  choosing $r$ not too small ($r$ such that $\ep/r\to0$ as $\ep\to0$).

On the other hand, 
for $i\neq i_0$,
$$\ep(i-i_0)=v(x_{i})-v(x_{i_0})\simeq v_x(x_{i_0})(x_i-x_{i_0})$$
from which (Lemma \ref{lemmaerrorshortdistance}).
\beqs\begin{split}\sum_{\stackrel{i\not= i_0}{|x_i-x_{i_0}|\leq r} }\dfrac{\ep}{x_i-x_{i_0}}&\simeq v_x(x_{i_0})\sum_{\stackrel{i\not= i_0}{|i-i_0|\leq v_x(x_{i_0}) \frac{r}{\ep} }}\frac{1}{(i-i_0)}\\&
\simeq v_x(x_{i_0})\left(\sum_{\stackrel{i\leq i_0-1}{} }\frac{1}{(i-i_0)}+\sum_{\stackrel{i\ge i_0+1}{ }}\frac{1}{(i-i_0)}\right)\\&
=v_x(x_{i_0})\left(-\sum_{\stackrel{k\geq 1}{} }\frac{1}{k}+\sum_{\stackrel{k\ge 1}{} }\frac{1}{k}\right)\\
&=0.
\end{split}
\eeqs
We  can control the error produced by choosing $r$ sufficiently small ($r\le\ep^\frac12$).
Combining the two estimates, we obtain \eqref{heuristicI1xi0}.

We actually show that for any $x$,  
$$ \I[v](x)\simeq \frac{1}{\pi}\sum_{\stackrel{}{|x_i-x|> r}} \dfrac{\ep}{x_i-x}$$
where  the error is uniform over $\R$, that is do not depend on the point $x$, while  the sum 
$$\sum_{\stackrel{i\neq i_0}{|x_i-x|\leq r}} \dfrac{\ep}{x_i-x}$$ may not be zero but depends on   the distance of $x$ from the closest  $x_i$.

All  our estimates hold true for any non-decreasing (non necessarily strictly increasing)  $C^{1,1}$  function. 
\subsection{Approximation of $v$}\label{heurconsub} Let $\phi$ be the transition layer defined by \eqref{phi}. It is known (see Lemma \ref{phiinfinitylem}) that 
if $H(x)$ is the Heaviside function, then $\phi$ exhibits the following behavior at infinity: for $|x|>>1$, 
\beq\label{phiinfqppr}\phi(x)\simeq H(x)-\frac{1}{\alpha\pi x},\eeq
where $\alpha=W''(0)$. 
 Using  estimates  \eqref{phiinfqppr} and \eqref{heuristicI1xi0},  we show (Proposition \ref{approxpropfinal}) that if $v\in C^{1,1}(\R)$ is non-decreasing and $x_i$ are defined by 
 \eqref{xiintro}, then 
 \beq\label{vappintr}v(x)\simeq \sum_{i=M_\ep}^{N_\ep}\ep\phi\left(\frac{x-x_i}{\ep\delta}\right)+\ep M_\ep.\eeq
Notice that $\ep M_\ep\simeq\inf_\R v$. Indeed, assume for simplicity that $x=x_{i_0}$ for some $M_\ep\leq i_0\leq N_\ep$.
Then,  for  $\ep$ and $\delta$ small:  $(x_{i_0}-x_i)/(\delta\ep)>>1$ if  $i\leq i_0-1$,  $(x_{i_0}-x_i)/(\delta\ep)<<-1$ if $i\geq i_0+1$.
Then, by \eqref{phiinfqppr} and  \eqref{heuristicI1xi0},
\beqs\begin{split}  \sum_{i=M_\ep}^{N_\ep}\ep\phi\left(\frac{x_{i_0}-x_i}{\ep\delta}\right)+\ep M_\ep
&=\sum_{\stackrel{i=M_\ep}{} }^{i_0-1}\ep\phi\left(\frac{x_{i_0}-x_i}{\ep\delta}\right)+\ep\phi(0)+\sum_{i=i_0+1}^{N_\ep}\ep\phi\left(\frac{x_{i_0}-x_i}{\ep\delta}\right)+\ep M_\ep\\&
\simeq \sum_{\stackrel{i=M_\ep}{} }^{i_0-1}\ep \left(1+\dfrac{\ep\delta}{\alpha\pi(x_i-x_{i_0})}\right)
+\frac{\ep\delta}{\alpha\pi}\sum_{i=i_0+1}^{N_\ep}\dfrac{\ep}{x_i-x_{i_0}}+\ep M_\ep\\&
=\frac{\ep\delta}{\alpha\pi} \sum_{i\not=i_0} \dfrac{\ep}{x_i-x_{i_0}}+\ep i_0\\&
\simeq \frac{\ep\delta}{\alpha}  \I[v](x_{i_0})+\ep i_0\\&
\simeq \ep i_0\\&
= v(x_{i_0}).
\end{split}
\eeqs
We prove that estimate \eqref{vappintr} holds true for any non-decreasing $C^{1,1}$ function $v$ and that the error is independent of the point $x$.

\subsection{Heuristic of the proof of Theorem \ref{mainthm}}\label{heuristicsec} 
As for an homogenization problem we fix $(t_0,x_0)\in(0,+\infty)\times\R$ and  find an ansatz for $u^\ep$ in a small box $Q_R$ of size $R$ centered 
at the point.  Let $u$ be the limit solution (that here we suppose  to exist and  be smooth). For $R$ small,  all the derivatives of $u$ can be considered constant in $Q_R$:
$$\partial_t u(t,x)\simeq \partial_t u(t_0,x_0),\quad \partial_x u(t,x)\simeq \partial_x u(t_0,x_0)$$
and 
$$\I[u(t,\cdot)](x)\simeq \I[u(t_0,\cdot)](x_0)=: L_0.$$
By the comparison principle $u^\ep$ and thus $u$ is non-decreasing in $x$. 
Assume that 
$$ \partial_x u(t_0,x_0)>0.$$ In particular $u$ is strictly increasing in $x$ in $Q_R$.
For $t$ close to $t_0$, we define the points $x_i(t)$ such that 
\beq\label{xiintrou}u(t,x_i(t))=\ep i.\eeq
Since $u$ is  strictly increasing in $x$ in $Q_R$, if $(t,x_i(t)),\,(t,x_{i+1}(t))\in Q_R$ then  $0<x_{i+1}-x_i\simeq \ep$ (see Lemma \ref{lemdistxi}).
For $i$ such that  $(t,x_i(t))\in Q_R$, by differentiating \eqref{xiintrou} we get
\beqs \partial_t u(t,x_i(t))+\partial_xu(t,x_i(t))\dot{x}_i(t)=0,\eeqs
from which
\beq\label{xidotintro}\dot{x}_i(t)=-\frac{ \partial_t u(t,x_i(t))}{\partial_xu(t,x_i(t))}\simeq -\frac{ \partial_t u(t_0,x_0)}{\partial_xu(t_0,x_0)}. \eeq
Next we consider as   ansatz for  $u^\ep$ the approximation of $u$ given by  \eqref{vappintr} plus a small correction:
\beqs \Phi^\ep(t,x):= \sum_{i=M_\ep}^{N_\ep}\ep\left(\phi\left(\frac{x-x_i(t)}{\ep\delta}\right)+\delta\psi\left(\frac{x-x_i(t)}{\ep\delta}\right)\right)+\ep M_\ep.\eeqs
The function $\psi$ is defined in the forthcoming equation   \eqref{psi} with $L=L_0$.
 For a detailed heuristic motivation of
this correction,
see Section 3.1 of \cite{gonzalezmonneau}.
By \eqref{vappintr}, $\Phi^\ep(t,x)\to u(t,x)$ as $\ep\to 0$.  Fix $(t,x)\in Q_R$ and let $x_{i_0}(t)$ be the closest point among  the $x_i(t)$'s to $x$ and $z_i=(x-x_i(t))/(\ep\delta)$. Plugging into \eqref{uepeq}, we get  (see proof of \eqref{mainlem2} in Section \ref{mainthmsec})
\beqs\begin{split} 0&=\delta\partial_t \Phi^\ep(t,x)-\I [\Phi^\ep(t,\cdot)](x)+\frac{1}{\delta}W'\left(\frac{\Phi_\ep(t,x)}{\ep}\right)
\\&\simeq 
-\phi'(z_{i_0})(\dot{x}_{i_0}(t)+c_0L_0)+
(W''(\phi(z_{i_0})) - W''(0))\left ( \frac{1}{\delta}\sum_{\stackrel{}{i\not=i_0}}^{} \tilde{\phi}(z_i) -\frac{L_0}{\alpha}\right)
\end{split}
\eeqs
where $\tilde\phi(z)=\phi(z)-H(z)$. Suppose for simplicity that 
$x=x_{i_0}(t)$, then 
by \eqref{phiinfqppr} and \eqref{heuristicI1xi0}
\beqs \frac{1}{\delta}\sum_{\stackrel{}{i\not=i_0}}^{} \tilde{\phi}(z_i) -\frac{L_0}{\alpha}\simeq \frac{1}{\alpha\pi}\sum_{i\not=i_0} \dfrac{\ep}{x_i-x_{i_0}}-\frac{L_0}{\alpha}\simeq 0. \eeqs
Since $\phi'>0$, we must have 
$$\dot{x}_{i_0}(t)\simeq-c_0L_0 $$ that is, by \eqref{xidotintro},
$$ \partial_t u(t_0,x_0)\simeq c_0  \partial_x u(t_0,x_0)\I[u(t_0,\cdot)](x_0).$$
Notice that if we define $$y_i(\tau):=\frac{x_i(\ep\tau)}{\ep}$$then the $y_i$'s  solve
\beqs \dot{y}_i(\tau)=\dot{x}_i(\ep\tau)\simeq -c_0L_0 \simeq \frac{c_0}{\pi}\sum_{j\not=i} \dfrac{\ep}{x_i-x_{j}}=\frac{c_0}{\pi}\ \sum_{j\not=i} \dfrac{1}{y_i-y_{j}},\eeqs
which is the discrete dislocations dynamics given in \eqref{DDD}.
\subsection{Viscosity sub and supersolutions}
By using  the comparison principle we show that the functions $u^\ep$ are bounded uniformly in $\ep$ (see Section \ref{mainthmsec}). In particular, 
$u^+:={\limsup_{\ep\rightarrow0}}^*u^\epsilon$ and 
$u^-:={\liminf_{\ep\rightarrow0}}_*u^\epsilon$ are
 everywhere finite. 
To formally prove the convergence result following the idea of Section  \ref{heuristicsec}, we  show that 
$u^+$ and $u^-$ are respectively viscosity sub and supersolution of \eqref{ubareq}. As in the perturbed test function method by Evans \cite{e1} in homogenization problems, we will proceed by contradiction. 
\subsection{Comparison with the solution of \eqref{ubareq}}


We prove that $u^+$ and $u^-$ are respectively viscosity sub and supersolution of \eqref{ubareq}, when testing with functions whose derivative in $x$  is different than 0.
This is not enough to conclude that by the comparison principle $u^+\leq u^-$. 
Thus, we  consider the  approximation $F_\ep(x)$ of the initial datum $u_0$ provided by \eqref{vappintr}. 
Since $\phi'\in L^p(\R)$ for all $p\in[1,\infty]$  and $\phi'>0$ (see Lemma \ref{phiinfinitylem}), for fixed 
$\ep,\,\delta>0$, the derivative of $F_\ep(x)$ belongs to $L^p(\R)$ for all $p\in[1,\infty]$ and is strictly positive.
By the results of  \cite{cc}  about equation   \eqref{transporteq} (see  Theorem \ref{globalesistthmderieq} in Section \ref{prelimsec}), 
we can construct a solution  $w^\ep(t,x)$ of \eqref{ubareq} such that $ w^\ep$  is smooth, $\partial_xw^\ep>0$,  $w^\ep(0,x)\simeq u_0(x)$
and $w^\ep\simeq \us$, with $\us$ the viscosity solution of \eqref{ubareq}.
We then show that 
\beq\label{u-winfintro}\lim_{|x|\to+\infty}u^+(t,x)-w^\ep(t,x)\simeq 0,\eeq moreover, 
\beq\label{u-winfintrobis} u^+(0,x)-w^\ep(0,x)\simeq 0.\eeq
We finally prove that $u^+(t,x)- w^\ep(t,x)\leq o_1(\ep)$.
Indeed, if not, by \eqref{u-winfintro} and  \eqref{u-winfintrobis}, $u^+- w^\ep$ must attain a global  positive maximum at some point 
in $(0,+\infty)\times\R$. Then,  using $w^\ep$ (whose derivative in $x$ is strictly positive) as  test function for $u^+$ we get a contradiction. Passing to the limit as $\ep\to0$, this shows that $u^+\le \us$.
Similarly, one can  prove that $u^-\ge \us$. Since the reverse inequality $u^-\leq u^+$ always holds true, we conclude that $u^-=u^+= \us$. 

As a byproduct of  our proof we show that the viscosity solution $\us$ of \eqref{ubareq} satisfies, for all $t\geq0$,
\beqs\lim_{x\to -\infty}\us(t,x)= \inf_\R u_0\quad\text{and}\quad \lim_{x\to +\infty}\us(t,x)=\sup_\R u_0, \eeqs 
which  is equivalent to say that    the  mass of the non-negative function  $\partial_x \us(t,x)$ is conserved: for all $t\geq0$,
$$ \|\partial_x\us (t,\cdot)\|_{L^1(\R)}=\|\partial_x  u_0\|_{L^1(\R)}.$$

\section{Preliminary results}\label{prelimsec}
In this section we recall some general auxiliary results that will be used in the rest of the paper.

\subsection{Short and long range interaction}
We start by recalling a basic fact about the operator $\I$.
Given $v\in C^{1,1}(\R)$ and $r>0$ we can split $\I[v]$ into the short and long range interaction as follows,
$$\I[v](x)=\I^{1,r}[v](x)+\I^{2,r}[v](x),$$ where $\I^{1,r}[v](x),\,\I^{2,r}[v](x)$ are defined respectively by \eqref{i1v} and \eqref{i2v}.
The short range interaction can be rewritten as 
$$ \I^{1,r}[v](x)=\frac{1}{2\pi}\int_{|y|<r}\dfrac{v(x+y)+v(x-y)-2v(x)}{y^2}dy,$$ 
Therefore, 
$$| \I^{1,r}[v](x)|\leq \frac{r}{\pi}\|v\|_{C^{1,1}(\R)}.$$
The long range interaction can be bounded as follows
$$| \I^{2,r}[v](x)|\leq \frac{4}{r\pi}\|v\|_{\infty}.$$
\subsection{The functions $\phi$ and $\psi$}

 In what follows we denote by  $H(x)$ the Heaviside function. Let $\alpha:=W''(0)>0$. 
\begin{lem}
\label{phiinfinitylem}Assume that  \eqref{Wass} holds, then there exists a unique solution $\phi$ of \eqref{phi}. Furthermore $\phi\in C^{2,\beta}(\R)$ and  there exist constants $K_0,K_1 >0$ such that
\begin{equation}\label{phiinfinity}\left|\phi(x)-H(x)+\frac{1}{\al \pi
x}\right|\leq \frac{K_1}{x^2},\quad\text{for }|x|\geq 1,
\end{equation}and for any $x\in\R$
\begin{equation}\label{phi'infinity}0<\frac{K_0}{1+x^2}\leq
\phi'(x)\leq\frac{K_1}{1+x^2}.\end{equation}
\end{lem}
\begin{proof} The existence of a unique solution of \eqref{phi} and estimate \eqref{phi'infinity} are proven in \cite{csm}. Estimate \eqref{phiinfinitylem} is proven in \cite{gonzalezmonneau}.  
\end{proof}
Let $c_0$ be defined as in \eqref{c0}. 
Let us introduce the function  $\psi$ to be the solution of
\begin{equation}\label{psi}
\begin{cases}\I[\psi]=W''(\phi)\psi+\frac{L}{\alpha}(W''(\phi)-W''(0))+c_0L\phi'&\text{in}\quad \R\\
\lim_{x\rightarrow{+\atop -}\infty}\psi(x)=0.
\end{cases}
\end{equation} 
For later purposes, we recall the following decay estimate
on the solution of~\eqref{psi}:
\begin{lem}
\label{psiinfinitylem}
Assume that  \eqref{Wass} holds, then there exists a unique solution $\psi$ to \eqref{psi}. Furthermore 
$\psi\in C^{1,\beta}(\R)$ 
and  for any $L\in\R$  there exist constants $K_2$
and $K_3$, with $K_3>0$, depending on $L$ such that
\begin{equation}\label{psiinfinity}\left|\psi(x)-\frac{K_2}{
x}\right|\leq\frac{K_3}{x^2},\quad\text{for }|x|\geq 1,
\end{equation}and for any $x\in\R$
\begin{equation}\label{psi'infinity}-\frac{K_3}{1+x^2}\leq
\psi'(x)\leq \frac{K_3}{1+x^2}.
\end{equation}
\end{lem}
\begin{proof}
The existence of a unique solution of  \eqref{psi} is proven in  \cite{gonzalezmonneau}. Estimates \eqref{psiinfinity}  and \eqref{psi'infinity} are shown in \cite{mp2}.
\end{proof}
The results of Lemmas \ref{phiinfinitylem} and  \ref{psiinfinitylem} have been generalized  in \cite{cs, dpv,dfv,psv,pv} to the case 
when the fractional operator is $-(-\Delta)^s$ for any $s\in(0,1)$.

\subsection{Definition of viscosity solution}\label{viscositysec}
We first recall the definition of viscosity solution for a general
first order non-local equation 
\beq\label{generalpbbdd}
\partial_tu=F(t,x,u,\partial_xu,\I[u])\quad\text{in}\quad (0,+\infty)\times\Om
\eeq
where $\Om$ is an open subset of $\R$ and $F(t,x,u,p,L)$ is continuous and 
non-decreasing in $L$. 
\begin{de}\label{defviscositybdd}A function $u\in USC_b((0,+\infty)\times\R)$ (resp., $u\in LSC_b((0,+\infty)\times\R)$) is a
viscosity subsolution (resp., supersolution) of
\eqref{generalpbbdd}  if for any $(t_0,x_0)\in(0,+\infty)\times\Om$, 
and any test function $\varphi\in
C_b^2((0,+\infty)\times\R)$ such that $u-\varphi$ attains a global maximum
(resp., minimum) at the point $(t_0,x_0)$, 
then 
\beqs\begin{split}&\p_t\varphi(t_0,x_0)-F(t_0,x_0,u(t_0,x_0),\partial_x\varphi(t_0,x_0),\I[\varphi(t_0,\cdot)](x_0))\leq0\\&\text{(resp., }\geq 0).\end{split}\eeqs  
A function $u\in
C_b((0,+\infty)\times\R)$ is a viscosity solution of
\eqref{generalpb} if it is a viscosity sub and supersolution
of \eqref{generalpbbdd}.
\end{de}
\begin{rem}
It is classical that 
the maximum (resp., the minimum) in Definition \ref{defviscositybdd}   can be assumed to be strict  and that
$$\varphi(t_0,x_0)=u(t_0,x_0).$$
This will be used later.
\end{rem}
Next, let us  consider the initial value problem
\begin{equation}\label{generalpb}
\begin{cases}
\partial_tu=F(t,x,u,\partial_xu,\I[u])&\text{in}\quad (0,+\infty)\times\R\\
u(0,x)=u_0(x)& \text{on}\quad \R,
\end{cases}
\end{equation} where  $u_0$  is a continuous function.

\begin{de}\label{defviscosity}A function $u\in USC_b((0,+\infty)\times\R)$ (resp., $u\in LSC_b((0,+\infty)\times\R)$) is a
viscosity subsolution (resp., supersolution) of the initial value problem 
\eqref{generalpb} if $u(0,x)\leq (u_0)(x)$ (resp., $u(0,x)\geq
(u_0)(x)$) and $u$ is viscosity subsolution (resp., supersolution)  of the equation 
$$\partial_tu=F(t,x,u,\partial_xu,\I[u])\quad\text{in}\quad (0,+\infty)\times\R.$$
 A function $u\in
C_b((0,+\infty)\times\R)$ is a viscosity solution of
\eqref{generalpb} if it is a viscosity sub and supersolution
of \eqref{generalpb}.
\end{de}
It is a classical result that smooth solutions are also viscosity solutions.
\begin{prop}
If $u\in C^1((0,+\infty);C^{1,\beta}_{loc}(\Om)\cap L^\infty(\R))$ for some $0<\beta\le 1$,  and $u$ satisfies pointwise 
$$\partial_tu- F(t,x,u,\partial_xu,\I[u])\leq 0 \text{ (resp. $\geq0$)}\quad\text{in}\quad (0,+\infty)\times\Om,$$ then 
$u$ is a viscosity subsolution (resp., supersolution) of \eqref{generalpbbdd}.
\end{prop}

\subsection{Comparison principle and existence results}
In this subsection, we successively give com\-pa\-ri\-son
principles and existence results for \eqref{uepeq} and
\eqref{ubareq}. The following comparison theorem is shown in
\cite{jk} for more general parabolic integro-PDEs.
\begin{prop}[Comparison Principle for \eqref{uepeq}]\label{comparisonuep} Consider
 $u\in USC_b((0,+\infty)\times\R)$ subsolution
and $v\in LSC_b((0,+\infty)\times\R)$ supersolution of \eqref{uepeq},
then $u\leq v$ on $(0,+\infty)\times\R$.
\end{prop}
Following \cite{jk} it can  also be proven the comparison
principle for \eqref{uepeq} in bounded domains. Since we deal with a
non-local equation, we need to compare the sub and the
supersolution everywhere outside the domain.
\begin{prop}[Comparison Principle on bounded domains for
\eqref{uepeq}]\label{comparisonbounded} Let $\Om$ be a bounded
domain of $(0,+\infty)\times\R$ and let $u\in USC_b((0,+\infty)\times\R)$
and $v\in LSC_b((0,+\infty)\times\R)$ be respectively a sub and a
supersolution of $$\delta\p_{t}
u=\I[u(t,\cdot)]-\frac{1}{\delta}W'\left(\frac{u}{\epsilon}\right)\quad\text{in }\Om.
$$If
$u\leq v$ outside $\Om$, then $u\leq v$ in $\Om$.
\end{prop}

\begin{prop}[Existence for \eqref{uepeq}]\label{existuep}For $\ep,\,\delta>0$ there exists
$u^{\ep}\in C_b([0,+\infty)\times\R)$ (unique) viscosity solution of
\eqref{uepeq}. Moreover, $u^\ep$ is non-decreasing in $x$.
\end{prop}
\dim 
We can construct a
solution by Perron's method if we construct sub and supersolutions
of \eqref{uepeq} which are equal to $u_0(x)$ at $t= 0$. Since $u_0\in C^{1,1}(\R)$, the two functions
$u^{\pm}(t,x):=u_0(x){\pm} \frac{C}{\delta^2} t$ are respectively a super and a
subsolution of \eqref{uepeq}, if
$$C\geq \frac{4\delta}{\pi}\|u_0\|_{C^{1,1}(\R)}+\|W'\|_\infty.$$
Moreover $u^+(0,x)=u^-(0,x)=u_0(x)$.
Since $u_0$ is non-decreasing, the comparison principle implies that $u^\ep$ is non-decreasing in $x$. 
\finedim We next recall the
comparison and the existence results for \eqref{ubareq}, see e.g. \cite{imr}, Proposition 3.
\begin{prop}\label{existHeff}
If $u\in USC_b([0,+\infty)\times\R)$ and  $v\in LSC_b([0,+\infty)\times\R)$
are respectively a sub and a supersolution of  
\beq
\label{ubareqbis}\begin{cases}
\partial_t u=c_0|\partial_x u|\,\I [u]&\text{in }(0,+\infty)\times\R\\
u(0,\cdot)=u_0&\text{on }\R,
\end{cases}\eeq
then $u\leq v$ on $(0,+\infty)\times\R$. Moreover, under assumption \eqref{u^0ass},  there exists a
(unique) viscosity solution of \eqref{ubareqbis} which is non-decreasing in $x$ and thus is   viscosity solution of  \eqref{ubareq}.
\end{prop}
\subsection{Existence of global solutions of equation \eqref{transporteq}}
\begin{thm}[\cite{cc}, Theorem 2.1]\label{globalesistthmderieq}
Let $f_0\in L^2(\R)\cap C^{\beta}(\R)$, for some $0<\beta\le1$ and $f_0>0$ in $\R$ (vanishing at infinity).  
Then, there exists a global solution $v$ of equation \eqref{transporteq} in $C^1((0;+\infty);analytic)$ with $f(0,x)=f_0(x)$. Moreover,
$f$ is vanishing at infinity and $H(f(t,\cdot))\in L^\infty(\R)$ for all $t\geq 0$.
If  $f_0\in L^2(\R)\cap C^{1,\beta}(\R)$, the solution is unique.
\end{thm}

\section{A discrete approximation of the operator $\I$}\label{discreteIsec}
Let $v\in  C^{0,1}(\R)$ be  non-decreasing and non-constant.
For $0<\ep<1$, define the points $x_i$ as follows
\beq\label{xi}x_i:=\inf\{x\in\R\,|\, v( x)=\ep i\}\qquad i=M_\ep,\ldots,N_\ep,\eeq
where 
\beq\label{MepNep}M_\ep:= \left\lceil\frac{\inf_\R v+\ep}{\ep}\right\rceil\quad\text{ and }\quad N_\ep:=\left \lfloor\frac{\sup_\R v-\ep}{\ep} \right \rfloor.\eeq
Since $v$ is continuous,
$$v(x_i)=\ep i,$$
and since  $v$ is non-decreasing, 
$$x_{i}<x_{i+1}\quad\text{for all }i=M_\ep,\ldots, N_\ep-1.$$
Notice that if $v$ is strictly increasing then 
$$ x_i=v^{-1}(\ep i).$$
In what follows given  $\xs\in\R$, we denote by  $x_{i_0}$  the closest point among the $ x_i$'s to $\xs$.
\begin{lem}\label{lemdistxi}Let $v\in C^{0,1}(\R)$ be non-decreasing and non-constant with $\|v_x\|_\infty\leq L$, and 
let $x_i$ be defined as in \eqref{xi}.
Then, 
\beq\label{proplippart1}x_{i+1}-x_i\geq \ep L^{-1}\quad\text{ for all }i=M_\ep,\ldots, N_\ep-1.\eeq
Moreover,  there exists $c>0$ independent of $v$ such that  for any $\xs\in \R$
\beq\label{i/k^2sum}\sum_{ i=M_\ep\atop i\neq i_0}^{N_\ep}\frac{\ep^2}{(x_i-\xs)^2}\leq cL^2.\eeq
If in addition $v_x\ge  a>0$ on an interval $I$, then for all $x_{i+1},\,x_i\in I$, we have
\beq\label{proplippart2}x_{i+1}-x_i\leq \ep a^{-1}.\eeq
\end{lem}
\begin{proof}
We have 
$$\ep=v(x_{i+1})-v(x_i)\leq L(x_{i+1}-x_i),$$ from which \eqref{proplippart1} follows.

Next, by \eqref{proplippart1}, if $x_{i_0}$ is the closest point to $\xs$, then
$$|x_i-\xs|\ge \frac{|i-i_0|\ep}{2L}\quad\text{for all }i.$$ 
Therefore, 
\beqs \sum_{ i=M_\ep\atop i\neq i_0}^{N_\ep}\frac{\ep^2}{(x_i-\xs)^2}\leq4L^2\sum_{ i=M_\ep\atop i\neq i_0}^{N_\ep}\frac{1}{(i-i_0)^2}
\le 8L^2 \sum_{i=1}^\infty\frac{1}{i^2}=cL^2,\eeqs
which proves \eqref{i/k^2sum}.

Finally,   if  $v_x\ge a$, then
$$\ep=v(x_{i+1})-v(x_i)\geq a (x_{i+1}-x_i)$$ from which  \eqref{proplippart2} follows.
\end{proof}

\begin{lem}[Short range interaction]\label{approxIshortlem}Let $v\in C^{1,1}(\R)$ be non-decreasing and non-constant and 
$x_i$ defined as in \eqref{xi}. 
Let  $r=r_\ep$ be such that $r\to0 $  and $\ep/r\to0$ as $\ep\to0$. Let $\rho\ge r$ and  and $\xs\in (x_{M_\ep}+\rho,x_{N_\ep}-\rho)$, then
\beq\label{I1rhoapplemeq} \frac{1}{\pi}\sum_{i\neq i_0\atop r\le |x_i-\xs|\leq\rho}\frac{\ep}{x_i-\xs}=\I^{1,\rho}[v](\xs)+\frac{1}{\pi}\frac{v(\xs+\rho)+v(\xs-\rho)-2v(\xs)}{\rho}+
o_\ep(1).
\eeq
\end{lem}
\begin{proof}
Since $v\in C^{1,1}(\R)$ and $r=o_\ep(1)$, there exists $C>0$ such that 
$$|\I^{1,r}[v](\xs)|\leq Cr=o_\ep(1).$$
Therefore, we have 
\beq\label{approxshortlemprima}\I^{1,\rho}[v](\xs)=\frac{1}{\pi}\int_{\xs-\rho}^{\xs-r} \dfrac{v(x) - v(\xs)}{(x-\xs)^2} dx+
\frac{1}{\pi}\int_{\xs+r}^{\xs+\rho} \dfrac{v(x) - v(\xs)}{(x-\xs)^2} dx+o_\ep(1).
\eeq
Let us estimate from above and below the first and second term in the right-hand side of \eqref{approxshortlemprima}.
We split
\beqs \int_{\xs-\rho}^{\xs-r} \dfrac{v(x) - v(\xs)}{(x-\xs)^2} dx 
=\int_{\xs-\rho}^{\xs-r} \dfrac{v(x)}{(x-\xs)^2} dx - \int_{\xs-\rho}^{\xs-r} \dfrac{v(\xs)}{(x-\xs)^2} dx.
\eeqs
Notice that we can integrate the second term as follows,
\beq \label{ndeq}
\int_{\xs-\rho}^{\xs-r} \dfrac{v(\xs)}{(x-\xs)^2} dx = v(\xs) \int_{\xs-\rho}^{\xs-r} \dfrac{1}{(x-\xs)^2} dx =  \dfrac{v(\xs) }{r} - \dfrac{v(\xs) }{\rho} .
\eeq
Next, we denote by $M_\rho$ and $M_r$ respectively the lowest and the biggest integer $i$ such that $x_i\in[\xs-\rho,\xs-r]$, that is
\beqs x_{M_\rho-1}<\xs-\rho\leq x_{M_\rho}\leq  x_{M_r}\leq\xs-r<x_{M_r+1}.\eeqs
Then, 
we split 
\beq
\int_{\xs-\rho}^{\xs-r} \dfrac{v(x)}{(x-\xs)^2} dx = \int_{\xs-\rho}^{x_{M_\rho}}\dfrac{v(x)}{(x-\xs)^2} dx + \sum_{i=M_\rho}^{M_r-1} \int_{x_i}^{x_{i+1}} \dfrac{v(x)}{(x-\xs)^2} dx + \int_{x_{M_r}}^{\xs-r} \dfrac{v(x) }{(x-\xs)^2} dx. 
\eeq
By using  the monotonicity of $v$, we obtain 
\beq\label{approlemmsec}\begin{split}
\int_{\xs-\rho}^{\xs-r} \dfrac{v(x)}{(x-\xs)^2} dx&
\leq \int_{\xs-\rho}^{x_{M_\rho}}\dfrac{ v(x_{M_\rho})}{(x-\xs)^2} dx + \sum_{i=M_\rho}^{M_r-1}  \int_{x_i}^{x_{i+1}} \dfrac{v(x_{i+1})}{(x-\xs)^2} dx + \int_{x_{M_r}}^{\xs-r} \dfrac{v(\xs-r)}{(x-\xs)^2} dx\\
&= -\dfrac{v(x_{M_\rho})}{\rho} - \dfrac{v(x_{M_\rho})}{x_{M_\rho}-\xs} +  \sum_{i=M_\rho}^{M_r-1}  \left(\dfrac{v(x_{i+1})}{x_i - \xs} - \dfrac{v(x_{i+1})}{x_{i+1} - \xs} \right ) \\&+  \dfrac{v(\xs-r)}{x_{M_r} - \xs}+ \dfrac{v(\xs-r)}{r}. 
\end{split}
\eeq
Recalling that $v(x_i)=\ep i$, we compute
\beqas
 \sum_{i=M_\rho}^{M_r-1}  \left(\dfrac{v(x_{i+1})}{x_i - \xs} - \dfrac{v(x_{i+1})}{x_{i+1} - \xs} \right ) 
 &=&
  \sum_{i=M_\rho}^{M_r-1}\left(\dfrac{\epsilon (i+1)}{x_i - \xs} - \dfrac{\epsilon (i+1)}{x_{i+1} - \xs} \right )\\
  &=& \sum_{i=M_\rho}^{M_r-1} \dfrac{\epsilon (i+1)}{x_i - \xs} -\sum_{i=M_\rho +1}^{M_r} \dfrac{\epsilon i}{x_i - \xs} \\
  &=&\sum_{i=M_\rho+1}^{M_r-1} \dfrac{\epsilon }{x_i - \xs}+ \dfrac{\epsilon (M_\rho+1)}{x_{M_\rho}- \xs}- \dfrac{\epsilon M_r }{x_{M_r} - \xs}\\
  &=&\sum_{i=M_\rho}^{M_r} \dfrac{\epsilon }{x_i - \xs}+ \dfrac{\epsilon M_\rho}{x_{M_\rho}- \xs}- \dfrac{\epsilon (M_r +1)}{x_{M_r} - \xs}\\
    &=&\sum_{i=M_\rho}^{M_r} \dfrac{\epsilon }{x_i - \xs}+ \dfrac{v(x_{M_\rho})}{x_{M_\rho}- \xs}- \dfrac{ v(x_{M_r })}{x_{M_r} - \xs}- \dfrac{\epsilon }{x_{M_r} - \xs}\\
   &\le&\sum_{i=M_\rho}^{M_r} \dfrac{\epsilon }{x_i - \xs}+ \dfrac{v(x_{M_\rho})}{x_{M_\rho}- \xs}- \dfrac{ v(x_{M_r })}{x_{M_r} - \xs}+
   \dfrac{\epsilon }{r}.
\eeqas
Plugging into \eqref{approlemmsec}, we obtain 

\beqas
\int_{\xs-\rho}^{\xs-r} \dfrac{v(x)}{(x-\xs)^2} dx
&\leq& \sum_{i=M_\rho}^{M_r} \dfrac{\epsilon}{x_i - \xs} + \dfrac{v(\xs-r)-v(x_{M_r})}{x_{M_r} - \xs} - \dfrac{v(x_{M_\rho})}{\rho} + \dfrac{v(\xs-r)}{r}+ \dfrac{\epsilon }{r}\\
&\leq& \sum_{i=M_\rho}^{M_r} \dfrac{\epsilon}{x_i - \xs} - \dfrac{v(x_{M_\rho})}{\rho} + \dfrac{v(\xs-r)}{r}+\dfrac{\epsilon}{r},\\
\eeqas
where in the last inequality we have used that 
$v(\xs-r)\geq v(x_{M_r})$ and $x_{M_r} < \xs$. Combining with \eqref{ndeq} and using that that $v(x_{M_\rho})\ge v(\xs-\rho)$, we obtain
\beq \label{3e1}\begin{split}
\int_{\xs-\rho}^{\xs-r} \dfrac{v(x) - v(\xs)}{(x-\xs)^2} dx& \leq \sum_{i=M_\rho}^{M_r} \dfrac{\epsilon}{x_i - \xs} + \dfrac{v(\xs-r) - v(\xs)}{r} -\dfrac{v(x_{M_\rho})-v(\xs)}{\rho} + \dfrac{\epsilon}{r}\\&
\leq  \sum_{i=M_\rho}^{M_r} \dfrac{\epsilon}{x_i - \xs} + \dfrac{v(\xs-r) - v(\xs)}{r} -\dfrac{v(\xs-\rho)-v(\xs)}{\rho} + \dfrac{\epsilon}{r}.
\end{split}
\eeq
Next, we will get a similar estimate for the second term in the right-hand side of \eqref{approxshortlemprima}. 
As before, we split
\beq\label{secondtermomabove} \begin{split}\int_{\xs+r}^{\xs+\rho} \dfrac{v(x) - v(\xs)}{(x-\xs)^2} dx&
= \int_{\xs+r}^{\xs+\rho} \dfrac{v(x)}{(x-\xs)^2} dx -  \int_{\xs+r}^{\xs+\rho}\dfrac{v(\xs)}{(x-\xs)^2} dx\\&
=\int_{\xs+r}^{\xs+\rho} \dfrac{v(x)}{(x-\xs)^2} dx- \dfrac{v(\xs)}{r} + \dfrac{v(\xs)}{\rho}.
\end{split}\eeq
Let $N_r$ and $N_\rho$ be respectively  the lowest and the biggest index $i$ such that $x_i\in[\xs+r,\xs+\rho]$, that is
\beqs x_{N_r-1}<\xs+r\leq x_{N_r}\leq x_{N_\rho}\le\xs+\rho< x_{N_\rho+1}.\eeqs
By the monotonicity of $v$, 
\beq\label{xmrhoxnrshortdilem} 
0\le  v(\xs+\rho)-v(x_{N_\rho})\le v(x_{N_\rho+1})-v(x_{N_\rho})=\ep\eeq
and 
\beq\label{xmrhoxnrshortdilembis} 
0\le v(x_{N_r})-v(\xs+r)\le v(x_{N_r})- v(x_{N_r-1})=\ep.
\eeq
By using again the monotonicity of $v$, we get
\beq\label{secondintlemapproxsfj}\begin{split}
\int_{\xs+r}^{\xs+\rho} \dfrac{v(x)}{(x-\xs)^2} dx
&= \int_{\xs+r}^{x_{N_r}}\dfrac{v(x)}{(x-\xs)^2} dx + \sum_{i=N_r}^{{N_\rho}-1} \int_{x_i}^{x_{i+1}} \dfrac{v(x)}{(x-\xs)^2} dx + \int_{x_{N_\rho}}^{\xs+\rho} \dfrac{v(x) }{(x-\xs)^2} dx\\
&\leq \int_{\xs+r}^{x_{N_r}}\dfrac{v(x_{N_r})}{(x-\xs)^2} dx + \sum_{i=N_r}^{{N_\rho}-1} \int_{x_i}^{x_{i+1}} \dfrac{v(x_{i+1})}{(x-\xs)^2} dx + \int_{x_{N_\rho}}^{\xs+\rho} \dfrac{v(\xs+\rho) }{(x-\xs)^2} dx\\
&= \dfrac{v(x_{N_r})}{r} - \dfrac{v(x_{N_r})}{x_{N_r} - \xs}+ \sum_{i=N_r}^{{N_\rho}-1} \left(\dfrac{v(x_{i+1})}{x_i-\xs} - \dfrac{v(x_{i+1})}{x_{i+1}-\xs}\right) \\&+\dfrac{v(\xs+\rho)}{x_{N_\rho}-\xs} - \dfrac{v(\xs+\rho)}{\rho}.
\end{split}
\eeq
As before, we compute
\beqas
\sum_{i=N_r}^{{N_\rho}-1} \left(\dfrac{v(x_{i+1})}{x_i-\xs} - \dfrac{v(x_{i+1})}{x_{i+1}-\xs}\right)
&=&  \sum_{i=N_r}^{{N_\rho}-1} \left(\dfrac{\epsilon (i+1)}{x_i-\xs} - \dfrac{\epsilon (i+1)}{x_{i+1}-\xs}\right)\\
&=& \sum_{i=N_r}^{{N_\rho}-1} \dfrac{\epsilon (i+1)}{x_i-\xs} -\sum_{i=N_r+1}^{{N_\rho}} \dfrac{\epsilon i}{x_{i}-\xs} \\
&=& \sum_{i=N_r+1}^{{N_\rho}-1} \dfrac{\epsilon }{x_i-\xs} +\dfrac{\epsilon (N_r+1)}{x_{N_r}-\xs} - \dfrac{\epsilon N_\rho}{x_{N_\rho}-\xs} \\
&=&\sum_{i=N_r}^{{N_\rho}-1} \dfrac{\epsilon }{x_i-\xs} +\dfrac{\epsilon N_r}{x_{N_r}-\xs} - \dfrac{\epsilon N_\rho}{x_{N_\rho}-\xs} \\
&=&\sum_{i=N_r}^{{N_\rho}-1} \dfrac{\epsilon }{x_i-\xs} +\dfrac{v(x_{N_r})}{x_{N_r}-\xs}-  \dfrac{v(x_{N_\rho})}{x_{N_\rho}-\xs} .
\eeqas
Plugging into \eqref{secondintlemapproxsfj} and using \eqref{xmrhoxnrshortdilem} and \eqref{xmrhoxnrshortdilembis}, we  obtain
\beqs\begin{split}
\int_{\xs+r}^{\xs+\rho} \dfrac{v(x)}{(x-\xs)^2} dx&\leq \sum_{i=N_r}^{{N_\rho}-1} \dfrac{\epsilon }{x_i-\xs}+   \dfrac{v(\xs+\rho)-v(x_{N_\rho})}{x_{N_\rho}-\xs}+ \dfrac{v(x_{N_r})}{r}  - \dfrac{v(\xs+\rho)}{\rho}\\
&\le \sum_{i=N_r}^{{N_\rho}-1} \dfrac{\epsilon }{x_i-\xs}+   \dfrac{\ep}{x_{N_\rho}-\xs}+ \dfrac{v(x_{N_r})}{r}  - \dfrac{v(\xs+\rho)}{\rho}
\\&=\sum_{i=N_r}^{{N_\rho}} \dfrac{\epsilon }{x_i-\xs}+ \dfrac{v(x_{N_r})}{r}  - \dfrac{v(\xs+\rho)}{\rho}\\&
\le\sum_{i=N_r}^{{N_\rho}} \dfrac{\epsilon }{x_i-\xs}+ \dfrac{v(\xs+r)}{r}  - \dfrac{v(\xs+\rho)}{\rho}+\frac{\ep}{r}.
\end{split}\eeqs
Inserting into  \eqref{secondtermomabove}, we get
\beq \label{3e2}
\int_{\xs+r}^{\xs+\rho} \dfrac{v(x)-v(\xs)}{(x-\xs)^2} dx \leq \sum_{i=N_r}^{{N_\rho}} \dfrac{\epsilon}{x_i-\xs} +\dfrac{v(\xs+r) - v(\xs)}{r} - \dfrac{v(\xs+\rho)-v(\xs)}{\rho}+\frac{\ep}{r}.
\eeq
Combining \eqref{3e1} and \eqref{3e2}, we obtain the upper bound
\beq \label{upb}\begin{split}
\int_{r\leq |x-\xs| \leq \rho} \dfrac{v(x)-v(\xs)}{(x-\xs)^2} dx 
&\leq \sum_{r\le |x_i-\xs|\le \rho} \dfrac{\epsilon}{x_i - \xs} +\dfrac{v(\xs+r) +v(\xs-r)- 2v(\xs)}{r}\\& - \dfrac{v(\xs+\rho)+v(\xs-\rho)-2v(\xs)}{\rho}+ \dfrac{2\epsilon}{r}.
\end{split}\eeq
Similarly,  one can get the following  lower bound estimate
\beq \label{lb}\begin{split}
\int_{r\leq |x-\xs| \leq \rho} \dfrac{v(x)-v(\xs)}{(x-\xs)^2} dx &
\geq\sum_{r\le|x_i-\xs|\le \rho} \dfrac{\epsilon}{x_i - \xs} +\dfrac{v(\xs+r) +v(\xs-r)- 2v(\xs)}{r}\\& - \dfrac{v(\xs+\rho)+v(\xs-\rho)-2v(\xs)}{\rho}
-\dfrac{2\epsilon}{r}.
\end{split}\eeq
Since $v\in C^{1,1}(\R)$, there exists a constant $C>0$ such that 
\beqs \left|\dfrac{v(\xs+r) +v(\xs-r)- 2v(\xs)}{r}\right|\leq Cr=o_\ep(1).
\eeqs 
Therefore, combining \eqref{upb} and \eqref{lb}, then dividing both sides by $\pi$ and using that $\ep/r=o_\ep(1)$, we finally obtain 
\beqs 
\frac{1}{\pi}\sum_{r\le|x_i-\xs|\le\rho} \dfrac{\epsilon}{x_i - \xs} =\frac{1}{\pi}\int_{r\le|x-\xs|\leq\rho} \dfrac{v(x) - v(\xs)}{(x-\xs)^2} dx
+ \frac{1}{\pi}\dfrac{v(\xs+\rho)+v(\xs-\rho)-2v(\xs)}{\rho}+o_\ep(1),
\eeqs
which together with \eqref{approxshortlemprima}  gives \eqref{I1rhoapplemeq}.
\end{proof}

\begin{lem}[Long range interaction]\label{approxIlonglem}
Under the assumptions of Lemma \ref{approxIshortlem} and for $r$ as in the lemma, 
 for any $\rho\ge r$ and $\xs\in(x_{M_\ep}+\rho,x_{N_\ep}-\rho)$,
\beq\label{approxIlonglemeq} \frac{1}{\pi}\sum_{|x_i-\xs|>\rho}\frac{\ep}{x_i-\xs}=\I^{2,\rho}[v](\xs)
-\frac{1}{\pi}\frac{v(\xs+\rho)+v(\xs-\rho)-2v(\xs)}{\rho}+ o_\ep(1).
\eeq
\end{lem}
\begin{proof}
We decompose  $\I^{2,\rho}[v](\xs)$ as follows
\beq\label{middletwoterms} \I^{2,\rho}[v](\xs) = \int_{-\infty}^{x_{M_\ep}} \frac{v(x) - v(\xs)}{(x-\xs)^2} dx + \int_{x_{M_\ep}}^{\xs-\rho} \frac{v(x) - v(\xs)}{(x-\xs)^2} dx + \int_{\xs+\rho}^{x_{N_\ep}} \frac{v(x) - v(\xs)}{(x-\xs)^2} dx + \int_{x_{N_\ep}}^{+\infty} \frac{v(x) - v(\xs)}{(x-\xs)^2} dx.
\eeq
By the monotonicity of $v$, we get 
\beq \label{tail1}  \int_{-\infty}^{x_{M_\ep}} \frac{v(x) - v(\xs)}{(x-\xs)^2} dx  \leq  \int_{-\infty}^{x_{M_\ep}} \frac{v(x_{M_\ep})-v(\xs)}{(x-\xs)^2} dx = \frac{v(x_{M_\ep})-v(\xs)}{\xs-x_{M_\ep}},\eeq
and
\beq \label{tail2}  \int_{x_{N_\ep}}^{+\infty} \frac{v(x) - v(\xs)}{(x-\xs)^2} dx  \leq  \int_{x_{N_\ep}}^{+\infty} \frac{\sup_\R v - v(\xs)}{(x-\xs)^2} dx = \dfrac{\sup_\R v - v(\xs)}{x_{N_\ep}-\xs}. \eeq
One can similarly obtain a lower bound as follows
\beq \label{tail3}  \int_{-\infty}^{x_{M_\ep}} \frac{v(x) - v(\xs)}{(x-\xs)^2} dx  \geq  \int_{-\infty}^{x_{M_\ep}} \frac{\inf_\R v-v(\xs)}{(x-\xs)^2} dx = \frac{\inf_\R v-v(\xs)}{\xs-x_{M_\ep}},\eeq 
and 
\beq \label{tail4}  \int_{x_{N_\ep}}^{+\infty} \frac{v(x) - v(\xs)}{(x-\xs)^2} dx  \geq  \int_{x_{N_\ep}}^{+\infty} \frac{v(x_{N_\ep}) - v(\xs)}{(x-\xs)^2} dx = \dfrac{v(x_{N_\ep}) - v(\xs)}{x_{N_\ep}-\xs}. \eeq
To get the estimates for the middle two terms in the right-hand side of \eqref{middletwoterms}, we will proceed  as in the proof of Lemma \ref{approxIshortlem}. By  respectively replacing $\xs-\rho,\, \xs-r$   with $x_{M_\ep}$ and $\xs-\rho$ in \eqref{3e1} and $\xs+r$, $\xs+\rho$  with 
$\xs+\rho$ and $x_{N_\ep}$   in \eqref{3e2} we obtain 
\beq \label{rhs1}\begin{split}
\int_{x_{M_\ep}}^{\xs-\rho} \dfrac{v(x)-v(\xs)}{(x-\xs)^2} dx + \int_{\xs+\rho}^{x_{N_\ep}}  \dfrac{v(x)-v(\xs)}{(x-\xs)^2} dx 
\le \sum_{|x_i-\xs|\ge\rho} \dfrac{\epsilon}{x_i-\xs} +\dfrac{2\epsilon}{\rho}\\
+\dfrac{v(\xs+\rho)+v(\xs-\rho) -2 v(\xs)}{\rho} -\dfrac{v(x_{M_\ep})-v(\xs)}{\xs-x_{M_\ep}}  - \dfrac{v(x_{N_\ep})-v(\xs)}{x_{N_\ep}-\xs}.
\end{split}\eeq
Similarly,
\beq \label{lhs1}\begin{split}
\int_{x_{M_\ep}}^{\xs-\rho} \dfrac{v(x)-v(\xs)}{(x-\xs)^2} dx + \int_{\xs+\rho}^{x_{N_\ep}}  \dfrac{v(x)-v(\xs)}{(x-\xs)^2} dx
\geq \sum_{|x_i-\xs|\geq\rho} \dfrac{\epsilon}{x_i - \xs}  -\dfrac{2\epsilon}{\rho}\\
+\dfrac{v(\xs+\rho) +v(\xs-\rho)-2 v(\xs)}{\rho} - \dfrac{v(x_{M_\ep})-v(\xs)}{\xs-x_{M_\ep}}  - \dfrac{v(x_{N_\ep})-v(\xs)}{x_{N_\ep}-\xs}.
\end{split}\eeq
Combining \eqref{tail1},   \eqref{tail2} and  \eqref{rhs1}, we get 
\beq\begin{split}\label{final1apprlemmlomg}
\int_{-\infty}^{\xs-\rho} \dfrac{v(x)-v(\xs)}{(x-\xs)^2} dx+\int_{\xs+\rho}^{+\infty} \dfrac{v(x)-v(\xs)}{(x-\xs)^2} dx\leq
 \sum_{|x_i-\xs|\geq\rho} \dfrac{\epsilon}{x_i-\xs} \\
 +\dfrac{v(\xs+\rho) +v(\xs-\rho)-2 v(\xs)}{\rho} + \dfrac{\sup_\R v-\ep N_\ep}{x_{N_\ep}-\xs}+
 \dfrac{2\epsilon}{\rho}.
 \end{split}\eeq
Combining \eqref{tail3}, \eqref{tail4} and \eqref{lhs1}, we get 
\beq\begin{split}\label{final2apprlemmlomg}
\int_{-\infty}^{\xs-\rho} \dfrac{v(x)-v(\xs)}{(x-\xs)^2} dx+\int_{\xs+\rho}^{+\infty} \dfrac{v(x)-v(\xs)}{(x-\xs)^2} dx\geq
 \sum_{|x_i-\xs|\geq\rho} \dfrac{\epsilon}{x_i-\xs} \\
 +\dfrac{v(\xs+\rho) +v(\xs-\rho)-2 v(\xs)}{\rho} - \dfrac{\ep M_\ep-\inf_\R v}{\xs-x_{M_\ep}}-
 \dfrac{2\epsilon}{\rho}.
 \end{split}\eeq
Recalling the definition \eqref{MepNep} of $N_ \ep$ and $M_\ep$, we see that $0\le\sup_\R v-\ep N_\ep\le 2 \ep$ and 
$0\le \ep M_\ep-\inf_\R v\le2\ep$. Since in addition $x_{N_\ep}-\xs> \rho$, $\xs-x_{M_\ep}>\rho$, $\rho\ge r$ and $\ep/r=o_\ep(1)$, 
from \eqref{final1apprlemmlomg} and \eqref{final2apprlemmlomg} we finally get \eqref{approxIlonglemeq}.
\end{proof}

The  following proposition is an immediate consequence of Lemma \ref{approxIshortlem} and Lemma \ref{approxIlonglem}.

\begin{prop}\label{apprIcor}
Let $v\in C^{1,1}(\R)$ be non-decreasing and non-constant and 
$x_i$ defined as in \eqref{xi}.
Let  $r=r_\ep$ be such that $r\to0 $  and $\ep/r\to0$ as $\ep\to0$. Then, for any  $ \xs\in(x_{M_\ep}+r,x_{N_\ep}-r)$,
\beqs\frac{1}{\pi}\sum_{|x_i-\xs|\ge r}\frac{\ep}{x_i-\xs}=\I[v](\xs)+ o_\ep(1).
\eeqs
\end{prop}
\begin{rem}
Notice that in Lemma \ref{approxIshortlem}, Lemma \ref{approxIlonglem} and Proposition \ref{apprIcor}, the error $o_\ep(1)$ satisfies
\beq\label{o1bound}o_\ep(1)=O(r)+O\left(\frac{\ep}{r}\right).\eeq
\end{rem}

\begin{lem}\label{lemmaerrorshortdistance}
Under the assumptions  of Lemma  \ref{approxIshortlem}, let 
$\xs=x_{i_0}+\ep\gamma$.Then, there exists $r=r_\ep $   satisfying  $\ep^\frac{5}{8}\le r\le c\ep^\frac{1}{2}$, with $c$ depending on the $C^{1,1}$ norm of $v$, such that 
\beq\label{shortdostaneqlem}\frac{1}{\pi}\sum_{i\neq i_0\atop |x_i-\xs|< r}\frac{\ep}{x_i-\xs}= O(\ep^\frac{1}{8})+O(\gamma).\eeq
\end{lem}
\begin{proof}
 In what follows we denote by $c$ and $C$ different  constants independent of $\ep$ and $\xs$.
 Let $K>0$ be such that $\|v_{xx}\|_{L^\infty(\R)}\leq K$.
 We divide the proof into three cases.
\medskip

{\em Case 1: $v_x({x_{i_0}})\le 12 K^\frac12\ep^\frac12$.} 

By making a Taylor expansion, we get 
\beqs\begin{split}\ep=v({x_{i_0+1}})-v({x_{i_0}})&\leq v_x({x_{i_0}})(x_{i_0+1}-x_{i_0})+ \frac{K}{2}(x_{i_0+1}-x_{i_0})^2
\\&\leq \frac{v_x({x_{i_0}})^2}{2 (12)^2K}+\left(\frac{12^2K}{2}+\frac{K}{2}\right) (x_{i_0+1}-x_{i_0})^2\\&\leq \frac\ep2 +\frac{12^2+1}{2}K(x_{i_0+1}-x_{i_0})^2,
\end{split}
\eeqs
 from which 
$$x_{i_0+1}-x_{i_0}\geq c\ep^\frac{1}{2}.$$
Similarly, one can prove that 
$$x_{i_0}-x_{i_0-1}\geq c\ep^\frac{1}{2}.$$
Since $x_{i_0}$ is the closest point to $\xs$, we must have that $\xs-x_{i_0-1}\geq c\ep^\frac{1}{2}/2$ and 
$x_{i_0+1}-\xs\geq c\ep^\frac{1}{2}/2$.
Therefore, if we choose $r=r_\ep=c\ep^\frac{1}{2}/4$, there is  no index  $i\neq i_0$ for which 
$|\xs-x_i|\leq r$ and thus  \eqref{shortdostaneqlem} is trivially true. 

\medskip

Next, we show  that 
\beq\label{centeredserier}\frac{1}{\pi}\sum_{i\neq i_0\atop |x_i-\xs|< r}\frac{\ep}{x_i-x_{i_0}}= O(\ep^\frac{1}{8}).
\eeq
We consider two more cases.

\medskip

{\em Case 2:  $12 K^\frac12\ep^\frac12\leq v_x({x_{i_0}})\le \ep^{\frac12-\tau}$, for some $\tau\in (0,1/4)$.}

If $|\xs-x_{i_0}|\geq \ep^\frac{1}{2}/(4K^\frac{1}{2})$, then we choose 
$r=\ep^\frac{1}{2}/(8K^\frac{1}{2})$ and as  in Case 1, there is  no index  $i\neq i_0$ for which 
$|\xs-x_i|\leq r$. Thus     \eqref{shortdostaneqlem} holds  true. 

Now, assume $|\xs-x_{i_0}|\leq \ep^\frac{1}{2}/(4K^\frac{1}{2})$ and define
\beq\label{randep}r:=\frac{\ep^\frac{1}{2}}{2K^\frac{1}{2}}\geq 2 |\xs-x_{i_0}|.\eeq

Let $M_r$ and $N_r$ be respectively the smallest and the larger index $i$ such that $x_i\in (\xs-r,\xs+r)$, that is
\beq\label{xnr-xs}\begin{split}
&x_{M_r-1}\le \xs-r< x_{M_r}\\&
x_{N_r}< \xs+r\le x_{N_r+1}.
\end{split}\eeq
By   the monotonicity of $v$ and  \eqref{xnr-xs},
\beq\label{allthevxdfejtiypf}\begin{split}
-\ep&=v(x_{i_0})-v(x_{i_0+1})\leq  v(x_{i_0})-v(\xs)\le v(x_{i_0})-v(x_{i_0-1})=\ep,\\
-\ep& =  v(x_{N_r})- v(x_{N_r+1})\leq v(x_{N_r})-v(\xs+r)\leq 0\\
0&\le v(x_{M_r})- v(\xs-r)\leq v(x_{M_r})-v(x_{M_r-1}) =\ep.\\
\end{split}
\eeq
By making a Taylor expansion, we get, for $i=M_r,\ldots,N_r$
\beqs \ep(i-i_0)=v(x_i)-v({x_{i_0}})=v_x({x_{i_0}})(x_i-x_{i_0})+O(r^2),\eeqs
where $|O(r^2)|\leq K(2r)^2/2=\ep/2$,
from which 
\beq\label{distancexixi0} x_i-x_{i_0}=\frac{\ep(i-i_0)+O(r^2)}{v_x({x_{i_0}})}.\eeq
Therefore, we can write
\beq\label{sumx-xi0smallprima}\begin{split}\sum_{i\neq i_0\atop |x_i-\xs|< r}\frac{\ep}{x_i-x_{i_0}}&=\sum_{i=M_r\atop i\neq i_0}^{N_r}\frac{v_x(x_{i_0})\ep}{\ep(i-i_0)+O(r^2)}\\&=
\sum_{i=M_r\atop }^{i_0-1}\frac{v_x(x_{i_0})\ep}{\ep(i-i_0)+O(r^2)}+\sum_{i=i_0+1}^{N_r}\frac{v_x(x_{i_0})\ep}{\ep(i-i_0)+O(r^2)}.
\end{split}
\eeq
Now, suppose without loss of generality that $N_r-i_0\leq i_0-M_r$.
Then, 
\beq\label{pedalssonno}\begin{split}
&\sum_{i=M_r\atop }^{i_0-1}\frac{\ep}{\ep(i-i_0)+O(r^2)}+\sum_{i=i_0+1}^{N_r}\frac{\ep}{\ep(i-i_0)+O(r^2)}\\&=
\sum_{k=1}^{i_0-M_r}\frac{\ep}{-\ep k+O(r^2)}+\sum_{k=1}^{N_r-i_0}\frac{\ep}{\ep k+O(r^2)}\\&
=\sum_{k=1}^{N_r-i_0}\ep\left(\frac{1}{-\ep k+O(r^2)}+\frac{1}{\ep k+O(r^2)}\right)+\sum_{k=N_r-i_0+1}^{i_0-M_r}\frac{\ep}{-\ep k+O(r^2)}.
\end{split}
\eeq
We can bound the first term of the right hand-side of the last equality  as follows
\beq\label{serveperclaim3}\begin{split}
\left|\sum_{k=1}^{N_r-i_0}\ep\left(\frac{1}{-\ep k+O(r^2)}+\frac{1}{\ep k+O(r^2)}\right)\right|
&=\frac{2|O(r^2)|}{\ep}\left|\sum_{k=1}^{N_r-i_0}\frac{1}{(-k+\frac{O(r^2)}{\ep})(k+\frac{O(r^2)}{\ep})}\right|\\&
\leq\sum_{k=1}^\infty \frac{1}{k^2-\frac14}\\&
=C,
\end{split}
\eeq
where we used that $|O(r^2)|/\ep\leq 1/2$. 
Therefore, 
\beq\label{pedalssonnobis}\begin{split}
v_x(x_{i_0})\left|\sum_{k=1}^{N_r-i_0}\ep\left(\frac{1}{-\ep k+O(r^2)}+\frac{1}{\ep k+O(r^2)}\right)\right|\leq Cv_x(x_{i_0})\leq C\ep^{\frac12-\tau}.
\end{split}
\eeq
Next, by using that $\sum_{k=n}^m1/k\leq (m-n+1)/n$, we get
\beq\label{secondpieceestilemii0}\begin{split}
\left|\sum_{k=N_r-i_0+1}^{i_0-M_r}\frac{\ep}{-\ep k+O(r^2)}\right|&\leq \sum_{k=N_r-i_0+1}^{i_0-M_r}\frac{\ep}{\ep k-|O(r^2)|}\\&\leq
\frac{-(\ep N_r+\ep M_r-2\ep i_0)}{\ep (N_r+1)-\ep i_0-|O(r^2)|}\\&
=\frac{-(v(x_{N_r})+v(x_{M_r})-2v(x_{i_0}))}{v(x_{N_r})-v(x_{i_0})+\ep-|O(r^2)|}\\&
\le \frac{-(v(x_{N_r})+v(x_{M_r})-2v(x_{i_0}))}{v(x_{N_r})-v(x_{i_0})}.
\end{split}
\eeq
By \eqref{allthevxdfejtiypf} and the regularity of $v$, 
\beq\label{topestinrmri0}
0\le-( v(x_{N_r})+v(x_{M_r})-2v(x_{i_0}))\leq -(v(\xs+r)+ v(\xs-r)-2v(\xs))+3\ep\leq Kr^2+3\ep\le C\ep.
\eeq 
Now, by using that  $v_x(x_{i_0})\ge 12K^\frac12\ep^\frac12$ and that $|\xs-x_{i_0}|\leq r/2$,  and by  \eqref{allthevxdfejtiypf}, we get  
\beq\label{bottomestinrmri0}\begin{split}
v(x_{N_r})-v(x_{i_0})&\geq v(\xs+r)-v(x_{i_0})-\ep\\&\ge v_x(x_{i_0})(r-|\xs-x_{i_0}|)-\frac{K}{2}(2r)^2-\ep\\&
\ge  v_x(x_{i_0})\frac r2-\frac32\ep \\&
=v_x(x_{i_0})\frac r2-12K^\frac12\ep^\frac12\frac{r}{4}\\&
\ge v_x(x_{i_0})\frac{r}{4}\\&
= v_x(x_{i_0})\frac{\ep^\frac12}{8K^\frac12}.
\end{split}
\eeq
From \eqref{secondpieceestilemii0}, \eqref{topestinrmri0} and \eqref{bottomestinrmri0}, we infer that
\beq\label{secondpiecefinallemii0}
\left|\sum_{i=i_0-M_r+1}^{N_r-i_0}\frac{\ep v_x(x_{i_0})}{-\ep k+O(r^2)}\right|\leq \frac{v_x(x_{i_0})C\ep 8K^\frac12}{v_x(x_{i_0})\ep^\frac12}\leq
C\ep^\frac12.
\eeq
Finally, \eqref{sumx-xi0smallprima}, \eqref{pedalssonno}, \eqref{pedalssonnobis} and \eqref{secondpiecefinallemii0} imply
\beqs \left|\frac{1}{\pi}\sum_{i\neq i_0\atop |x_i-\xs|< r}\frac{\ep}{x_i-x_{i_0}}\right|\leq C\ep^{\frac12-\tau}\leq C\ep^\frac{1}{4},
\eeqs
which gives \eqref{centeredserier}.

\medskip

{\em Case 3:  $ v_x({x_{i_0}})\ge \ep^{\frac12-\tau}$, for some $\tau\in(0,1/4)$.}

As in Case 2,   we can assume that $|\xs-x_{i_0}|\le \ep^{\frac{1+\tau}{2}}$. 
Then, we define
\beq\label{rcase3}r:=2\ep^{\frac{1+\tau}{2}}\ge 2|\xs-x_{i_0}|.\eeq
Notice that $r\ge \ep^\frac{5}{8}$.
Assume,    without loss of generality,  that $N_r-i_0\le i_0-M_r$. Then   as before, we  write 
\beq\label{sumx-xi0smallprimacase3}\begin{split}\sum_{i\neq i_0\atop |x_i-\xs|< r}\frac{\ep}{x_i-x_{i_0}}&=
\sum_{k=1}^{N_r-i_0}\ep v_x(x_{i_0})\left(\frac{1}{-\ep k+O(r^2)}+\frac{1}{\ep k+O(r^2)}\right)+ \sum_{k=N_r-i_0+1}^{i_0-M_r}\frac{\ep v_x(x_{i_0})}{-\ep k+O(r^2)}.
\end{split}
\eeq
By \eqref{serveperclaim3} and the definition \eqref{rcase3}  of $r$, 
\beq\label{ccase3step1}
\left|\sum_{k=1}^{N_r-i_0}\ep v_x(x_{i_0})\left(\frac{1}{-\ep k+O(r^2)}+\frac{1}{\ep k+O(r^2)}\right)\right|\leq 
 Cv_x(x_{i_0})\frac{|O(r^2)|}{\ep}\leq C\ep^\tau.
\eeq
By \eqref{secondpieceestilemii0}, \eqref{topestinrmri0}
 and \eqref{bottomestinrmri0},  and by using that $ v_x({x_{i_0}})\ge C\ep^{\frac12-\tau}$ and \eqref{rcase3}, we get
  \beq\label{secondpieceestilemii0case3}\begin{split}
\left|\sum_{k=N_r-i_0+1}^{i_0-M_r}\frac{\ep}{-\ep k+O(r^2)}\right|&\leq\frac{C\ep}{ v_x(x_{i_0})\frac r2-\frac32\ep }\leq C\ep^\frac{\tau}{2},
\end{split}
\eeq
for $\ep$ small enough (independently of $\xs$).
Estimates  \eqref{sumx-xi0smallprimacase3},\eqref{ccase3step1} and  \eqref{secondpieceestilemii0case3} imply
\beqs\left|\sum_{i\neq i_0\atop |x_i-\xs|< r}\frac{\ep}{x_i-x_{i_0}}\right|\leq C\ep^\frac{\tau}{2}\leq C\ep^\frac{1}{8},
\eeqs
which gives \eqref{centeredserier}
\medskip

Finally, to prove \eqref{shortdostaneqlem}, we estimate
\beq\label{sumseriesdiffexsxi0}\begin{split}\left|\sum_{i\neq i_0\atop |x_i-\xs|< r}\frac{\ep}{x_i-\xs}
-\sum_{i\neq i_0\atop |x_i-x_{i_0}|<r}\frac{\ep}{x_i-x_{i_0}}\right|&=\left|
\sum_{i\neq i_0\atop |x_i-\xs|< r}\frac{\ep^2\gamma}{(x_i-\xs)(x_i-x_{i_0})}\right|.
\end{split}\eeq
Assume, without loss of generality that $\xs=x_{i_0}+\ep\gamma$, with $\gamma\ge0$, that is, $\xs\in[x_{i_0},x_{i_0+1})$.
Then, 
\beqs |x_i-\xs|\ge\begin{cases}|x_i-x_{i_0}|&\text{if }i\leq i_0-1\\
\frac{x_{i_0+1}-x_{i_0}}{2}&\text{if }i=i_0+1 \\
x_i-x_{i_0+1}&\text{if }i\geq i_0+2.
\end{cases}
\eeqs
Moreover, by \eqref{proplippart1}, $x_{i_0+1}-x_{i_0}\ge \ep L^{-1}$. Therefore, 
\beq\label{lastlemmagammaxixi0}\left|
\sum_{i\neq i_0\atop |x_i-\xs|< r}\frac{\ep^2\gamma}{(x_i-\xs)(x_i-x_{i_0})}\right|\le \sum_{i\leq i_0-1}\frac{\ep^2\gamma}{(x_i-x_{i_0})^2}
+2L^2\gamma+ \sum_{i\geq i_0+2}\frac{\ep^2\gamma}{(x_i-x_{i_0+1})^2}\le C\gamma,
\eeq
where in the last inequality we used \eqref{i/k^2sum}.
By \eqref{sumseriesdiffexsxi0} and \eqref{lastlemmagammaxixi0} we get 
\beqs
\left|\sum_{i\neq i_0\atop |x_i-\xs|<r}\frac{\ep}{x_i-\xs}
-\sum_{i\neq i_0\atop |x_i-x_{i_0}|<r}\frac{\ep}{x_i-x_{i_0}}\right|\le C\gamma,
\eeqs
which together with \eqref{centeredserier} gives  \eqref{shortdostaneqlem}. 
\end{proof}

The following proposition is an immediate consequence of  Lemma \ref{approxIshortlem}, Proposition \ref{apprIcor} and Lemma \ref{lemmaerrorshortdistance}.
\begin{prop}\label{apprIcorall}
Let $v\in  C^{1,1}(\R)$ be non-decreasing and non-constant and 
$x_i$ defined as in \eqref{xi}. Then, there  exists $c>0$ depending on the $C^{1,1}$ norm of $v$ such that if $\rho\ge c\ep^\frac12$, 
and $\xs\in(x_{M_\ep}+\rho,x_{N_\ep}-\rho)$, 
$\xs=x_{i_0}+\ep\gamma$, then
\beq\label{apprIcorallest1}\frac{1}{\pi}\sum_{i\neq i_0\atop |x_i-\xs|\le \rho}\frac{\ep}{x_i-\xs}=\I^{1,\rho}[v](\xs)+O(\gamma)+\frac{1}{\pi}\frac{v(\xs+\rho)+v(\xs-\rho)-2v(\xs)}{\rho}+o_\ep(1),
\eeq
and
\beq\label{apprIcorallest2}\frac{1}{\pi}\sum_{i\neq i_0}\frac{\ep}{x_i-\xs}=\I[v](\xs)+ o_\ep(1)+O(\gamma).
\eeq
\end{prop}
\begin{proof}
Fix $\xs$ and let $r$ and $c$ be given by Lemma  \ref{lemmaerrorshortdistance}. Then $\ep^\frac{5}{8}\le r\le c\ep^\frac{1}{2}\le\rho$.
By Lemma \ref{approxIshortlem} and recalling \eqref{o1bound},
\beqs \frac{1}{\pi}\sum_{i\neq i_0\atop r\le |x_i-\xs|\le \rho}\frac{\ep}{x_i-\xs} =\I^{1,\rho}[v](\xs)+\frac{1}{\pi}\frac{v(\xs+\rho)+v(\xs-\rho)-2v(\xs)}{\rho}+O\left(\ep^\frac{3}{8}\right).
\eeqs
Combining this estimate with \eqref{shortdostaneqlem} yields \eqref{apprIcorallest1}. 

Similarly, by Proposition \ref{apprIcor} and Lemma \ref{lemmaerrorshortdistance}, we get \eqref{apprIcorallest2}.
\end{proof}
\begin{rem}\label{vcopntstanti1rem}
If $\ep|\gamma|=|\xs-x_{i_0}|>c\ep^\frac{1}{2}\ge r$, then $|\xs-x_{i}|>r$ for all $i$ and 
\beqs\label{vcopntstanti1}\begin{split} \frac{1}{\pi}\sum_{i\neq i_0\atop |x_i-\xs|\le \rho}\frac{\ep}{x_i-\xs}&=
\frac{1}{\pi}\sum_{r< |x_i-\xs|\le \rho}\frac{\ep}{x_i-\xs}\\&=
\I^{1,\rho}[v](\xs)+ \frac{1}{\pi}\frac{v(\xs+\rho)+v(\xs-\rho)-2v(\xs)}{\rho}+o_\ep(1).
\end{split} \eeqs
\end{rem}
\begin{rem}\label{remgamma0}
If $\xs=x_{i_0}$, then $\gamma=0$ and 
\beq\label{i1onparticles}\frac{1}{\pi}\sum_{i\neq i_0}\frac{\ep}{x_i-x_{i_0}}=\I[v](x_{i_0})+ o_\ep(1).\eeq
\end{rem}

\begin{lem}\label{vapproxphisteps}
Let $v\in  C^{1,1}(\R)$ be non-decreasing and non-constant and 
$x_i$ be defined as in \eqref{xi}. Let $\phi$ be defined by \eqref{phi}. Let 
$M_\ep\le M<N\leq N_\ep$ and $R\ge c\ep^\frac{1}{2}$, with $c>0$  given by Proposition \ref{apprIcorall}.  Then, 
for all  $x\in (x_{M}+R,x_N-R)$
$$\left|\sum_{i=M}^{N}\ep \phi\left(\frac{x-x_i}{\ep\delta}\right)+\ep M-v(x)\right|\leq o_\ep(1)\left(1+\frac{\delta}{R}\right),$$
with $o_\ep(1)$ independent of $R$ and $x$.  
\end{lem}
\begin{proof}
Fix $x\in (x_{M}+R,x_N-R)$, and let $x_{i_0}$ be the closest point among the $x_i$'s to $x$. Then, $x_{i_0-1}<x<x_{i_0+1}$ and by the monotonicity of $v$,
\beq\label{vxi0lem}\ep (i_0-1)=v( x_{i_0-1})\le v(x)\leq v( x_{i_0+1})=\ep(i_0+1).\eeq
By using \eqref{vxi0lem},  estimate \eqref{phiinfinity} and that $\phi\leq 1$, we get
\beqs\begin{split}
&\sum_{i=M}^{N}\ep \phi\left(\frac{x-x_i}{\ep\delta}\right)+\ep M-v(x)\\&=
\sum_{i=M}^{i_0-1}\ep \phi\left(\frac{x-x_i}{\ep\delta}\right)+\ep\phi\left(\frac{x-x_{i_0}}{\ep\delta}\right)+ \sum_{i=i_0+1}^{N}\ep \phi\left(\frac{x-x_i}{\ep\delta}\right)+\ep M-v(x)\\&
\leq \sum_{i=M}^{i_0-1}\ep \left(1+\frac{\ep\delta}{\alpha\pi(x_i-x)}+\frac{K_1\ep^2\delta^2}{(x_i-x)^2}\right)+\ep\\&+
\sum_{i=i_0+1}^{N}\ep\left(\frac{\ep\delta}{\alpha\pi(x_i-x)}+\frac{K_1\ep^2\delta^2}{(x_i-x)^2}\right)+\ep M-\ep (i_0-1)\\&
=\ep\delta  \sum_{i=M\atop i\neq i_0}^{N} \frac{\ep}{\alpha\pi(x_i-x)}+\ep\delta^2 K_1 \sum_{i=M\atop i\neq i_0}^{N}\frac{\ep^2}{(x_i-x)^2}+2\ep\\&
=\ep\delta   \sum_{ i\neq i_0\atop |x_i-x|\le R}\frac{\ep}{\alpha\pi(x_i-x)}+\ep\delta  \sum_{i=M\atop |x_i-x|>R}^{N} \frac{\ep}{\alpha\pi(x_i-x)}
+\ep\delta^2 K_1 \sum_{i=M\atop i\neq i_0}^{N}\frac{\ep^2}{(x_i-x)^2}+2\ep.
\end{split}
\eeqs
We can bound the second term above as follows
\beqs\begin{split}
 \ep \delta\left| \sum_{i=M\atop |x_i-x|>R}^{N} \frac{\ep}{\alpha\pi(x_i-x)}\right|&\leq \ep \delta
 \sum_{i=M\atop |x_i-x|>R}^{N} \frac{\ep}{\alpha\pi|x_i-x|}\leq \frac{\ep\delta(\ep N-\ep M+\ep)}{\alpha\pi R}\\&=
 \frac{\ep\delta(v(x_N)-v(x_M)+\ep)}{\alpha\pi R}\leq \ep(2\|v\|_\infty+\ep) \frac{\delta }{\alpha\pi R}.
\end{split}\eeqs
Therefore, by  Proposition \ref{apprIcorall}, Remark  \ref{vcopntstanti1rem} and  \eqref{i/k^2sum}, we get
\beqs\begin{split}
\sum_{i=M}^{N}\ep \phi\left(\frac{x-x_i}{\ep\delta}\right)+\ep M-v(x)&\le\frac{ \ep\delta}{\alpha}\left( \I^{1,R}[v](\xs)+ O(\ep^{-\frac12})+C\right)
\\&+\ep(2\|v\|_\infty +\ep)\frac{\delta }{\alpha\pi R}+C\ep\delta^2+2\ep\\&
\leq o_\ep(1)\left(1+\frac{\delta}{R}\right).
\end{split}
\eeqs
Similarly, one can prove that 
\beqs
\sum_{i=M}^{N}\ep \phi\left(\frac{x-x_i}{\ep\delta}\right)+\ep M-v(x)\geq o_\ep(1)\left(1+\frac{\delta}{R}\right)\eeqs
and this concludes the proof of the lemma.
\end{proof}
\begin{lem}\label{vapprowhatremains}
Under the assumptions of Lemma \ref{vapproxphisteps},    there exists $C>0$  independent of $\ep$  and $R$ such that 
for all  $x>x_N+R$, 
\beq\label{vapprowhatremainseq1}
\left|\sum_{i=M}^{N}\ep \phi\left(\frac{x-x_i}{\ep\delta}\right)+\ep M-v(x_N)\right|\leq C\ep\left(1+\frac{\delta}{R}\right),
\eeq
and 
for all  $x<x_M-R$, 
\beq\label{vapprowhatremainseq2}
\left|\sum_{i=M}^{N}\ep \phi\left(\frac{x-x_i}{\ep\delta}\right)\right|\leq C\ep\left(1+\frac{\delta}{R}\right).
\eeq
\end{lem}
\begin{proof}
Let $x>x_N+R$, then $x-x_i>R$ for all $i=M,\ldots,N$ and by using that $\phi\le1$, we get
\beqs\begin{split}
\sum_{i=M}^{N}\ep \phi\left(\frac{x-x_i}{\ep\delta}\right)+\ep M&\leq (N+1)\ep=v(x_{N})+\ep.
\end{split}
\eeqs

On the other hand, by \eqref{phiinfinity} and \eqref{i/k^2sum},
\beqs\begin{split}
\sum_{i=M}^{N}\ep \phi\left(\frac{x-x_i}{\ep\delta}\right)+\ep M&\geq \sum_{i=M}^{N}\ep\left(1+\frac{\ep\delta}{\alpha\pi(x_i-x)}
-\frac{K_1\delta^2\ep^2}{(x_i-x)^2}\right)+\ep M \\&
\geq (N+1)\ep-\frac{\ep}{\alpha\pi} (\ep N-\ep M+\ep)\frac{\delta}{R}-C\ep\delta^2
\\&=v(x_N)+\ep -\frac{\ep}{\alpha\pi}(v(x_N)-v(x_M)+\ep)\frac{\delta}{R}-C\ep\delta^2\\&
\ge v(x_N)-C\ep\left(1+\frac{\delta}{R}\right).
\end{split}
\eeqs
This proves \eqref{vapprowhatremainseq1}.

Now, let $x<x_M-R$, then $x-x_i<-R$ for all $i=M,\ldots,N$ and by \eqref{phiinfinity} and \eqref{i/k^2sum},
\beqs\begin{split}
\sum_{i=M}^{N}\ep \phi\left(\frac{x-x_i}{\ep\delta}\right)&\le \sum_{i=M}^{N}\ep\left( \frac{\ep\delta}{\alpha\pi(x_i-x)}+\frac{K_1\delta^2\ep^2}{(x-x_i)^2}\right)\\&
\le \frac{\ep}{\alpha\pi} (\ep N-\ep M+\ep)\frac{\delta}{R}+C\ep\delta^2\\&
=\frac{\ep}{\alpha\pi} ( v(x_N)-v(x_M)+\ep)\frac{\delta}{R}+C\ep\delta^2\\&
\le C\ep\left(1+\frac{\delta}{R}\right).
\end{split}
\eeqs
On the other hand $$\sum_{i=M}^{N}\ep \phi\left(\frac{x-x_i}{\ep\delta}\right)\ge 0.$$
This concludes the proof of   \eqref{vapprowhatremainseq2} and of the lemma.
\end{proof}

\begin{prop}\label{approxpropfinal}
Let $v\in C^{1,1}(\R)$ be non-decreasing and non-constant and 
$x_i$ be defined as in \eqref{xi}. Let $\phi$ be defined by \eqref{phi}. 
Then, for all  $x\in\R$, 
\beq\label{approxpripofinaleq}
\left|\sum_{i=M_\ep}^{N_\ep}\ep \phi\left(\frac{x-x_i}{\ep\delta}\right)+\ep M_\ep-v(x)\right|\leq o_\ep(1),
\eeq
where $o_\ep(1)$ is independent of $x$.
\end{prop}
\begin{proof}
Let $R=R_\ep:=\max\{\delta,c\ep^\frac{1}{2}\}$, with  $c$  given in Proposition \ref{apprIcorall}.  If $x\in (x_{M_\ep}+R, x_{N_\ep}-R)$,  then \eqref{approxpripofinaleq} follows from Lemma \ref{vapproxphisteps}.

Next, let us assume $x>x_{N_\ep}+R$. Then, by \eqref{vapprowhatremainseq1}
\beqs\begin{split}
\left|\sum_{i=M_\ep}^{N_\ep}\ep \phi\left(\frac{x-x_i}{\ep\delta}\right)+\ep M_\ep-v(x)\right|\leq |v(x_{N_\ep})-v(x)|+C\ep.
\end{split}\eeqs
Now, by the monotonicity of $v$, 
\beqs v(x_{N_\ep})-v(x)\leq 0.\eeqs
On the other hand, by the definition \eqref{MepNep} of $N_\ep$, we have
\beqs v(x)- v(x_{N_\ep})=v(x)-\ep N_\ep\leq \sup_\R v- \ep N_\ep\leq2 \ep.\eeqs
This proves \eqref{approxpripofinaleq} when  $x>x_{N_\ep}+R$.
By using \eqref{vapprowhatremainseq2}, one can similarly prove 
\eqref{approxpripofinaleq} when  $x<x_{M_\ep}-R$.

Now, assume $x_{N_\ep}-R\le x\le x_{N_\ep}+R$. Then by \eqref{approxpripofinaleq} applied at $x_{N_\ep}-2R$ and 
$x_{N_\ep}+2R$,  the monotonicity of $\phi$ and the regularity of $v$, we get 
\beqs\begin{split}
&\sum_{i=M_\ep}^{N_\ep}\ep \phi\left(\frac{x-x_i}{\ep\delta}\right)+\ep M_\ep-v(x)\\&\leq 
\sum_{i=M_\ep}^{N_\ep}\ep \phi\left(\frac{x_{N_\ep}+2R-x_i}{\ep\delta}\right)+\ep M_\ep-v(x_{N_\ep}+2R)+O(R)\\&
\le o_\ep(1),
\end{split}\eeqs
and 
\beqs\begin{split}
&\sum_{i=M_\ep}^{N_\ep}\ep \phi\left(\frac{x-x_i}{\ep\delta}\right)+\ep M_\ep-v(x)\\&\geq 
\sum_{i=M_\ep}^{N_\ep}\ep \phi\left(\frac{x_{N_\ep}-2R-x_i}{\ep\delta}\right)+\ep M_\ep-v(x_{N_\ep}-2R)+O(R)\\&
\ge o_\ep(1).
\end{split}\eeqs
This proves  \eqref{approxpripofinaleq} when $x_{N_\ep}-R\le x\le x_{N_\ep}+R$.
Similarly, one can prove   \eqref{approxpripofinaleq} when $x_{M_\ep}-R\le x\le x_{M_\ep}+R$ and the proof of the proposition is completed. 
\end{proof}
We conclude this section with the following lemma that will be used later on.
\begin{lem}\label{sumxialwaysfiniteprop}
Let $v\in C^{1,1}(\R)$ be non-decreasing and non-constant and 
$x_i$ be defined as in \eqref{xi}. Then, there exists $C>0$ such that for all $x\in\R$, 
\beq\label{sumxialwaysfiniteeq}\left| \sum_{i\neq i_0}\frac{\ep}{x_i-x}\right|\leq C.
\eeq
\end{lem}
\begin{proof}
Let us  fix  $\xs\in\R$. 
In what follows we denote by $C$ several positive constants independent of $\ep$ and $\xs$. 
Let $x_{i_0}$ be the closest point to $\xs$ among the $x_i$'s.
Then,  by  \eqref{proplippart1}, $|x_i-\xs|\geq \ep/(2L)$ for $i\neq i_0$. 
 Since $v\in C^{1,1}(\R)$, there exists $C>0$ such that $|\I[v](\xs)|\leq C.$ Moreover,
\beq \begin{split}\label{uno}\I[v](\xs)&=\frac{1}{\pi}\int_{-\infty}^{\xs-\frac{\ep}{2L}}\dfrac{v(x) - v(\xs)}{(x-\xs)^2} dx+\frac{1}{\pi}PV\int_{\xs-\frac{\ep}{2L}}^{\xs+\frac{\ep}{2L}}\dfrac{v(x) - v(\xs)}{(x-\xs)^2} dx
+\frac{1}{\pi}\int_{\xs+\frac{\ep}{2L}}^{+\infty}\dfrac{v(x) - v(\xs)}{(x-\xs)^2} dx\\&
=\frac{1}{\pi}\int_{-\infty}^{\xs-\frac{\ep}{2L}}\dfrac{v(x) - v(\xs)}{(x-\xs)^2} dx+\frac{1}{\pi}\int_{\xs+\frac{\ep}{2L}}^{+\infty}\dfrac{v(x) - v(\xs)}{(x-\xs)^2} dx+O(\ep),
\end{split}\eeq
where $|O(\ep)|\leq C\ep$. If $\xs\in (x_{M_\ep}+\ep/(2L), x_{N_\ep}-\ep/(2L))$, then we write
\beqs \begin{split}
\int_{-\infty}^{\xs-\frac{\ep}{2L}}\dfrac{v(x) - v(\xs)}{(x-\xs)^2} dx&=\int_{-\infty}^{x_{M_\ep}}\dfrac{v(x) - v(\xs)}{(x-\xs)^2} dx
+\sum_{i=M_\ep}^{i_0-2}\int_{x_i}^{x_{i+1}}\dfrac{v(x) - v(\xs)}{(x-\xs)^2} dx+\int_{x_{i_0-1}}^{\xs-\frac{\ep}{2L}}\dfrac{v(x) - v(\xs)}{(x-\xs)^2} dx,
\end{split}\eeqs
where we define $x_{i_0-1}=x_{M_\ep}$ if $i_0=M_\ep.$
By the monotonicity of $v$,
\beqs 0\ge \int_{-\infty}^{x_{M_\ep}}\dfrac{v(x) - v(\xs)}{(x-\xs)^2} dx\geq -\frac{\inf_\R v- v(\xs)}{x_{M_\ep}-\xs}.\eeqs
 As in the proof of Lemma \ref{approxIshortlem}, 
 \beqs\begin{split}\sum_{i=M_\ep}^{i_0-2}\int_{x_i}^{x_{i+1}}\dfrac{v(x) - v(\xs)}{(x-\xs)^2} dx&\leq 
\sum_{i=M_\ep}^{i_0-1}\frac{\ep}{x_i-\xs}+\frac{\ep M_\ep-v(\xs)}{x_{M_\ep}-\xs}-\frac{\ep (i_0-1)-v(\xs)}{x_{i_0-1}-\xs}+
\frac{\ep}{\xs-x_{i_0-1}}\\&
= \sum_{i=M_\ep}^{i_0-1}\frac{\ep}{x_i-\xs}+\frac{v(x_{ M_\ep})-v(\xs)}{x_{M_\ep}-\xs}-\frac{v(x_{i_0-1})-v(\xs)}{x_{i_0-1}-\xs}+
\frac{\ep}{\xs-x_{i_0-1}},\\&
\end{split}\eeqs
and 
\beqs\sum_{i=M_\ep}^{i_0-2}\int_{x_i}^{x_{i+1}}\dfrac{v(x) - v(\xs)}{(x-\xs)^2} dx\geq 
\sum_{i=M_\ep}^{i_0-1}\frac{\ep}{x_i-\xs}+\frac{v(x_{ M_\ep})-v(\xs)}{x_{M_\ep}-\xs}-\frac{v(x_{i_0-1})-v(\xs)}{x_{i_0-1}-\xs}
+\frac{\ep}{\xs-x_{M_\ep}}.
\eeqs
Therefore, by the Lipschitz regularity of $v$ and using that 
$\xs-x_{i_0-1}\ge \ep/(2L)$, we get 
\beqas&& \int_{-\infty}^{x_{M_\ep}}\dfrac{v(x) - v(\xs)}{(x-\xs)^2} dx+\sum_{i=M_\ep}^{i_0-2}\int_{x_i}^{x_{i+1}}\dfrac{v(x) - v(\xs)}{(x-\xs)^2} dx
\\ &\leq&  \sum_{i=M_\ep}^{i_0-1}\frac{\ep}{x_i-\xs}+\frac{v(x_{ M_\ep})-v(\xs)}{x_{M_\ep}-\xs}-\frac{v(x_{i_0-1})-v(\xs)}{x_{i_0-1}-\xs}+
\frac{\ep}{\xs-x_{i_0-1}}\\
&\le&  \sum_{i=M_\ep}^{i_0-1}\frac{\ep}{x_i-\xs}+4L,
\eeqas
and 
\beqas&& \int_{-\infty}^{x_{M_\ep}}\dfrac{v(x) - v(\xs)}{(x-\xs)^2} dx+\sum_{i=M_\ep}^{i_0-2}\int_{x_i}^{x_{i+1}}\dfrac{v(x) - v(\xs)}{(x-\xs)^2} dx
\\ &\geq&  \sum_{i=M_\ep}^{i_0-1}\frac{\ep}{x_i-\xs}+\frac{v(x_{ M_\ep})-\inf_\R v}{x_{M_\ep}-\xs}-\frac{v(x_{i_0-1})-v(\xs)}{x_{i_0-1}-\xs}
\\
&\geq&  \sum_{i=M_\ep}^{i_0-1}\frac{\ep}{x_i-\xs}-5L,
\eeqas
where in the last inequality we  used that $v(x_{ M_\ep})-\inf_\R v\le 2\ep$ and $x_{M_\ep}-\xs\le -\ep/(2L)$.

Next, using that $v(\xs)-v(x_{i_0-1})\leq v(x_{i_0+1})-v(x_{i_0-1})=2\ep$, the monotonicity of $v$ and that $\xs-x_{i_0-1}\ge \ep/(2L)$, we have
\beqs 0\ge \int_{x_{i_0-1}}^{\xs-\frac{\ep}{2L}}\dfrac{v(x) - v(\xs)}{(x-\xs)^2} dx\ge \left(v(x_{i_0-1})-v(\xs)\right)\left(\frac{2L}{\ep}-\frac{1}{\xs-x_{i_0-1}}\right)\ge -C.
\eeqs
We conclude that
\beq\label{due}
\int_{-\infty}^{\xs-\frac{\ep}{2L}}\dfrac{v(x) - v(\xs)}{(x-\xs)^2} dx-C\leq \sum_{i=M_\ep}^{i_0-1}\frac{\ep}{x_i-\xs}\leq \int_{-\infty}^{\xs-\frac{\ep}{2L}}\dfrac{v(x) - v(\xs)}{(x-\xs)^2} dx+C.
\eeq
Similarly, 
\beq\label{tre}
\int_{\xs+\frac{\ep}{2L}}^{+\infty}\dfrac{v(x) - v(\xs)}{(x-\xs)^2} dx-C\le\sum_{i=i_0+1}^{N_\ep }\frac{\ep}{x_i-\xs}
\leq\int_{\xs+\frac{\ep}{2L}}^{+\infty}\dfrac{v(x) - v(\xs)}{(x-\xs)^2} dx+C.
\eeq
From  \eqref{uno}, \eqref{due} and \eqref{tre}, 
\beqs 
\sum_{i\neq i_0}\frac{\ep}{x_i-x}\leq  \int_{-\infty}^{\xs-\frac{\ep}{2L}}\dfrac{v(x) - v(\xs)}{(x-\xs)^2} dx+\int_{\xs+\frac{\ep}{2L}}^{+\infty}\dfrac{v(x) - v(\xs)}{(x-\xs)^2} dx+C\le \I[v](\xs)+C\le C,
\eeqs
and 

\beqs 
\sum_{i\neq i_0}\frac{\ep}{x_i-x}\geq  \int_{-\infty}^{\xs-\frac{\ep}{2L}}\dfrac{v(x) - v(\xs)}{(x-\xs)^2} dx+\int_{\xs+\frac{\ep}{2L}}^{+\infty}\dfrac{v(x) - v(\xs)}{(x-\xs)^2} dx-C\ge \I[v](\xs)-C\ge -C,
\eeqs
which gives \eqref{sumxialwaysfiniteeq}.

If $x\le x_{M_\ep}+\ep/(2L)$, then $x_{i_0}=x_{M_\ep}$ and we write 
\beqs \begin{split}
\int_{\xs+\frac{\ep}{2L}}^{+\infty}\dfrac{v(x) - v(\xs)}{(x-\xs)^2} dx&=\int_{\xs+\frac{\ep}{2L}}^{x_{M_\ep+1}}\dfrac{v(x) - v(\xs)}{(x-\xs)^2} dx
+\sum_{i=M_\ep+1}^{N_\ep-1}\int_{x_i}^{x_{i+1}}\dfrac{v(x) - v(\xs)}{(x-\xs)^2} dx+\int_{x_{N_\ep}}^{+\infty}\dfrac{v(x) - v(\xs)}{(x-\xs)^2} dx.\end{split}\eeqs
If $x\ge x_{N_\ep}-\ep/(2L)$,  then $x_{i_0}=x_{N_\ep}$ and we write 
\beqs \int_{-\infty}^{\xs-\frac{\ep}{2L}}\dfrac{v(x) - v(\xs)}{(x-\xs)^2} dx=
 \int_{-\infty}^{x_{M_\ep}}\dfrac{v(x) - v(\xs)}{(x-\xs)^2} dx+\sum_{i=M_\ep}^{N_\ep-2}\int_{x_i}^{x_{i+1}}\dfrac{v(x) - v(\xs)}{(x-\xs)^2} dx+
 \int_{x_{N_\ep-1}}^{\xs-\frac{\ep}{2L}}\dfrac{v(x) - v(\xs)}{(x-\xs)^2} dx.
\eeqs
Similar computations as before show \eqref{sumxialwaysfiniteeq}. This concludes the proof of the lemma.
\end{proof}

\section{Proof of Theorem \ref{mainthm}}\label{mainthmsec}
We  first show that the functions $u^\ep$ are bounded uniformly in $\ep$. Since $W'(z)=0$ for any $z\in\Z$, integers are stationary solutions to \eqref{uepeq}. Let $k_1,\,k_2\in \Z$ be such that $k_1\leq \inf_\R u_0\leq\sup_\R u_0\leq k_2$. Then by the comparison principle
we have that for any $\ep>0$
$$k_1\leq u^\ep(t,x)\leq k_2\quad\text{for all }(t,x)\in (0,+\infty)\times\R.$$
In particular, 
$u^+:=\limsup^*_{\ep\rightarrow0}u^\epsilon$ is everywhere finite.
We will prove that 
 $u^+$ is a  viscosity  subsolution of \eqref{ubareq} when testing with test functions whose  derivative in $x$   at the maximum  point is different than 0.
Similarly, we can prove that
$u^-:={\liminf_*}_{\ep\rightarrow0}u^\epsilon$ is a supersolution
of \eqref{ubareq} when testing with functions whose derivative in $x$  at the  minimum point  is different  than 0. We will show that this is enough to conclude that 
the following comparison principle holds true: if  $\us$ is the viscosity solution of \eqref{ubareq}, then 
\beq\label{comparisonu+u-}u^+\leq \us\le u^-.\eeq
Since the reverse inequality $u^-\leq u^+$ always holds true, we
conclude that the two functions coincide with $\us$ and that  $u^\ep\to \us$ as $\ep\to 0$, uniformly on compact sets.
We will prove \eqref{comparisonu+u-} in Section \ref{comparisonu+u-sec}.

\medskip

Let $\eta\in C_b^2((0,+\infty)\times\R)$ be such that 
\beq\label{inequalitycomparison}u^+(t,x)-\eta(t,x)<u^+(t_0,x_0)-\eta(t_0,x_0)=0\quad\text{for all } (t,x)\neq(t_0,x_0),\eeq
and assume $\partial_x\eta(t_0,x_0)\neq 0.$ By the comparison principle,   $u^\ep$ is non-decreasing in $x$, and thus 
also $u^+$ is
non-decreasing in $x$.  The monotonicity of $u^+$ and \eqref{inequalitycomparison} imply that $\partial_x\eta(t_0,x_0)\geq0.$
Therefore, we have
\beq\label{deretacond}\partial_x\eta(t_0,x_0)>0.\eeq
The goal is  to show that 
\beq\label{goal}\partial_t\eta(t_0,x_0)\leq c_0 \partial_t\eta_x(t_0,x_0)\,\I[\eta(t_0,\cdot)](x_0).\eeq
Assume by contradiction that 
\beq\label{contradictionhp} \partial_t\eta(t_0,x_0)> c_0 \partial_t\eta_x(t_0,x_0)\,\I[\eta(t_0,\cdot)](x_0).\eeq 
Denote
$$L_0:=\I[\eta(t_0,\cdot)](x_0).$$
By \eqref{deretacond} and \eqref{contradictionhp}, there exists $0<\rho<1$ and $L_1>0$ such that 
\beq\label{etaxposirho}\partial_x\eta(t,x)\ge \frac{\partial_x\eta(t_0,x_0)}{2}>0\quad\text{for all  }(t,x)\in Q_{2\rho,2\rho}(t_0,x_0),\eeq
and 
\beq\label{contractconserho}\partial_t\eta(t,x)\ge c_0\partial_x\eta(t,x)(L_0+L_1)\quad\text{for all  }(t,x)\in Q_{2\rho,2\rho}(t_0,x_0).\eeq
By \eqref{etaxposirho}, $\eta$ is increasing in $x$ over $Q_{2\rho,2\rho}(t_0,x_0)$.
Without loss of generality, we can assume $\eta(t,\cdot) $ to be non-decreasing over $\R$, for $|t-t_0|<2\rho$. 
Indeed, if not, since $\eta>u^+$ outside  $Q_{2\rho,2\rho}(t_0,x_0)$ and $u^+(t,\cdot)$ is non-decreasing over $\R$, we can replace $\eta$  with $\tilde\eta$ such that $\eta=\tilde\eta$ in $Q_{2\rho,2\rho}(t_0,x_0)$, $\tilde\eta(t,\cdot)$ is non-decreasing  over $\R$ for $|t-t_0|<2\rho$ ,  $\tilde\eta\in C_b^2((t_0-2\rho,t_0+2\rho)\times\R)$,  $u^+\leq\tilde\eta\leq\eta$ in $(t_0-2\rho,t_0+2\rho)\times (-K,K)$. 
If we prove \eqref{goal} for $\tilde\eta$, then, since $\partial_t\tilde\eta(t_0,x_0)= \partial_t\eta(t_0,x_0)$, 
 $\partial_x\tilde\eta(t_0,x_0)= \partial_x\eta(t_0,x_0)$ and $\I^{1,K}[\tilde\eta(t_0,\cdot)](x_0)\leq\I^{1,K}[\eta(t_0,\cdot)](x_0)$,  by letting $K$ go to $+\infty$, 
 \eqref{goal}  holds true for $\eta$. Therefore in what follows we assume $\eta$  non-decreasing with respect to $x$ over $\R$ for $|t-t_0|<2\rho$. 

\medskip
We then define the points 
$$x^0_{M_\ep}<\ldots<x^0_i<x^0_{i+1}<\ldots<x^0_{N_\ep}$$
such that 
$$x_i^0:=\inf\{x\,|\,\eta(t_0,x)=\ep i\}\qquad i=M_\ep,\ldots,N_\ep,$$
where
$$M_\ep:= \left\lceil\frac{\inf_\R \eta(t_0,\cdot)+\ep}{\ep}\right\rceil\quad\text{ and } \quad N_\ep:=\left \lfloor\frac{\sup_\R \eta(t_0,\cdot)-\ep}{\ep} \right \rfloor.$$ 
Next, for $0<R<<\rho$ to be determined, 
 let $M_\rho$ be the biggest integer such that $x^0_{M_\rho}$ is  smaller than $x_0-(\rho+R)$ and 
$N_\rho$ is the lowest  integer such that $x^0_{N_\rho}$ is  bigger  than $x_0+(\rho+R)$, that is
\beq\label{initialpartcoleposM} x^0_{M_\rho}<x_0-(\rho+R)\leq x^0_{M_\rho+1}\eeq
and 
\beq\label{initialpartcoleposN} x^0_{N_\rho-1}\le x_0+(\rho+R)< x^0_{N_\rho}.\eeq
Then, we  define the points $x_i(t)$ as follows
\beq\label{etaxi}x_i(t):=\inf\{x\,|\,\eta(t,x)=\ep i\} \qquad\text{for }i=M_\rho,\ldots, N_\rho.\eeq
By definition,
\beq\label{mainxieq}\eta(t,x_i(t))=\ep i,\eeq moreover,
\beq\label{mainxieqinicond}x_i(t_0)=x_i^0.\eeq
\begin{lem}\label{partcilescontrollem}
Let $B_0:=\partial_x\eta(t_0,x_0)/(2\|\partial_t \eta\|_\infty)$ and $x_i(t)$ be defined by  \eqref{etaxi}, $i=M_\rho,\ldots, N_\rho$.
Then, there exists $\ep_0=\ep_0(\rho)$ such that for  $\ep<\ep_0$ and    $R<\rho/3$, $x_i\in C^1(t_0-B_0R,t_0+B_0R)$ and for $|t-t_0|<B_0R$,
\beq\label{xdotbound}|\dot{x}_i(t)|\le B_0^{-1},\eeq
\beq\label{x_ncontrolenq}x_0+\rho<x_{N_\rho}(t)<x_0+\rho+3R,\eeq
\beq\label{x_mcontrolenq}x_0-(\rho+3R) <x_{M_\rho}(t)<x_0-\rho.\eeq
In particular $(t,x_i(t))\in Q_{2\rho,2\rho}(t_0,x_0)$.
\end{lem}
We postpone  the proof  of Lemma \ref{partcilescontrollem} to Section \ref{lemmatasec}.

Now, since by the lemma the $x_i(t)$'s are of class $C^1$ and $(t,x_i(t))\in Q_{2\rho,2\rho}(t_0,x_0)$, we can differentiate in $t$
equation \eqref{mainxieq} 
$$\partial_t\eta(t,x_i(t))+\partial_x\eta(t,x_i(t))\dot x_i(t)=0$$
and use \eqref{contractconserho} to get, for $|t-t_0|<B_0R$, 
\beq\label{dotxiimport}-\dot x_i(t)\geq c_0(L_0+L_1),\quad i=M_\rho,\ldots,N_\rho.\eeq
 \medskip

Next, we are going to construct a supersolution of~ \eqref{uepeq} in $Q_{B_0R,R}(t_0,x_0)$ for~$R<<\rho< 1$.

Since the maximum of $u^+-\eta$ is strict, there exists $\gamma_R>0$ such that 
\beq\label{u-etastrictmax} u^+-\eta\le- 2\gamma_R<0 \quad\text{in } Q_{2\rho,2\rho}(t_0,x_0)\setminus Q_{B_0R,R}(t_0,x_0).\eeq
Then, we define
\beq\label{Phiep}\Phi^\ep(t,x):=\begin{cases}h^\ep(t,x)+\ep M_\ep +\frac{\ep\delta L_1}{\alpha}-\ep\left\lfloor \frac{\gamma_R}{\ep}\right \rfloor&\text{for }(t,x)\in Q_{B_0R,\frac\rho 2}(t_0,x_0)\\
u^\ep(t,x)& \text{outside} \\
\end{cases}
\eeq
where

\beq\label{hfunct}\begin{split}
h^\ep(t,x)&=\sum_{i=M_\rho}^{N_\rho}\ep\left(\phi\left(\frac{x-x_i(t)}{\ep \delta}\right)+\delta\psi\left(\frac{x-x_i(t)}{\ep \delta}\right)\right)\\&+
\sum_{i=M_\ep}^{M_\rho-1}\ep\phi\left(\frac{x-x_i^0}{\ep \delta}\right)+\sum_{i=N_\rho+1}^{N_\ep}\ep\phi\left(\frac{x-x_i^0}{\ep \delta}\right),
\end{split}
\eeq
with  $\phi$ solution of  \eqref{phi} and  $\psi$  solution  of \eqref{psi} with $L=L_0+L_1$.
\begin{rem}
We choose $x_i(t)=x_i^0$ to be constant in time for $i=M_\ep,\ldots, M_\rho-1$ and $i=N_\rho+1,\ldots, N_\ep$, because we cannot bound
 the derivative $\dot{x}_i(t)$ for all $i=M_\ep,\ldots,N_\ep$. This will produce an error $O(R)$ when comparing $\Phi^\ep$ with $\eta$ 
 when $|t-t_0|<B_0R$ and $|x-x_0|\geq\rho-R$, see Lemma \ref{spproxetalem2}.
\end{rem}
\begin{lem}\label{supersolutiponlemma}
There exists $0<R<<\rho$ and $\ep_0=\ep_0(R,\rho)>0$ such that for any $\ep<\ep_0$,  the function $\Phi^\ep$ defined by \eqref{Phiep} satisfies
\beq\label{mainlem1}\Phi^\ep\ge u^\ep\quad\text{ outside }Q_{B_0R,R}(t_0,x_0),\eeq 
\beq\label{mainlem2}\delta\partial_t \Phi^\ep\ge\I [\Phi^\ep]-\displaystyle\frac{1}{\delta}W'\left(\frac{\Phi_\ep}{\ep}\right)\quad\text{in }Q_{B_0R,R}(t_0,x_0),\eeq
and 
\beq\label{mainlem3}\Phi^\ep\le \eta+o_\ep(1)-\ep\left\lfloor \frac{\gamma_R}{\ep}\right \rfloor\quad\text{ in }Q_{B_0R,R}(t_0,x_0).\eeq
\end{lem}

We are now in  position to conclude the proof of Theorem \ref{mainthm}.

By \eqref{mainlem1} and \eqref{mainlem2} and the comparison principle, Proposition \ref{comparisonbounded}, we have
$$u^\ep(t,x)\leq \Phi^\ep(t,x)\quad \text{for all }(t,x)\in Q_{B_0R,R}(t_0,x_0).$$
Passing to the  upper limit as $\ep\to0$ and using \eqref{mainlem3}  and  that $u^+(t_0,x_0)=\eta(t_0,x_0)$, we obtain
$$0\leq-\gamma_R,$$
a contradiction. This concludes the proof of Theorem \ref{mainthm}.
\subsection{Proof of Lemma \ref{supersolutiponlemma}}
We divide the proof of Lemma \ref{supersolutiponlemma} in several steps. We start with the following lemma.

\begin{lem}\label{spproxetalem1} There exists $\ep_0=\ep_0(R,\rho)>0$ such that for any $\ep<\ep_0$ and for any $(t,x)\in Q_{B_0R, \rho-R}(t_0,x_0)$, we have
$$|h^\ep(t,x)+\ep M_\ep-\eta(t,x)|\leq  o_\ep(1).$$ 
\end{lem}

We postpone  the proof  of Lemma  \ref{spproxetalem1} to Section \ref{lemmatasec}.

\medskip
\noindent{\em Proof of \eqref{mainlem1}}.
Outside $Q_{B_0R,\frac\rho2}(t_0,x_0)$, by definition  \eqref{Phiep} of $\Phi^\ep$, $ \Phi^\ep(t,x)=u^\ep(t,x)$.

 Next, by Lemma \ref{spproxetalem1} and \eqref{u-etastrictmax}, for $(t,x)\in Q_{B_0R,\frac\rho2}(t_0,x_0)\setminus Q_{B_0R,R}(t_0,x_0)$, 
\beqs\begin{split} \Phi^\ep(t,x)&=h^\ep(t,x)+\ep M_\ep +\frac{\ep\delta L_1}{\alpha}-\ep\left\lfloor \frac{\gamma_R}{\ep}\right \rfloor\\
&\ge \eta(t,x)+
o_\ep(1)-\ep\left\lfloor \frac{\gamma_R}{\ep}
\right \rfloor\\
&
\geq u^+(t,x)+o_\ep(1)+2\gamma_R-\ep\left\lfloor \frac{\gamma_R}{\ep}\right \rfloor\\&
\ge u^\ep(t,x)
\end{split}\eeqs
for $\ep$ small enough, where in the last inequality we  have used that  $u^+(t,x)\ge u^\ep(t,x)+o_\ep(1)$ and $2\gamma_R-\ep\left\lfloor \frac{\gamma_R}{\ep}\right \rfloor\to\gamma_R>0$ as $\ep\to0$. This concludes the proof of  \eqref{mainlem1}.

\medskip
\noindent{\em Proof of \eqref{mainlem3}}. By Lemma  \ref{spproxetalem1}, for $(t,x)\in  Q_{B_0R,R}(t_0,x_0)$
$$\Phi^\ep(t,x)=h^\ep(t,x)+\ep M_\ep+\frac{\ep\delta L_1}{\alpha}-\ep\left\lfloor \frac{\gamma_R}{\ep}\right \rfloor
\leq \eta(t,x)+o_\ep(1)-\ep\left\lfloor \frac{\gamma_R}{\ep}\right \rfloor,$$
which gives \eqref{mainlem3}.

\medskip
Next, we need some preliminaries results in order to prove \eqref{mainlem2}.
\begin{lem}\label{psismall}
There exists $C>0$ independent of $\ep$ and $\rho$ such that, for any $x\in\R$, $$\left|\sum_{i=M_\rho}^{N_\rho}\ep \delta\psi\left(\frac{x-x_i(t)}{\ep \delta}\right)\right|\leq C\delta.$$
\end{lem}
\begin{proof}
We have,
\beqs \begin{split}\left|\sum_{i=M_\rho}^{N_\rho}\ep \delta\psi\left(\frac{x-x_i(t)}{\ep \delta}\right)\right|&\le \delta \|\psi\|_\infty \ep( N_\rho- M_\rho+1)\\&=\delta \|\psi\|_\infty (\eta(t,x_{N_\rho}(t))- \eta(t,x_{M_\rho}(t))+\ep)\\&\leq  C\delta.
\end{split}
\eeqs

\end{proof}
\begin{lem}\label{spproxetalem2} There exists $\ep_0=\ep_0(R,\rho)>0$ such that for any $\ep<\ep_0$, 
 if $|t-t_0|<B_0R$,  and $|x-x_0|\geq \rho-R$, then 
 $$|h^\ep(t,x)+\ep M_\ep-\eta(t,x)|\leq o_\ep(1)+O(R).$$
\end{lem}

We postpone  the proof  of Lemma  \ref{spproxetalem2} to Section \ref{lemmatasec}.


\begin{cor} \label{coroIphiep}There exists $\ep_0=\ep_0(R,\rho)>0$ such that for any $\ep<\ep_0$, $R<\rho/4$, 
and any $(t,x)\in Q_{B_0R,R}(t_0,x_0)$, we have
\beq\label{eq2cori1uephfinal} \I[\Phi^\ep(t,\cdot)](x)\leq \I[h^\ep(t,\cdot)](x)+o_\ep(1)+\frac{o_R(1)}{\rho}.\eeq
\end{cor}
\begin{proof}
We have

$$\I[\Phi^\ep(t,\cdot)](x)=\I^{1,\frac\rho4}[\Phi^\ep(t,\cdot)](x)+\frac1\pi\int_{\frac\rho4<|y-x|<\rho}\frac{\Phi^\ep(t,y)-\Phi^\ep(t,x)}{(y-x)^2}dy+\I^{2,\rho}[\Phi^\ep(t,\cdot)](x).$$
If $(t,x)\in Q_{B_0R,R}(t_0,x_0)$ and $|y-x|<\rho/4$ then for $R<\rho/4$, $|y-x_0|<\rho/2$, that is $(t,y)\in Q_{B_0R,\frac{\rho}{2}}(t_0,x_0)$. Therefore, by the definition \eqref{Phiep} of $\Phi^\ep$, 
\beq\label{eq2cori1ueph}\I^{1,\frac\rho4}[\Phi^\ep(t,\cdot)](x)=\I^{1,\frac\rho4}[h^\ep(t,\cdot)](x).\eeq
If $(t,x)\in Q_{B_0R,R}(t_0,x_0)$ and $|y-x|>\rho$ then $|y-x_0|>\rho/2$,  therefore $\Phi^\ep(t,y)=u^\ep(t,y)$. Then,
by Lemma  \ref{spproxetalem1}, Lemma \ref{spproxetalem2} and using that $u^\ep\le u^++o_\ep(1)\leq\eta +o_\ep(1)$, we get
\beqs\begin{split}
\I^{2,\rho}[\Phi^\ep(t,\cdot)](x)&=\frac1\pi\int_{|y-x|>\rho}\frac{\Phi^\ep(t,y)-\Phi^\ep(t,x)}{(y-x)^2}dy\\
&=\frac1\pi\int_{|y-x|>\rho}\frac{u^\ep(t,y)-(h^\ep(t,x)+\ep M_\ep+O(\ep)+O(\gamma_R))}{(y-x)^2}dy\\
&\leq \frac1\pi\int_{|y-x|>\rho}\frac{\eta(t,y)-(h^\ep(t,x)+\ep M_\ep)}{(y-x)^2}dy+\frac{o_\ep(1)+O(\gamma_R)}{\rho}\\
&\le  \frac1\pi\int_{|y-x|>\rho}\frac{h^\ep(t,y)-h^\ep(t,x)}{(y-x)^2}dy+\frac{o_\ep(1)+O(\gamma_R)+O(R)}{\rho}.
\end{split}
\eeqs
Therefore,
\beq\label{eq1cori1ueph} \I^{2,\rho}[\Phi^\ep(t,\cdot)](x)\leq \I^{2,\rho}[h^\ep(t,\cdot)](x)+\frac{o_\ep(1)+o_R(1)}{\rho}.
\eeq
Finally, if $\rho/4<|y-x|<\rho$ then either $\Phi^\ep(t,y)=u^\ep(t,y)$ and by  Lemma \ref{spproxetalem1} and Lemma \ref{spproxetalem2}, 
$\Phi^\ep(t,y)\leq h^\ep(t,y)+\ep M_\ep+o_\ep(1)+O(R)$ or $\Phi^\ep(t,y)=h^\ep(t,y)+\ep M_\ep+o_\ep(1)+o_R(1)$. In both cases,
\beq\label{eq3cori1ueph}\int_{\frac\rho4<|y-x|<\rho}\frac{\Phi^\ep(t,y)-\Phi^\ep(t,x)}{(y-x)^2}dy\leq 
\int_{\frac\rho4<|y-x|<\rho}\frac{h^\ep(t,y)-h^\ep(t,x)}{(y-x)^2}dy+\frac{o_\ep(1)+o_R(1)}{\rho}.\eeq
From \eqref{eq2cori1ueph}, \eqref{eq1cori1ueph} and \eqref{eq3cori1ueph}, inequality \eqref{eq2cori1uephfinal} follows.
\end{proof}

\medskip
Now, we are ready to prove \eqref{mainlem2}.

\noindent{\em Proof of \eqref{mainlem2}}.

 Denote 
\beqs \Lambda := \delta\partial_t \Phi^\ep-\I [\Phi^\ep]+\frac{1}{\delta}W'\left(\frac{\Phi_\ep}{\ep}\right).
\eeqs
We want to show that $\Lambda(t,x) \geq 0$ for all $(t,x) \in Q_{B_0R,R}(t_0,x_0).$  
Fix $(\ts,\xs)\in Q_{B_0R,R}(t_0,x_0)$. 
By Corollary \ref{coroIphiep},  
\beq\begin{split} \label{scalingpro+cor}\I[\Phi^\ep(\ts,\cdot)](\xs) &\leq \I[h^\ep(\ts,\cdot)](\xs) + o_\ep(1) + \frac{o_R(1)}{\rho}\\
&=\sum_{i=M_\rho}^{N_\rho}\frac{1}{\delta}\I[\phi](z_i)+\sum_{i=M_\ep}^{M_\rho-1} \frac{1}{\delta}\I[\phi](z_i^0)  + \sum_{i=N_\rho+1}^{N_\ep}\frac{1}{\delta}\I[\phi](z_i^0)\\
&+\sum_{i=M_\rho}^{N_\rho}\I[\psi](z_i)+o_\ep(1) + \frac{o_R(1)}{\rho},
\end{split}
\eeq
where we denote $z_i^0 = (\xs-x_i^0)/(\ep \delta)$ and $z_i = (\xs-x_i(\ts))/(\ep \delta)$. Let $i_0$ be such that $x_{i_0}(\ts)$ is the closest point to $\xs$. Since $(\ts,\xs)\in Q_{B_0R,R}(t_0,x_0)$, by  Lemma \ref{partcilescontrollem} we   have $M_\rho < i_0 < N_\rho$. If  $\xs = x_{i_0} + \ep \gamma$, then   \eqref{proplippart2} and \eqref{etaxposirho} imply that $|\gamma|\leq 2/\partial_x\eta(t_0,x_0)$.
Note that $z_{i_0}=\gamma/\delta$. 
By \eqref{scalingpro+cor}, we have 
\beqs \begin{split}
\Lambda(\ts,\xs) &= \delta\partial_t \Phi^\ep(\ts,\xs)-\I[\Phi^\ep(\ts,\cdot)](\xs)+\frac{1}{\delta}W'\left(\frac{\Phi_\ep(\ts,\xs)}{\ep}\right)\\
&\geq \sum_{i=M_\rho}^{N_\rho} \left [-\dot{x}_i(\ts)\phi'(z_i)-\delta \dot{x}_i(\ts) \psi'(z_i) \right] + o_\ep(1)+\frac{o_R(1)}{\rho}\\
&-\sum_{i=M_\rho}^{N_\rho}\frac{1}{\delta}\I[\phi](z_i)-\sum_{i=M_\ep}^{M_\rho-1} \frac{1}{\delta}\I[\phi](z_i^0) 
- \sum_{i=N_\rho+1}^{N_\ep}\frac{1}{\delta}\I[\phi](z_i^0) - \sum_{i=M_\rho}^{N_\rho}\I[\psi](z_i)\\
&+\frac{1}{\delta}W'\left (  \sum_{i=M_\rho}^{N_\rho} [\phi(z_i) +\delta \psi(z_i)]+ \sum_{i=M_\ep}^{M_\rho-1} \phi(z_i^0)+\sum_{i=N_\rho+1}^{N_\ep} \phi(z_i^0)  + \frac{\delta L_1}{\alpha} \right ),
\end{split}
\eeqs
where we have used the periodicity of $W'$ in the last term. Let us denote
\beq \label{E0}E_0 := o_\ep(1) + \frac{o_R(1)}{\rho}, \eeq
and $$\tilde\phi(z):=\phi(z)-H(z),$$ where $H$ is the Heaviside function. 
Then, by \eqref{dotxiimport}, \eqref{phi}, the periodicity of $W'$ and making a  Taylor expansion of $W'$ around $\phi(z_{i_0})$, we obtain
\beqs \begin{split}
\Lambda(\ts,\xs) &\geq c_0(L_0+L_1)\phi'(z_{i_0}) \\
&+\frac{1}{\delta}\left ( -W'(\phi(z_{i_0})) - \sum_{\stackrel{i=M_\rho}{i\not=i_0}}^{N_\rho} W'(\tilde{\phi}(z_i)) - \sum_{i=M_\ep}^{M_\rho-1} W'(\tilde{\phi}(z_i^0)) - \sum_{i=N_\rho+1}^{N_\ep} W'(\tilde{\phi}(z_i^0)) \right )\\
&-\sum_{\stackrel{i=M_\rho}{i\not=i_0}}^{N_\rho} \I[\psi](z_i) -\I[\psi](z_{i_0})+\frac{1}{\delta}W'(\phi(z_{i_0})) \\
&+\frac{1}{\delta} W''(\tilde{\phi}(z_{i_0})) \left (\sum_{\stackrel{i=M_\rho}{i\not=i_0}}^{N_\rho} [\tilde{\phi}(z_i) +\delta \psi(z_i)]+
\delta \psi(z_{i_0}) + \sum_{i=M_\ep}^{M_\rho-1}\tilde{\phi}(z_i^0) + \sum_{i=N_\rho+1}^{N_\ep} \tilde{\phi}(z_i^0) + \frac{\delta L_1}{\alpha} \right ) \\
&+E_0+E_1+ E_2,\\
\end{split} \eeqs
where we define $E_1$  as follows
\beqs
E_1 := -\sum_{\stackrel{i=M_\rho}{i\not=i_0}}^{N_\rho} \dot{x}_i(\ts)\phi'(z_i) - \delta \sum_{\stackrel{i=M_\rho}{i\not=i_0}}^{N_\rho}\dot{x}_i(\ts) \psi'(z_i) -\delta\dot{x}_{i_0}(\ts)\psi'(z_{i_0}) ,
\eeqs
and  $E_2$ as the error from the Taylor expansion,
\beqs
E_2 := \frac{1}{\delta}O\left (\sum_{\stackrel{i=M_\rho}{i\not=i_0}}^{N_\rho}[ \tilde{\phi}(z_i) + \delta \psi(z_i)]+ \delta \psi(z_{i_0}) + \sum_{i=M_\ep}^{M_\rho-1}\tilde{\phi}(z_i^0) + \sum_{i=N_\rho+1}^{N_\ep} \tilde{\phi}(z_i^0) + \frac{\delta L_1}{\alpha} \right )^2.
\eeqs
Making a  Taylor expansion of $W'$ around 0, using that $W'(0)=0$ and rearranging the terms, we obtain
\beqs \begin{split}
\Lambda(\ts,\xs) &\geq c_0(L_0+L_1)\phi'(z_{i_0}) -\I[\psi](z_{i_0}) + W''(\phi(z_{i_0}))\psi(z_{i_0})\\
&+\frac{1}{\delta}\left (-W''(0)\sum_{\stackrel{i=M_\rho}{i\not=i_0}}^{N_\rho} \tilde{\phi}(z_i) - W''(0)\sum_{i=M_\ep}^{M_\rho-1} \tilde{\phi}(z_i^0)-W''(0)\sum_{i=N_\rho+1}^{N_\rho} \tilde{\phi}(z_i^0)\right ) \\&-\sum_{\stackrel{i=M_\rho}{i\not=i_0}}^{N_\rho} \I[\psi](z_i)\\
&+\frac{1}{\delta} W''(\tilde{\phi}(z_{i_0})) \left (\sum_{\stackrel{i=M_\rho}{i\not=i_0}}^{N_\rho} [\tilde{\phi}(z_i) + \delta\psi(z_i)] + \sum_{i=M_\ep}^{M_\rho-1}\tilde{\phi}(z_i^0) + \sum_{i=N_\rho+1}^{N_\ep} \tilde{\phi}(z_i^0) + \frac{\delta L_1}{\alpha} \right )\\
&+ E_0 + E_1 + E_2 + E_3,
\end{split} \eeqs
where $E_3$ is defined by 
\beqs
E_3 := \frac{1}{\delta}\sum_{\stackrel{i=M_\rho}{i\not=i_0}}^{N_\rho} O(\tilde{\phi}(z_i))^2 + \frac{1}{\delta}\sum_{i=M_\ep}^{M_\rho-1} O(\tilde{\phi}(z_i^0))^2 + \frac{1}{\delta}\sum_{i=N_\rho+1}^{N_\ep} O(\tilde{\phi}(z_i^0))^2.
\eeqs
Since $\psi$ solves \eqref{psi} with $L=L_0+L_1$, we have  that 
$$c_0(L_0+L_1)\phi'(z_{i_0}) -\I[\psi](z_{i_0}) + W''(\phi(z_{i_0}))\psi(z_{i_0})=-\frac{L_0+L_1}{\alpha}(W''(\phi(z_{i_0})) - W''(0)).$$
Therefore,  
\beqs \begin{split}
\Lambda(\ts,\xs) &\geq -\frac{L_0+L_1}{\alpha}(W''(\phi(z_{i_0})) - W''(0)) \\
&+(W''(\phi(z_{i_0})) - W''(0))\left ( \frac{1}{\delta}\sum_{\stackrel{i=M_\rho}{i\not=i_0}}^{N_\rho} \tilde{\phi}(z_i) +  \frac{1}{\delta}\sum_{i=M_\ep}^{M_\rho-1}\tilde{\phi}(z_i^0) +  \frac{1}{\delta}\sum_{i=N_\rho+1}^{N_\ep} \tilde{\phi}(z_i^0) \right )\\
&+W''(\tilde{\phi}(z_{i_0}))\frac{L_1}{\alpha} + W''(\tilde{\phi}(z_{i_0})) \sum_{\stackrel{i=M_\rho}{i\not=i_0}}^{N_\rho} \psi(z_i) -\sum_{\stackrel{i=M_\rho}{i\not=i_0}}^{N_\rho} \I[\psi](z_i)\\
&+ E_0 + E_1 + E_2 + E_3.
\end{split} \eeqs
Rearranging the terms and recalling that $\alpha=W''(0)$, we finally get 
\beq \begin{split}\label{Lambdafinal}
\Lambda(\ts,\xs) &\ge (W''(\phi(z_{i_0})) - W''(0))\left ( \frac{1}{\delta}\sum_{\stackrel{i=M_\rho}{i\not=i_0}}^{N_\rho} \tilde{\phi}(z_i) +  \frac{1}{\delta}\sum_{i=M_\ep}^{M_\rho-1}\tilde{\phi}(z_i^0) +  \frac{1}{\delta}\sum_{i=N_\rho+1}^{N_\ep} \tilde{\phi}(z_i^0) - \frac{L_0}{\alpha} \right )\\
&+L_1 + E_0 + E_1 + E_2 + E_3 + E_4, 
\end{split} \eeq
where $E_4$ is given by 
\beq\label{E4}
E_4 :=  W''(\tilde{\phi}(z_{i_0})) \sum_{\stackrel{i=M_\rho}{i\not=i_0}}^{N_\rho} \psi(z_i) -\sum_{\stackrel{i=M_\rho}{i\not=i_0}}^{N_\rho} \I[\psi](z_i).
\eeq
Next, for fixed $L_1>0$, we are going to show that all the other terms on the right-hand side of  
\eqref{Lambdafinal} are small.  Recall that 
\beqs
L_0=\I[\eta(t_0,\cdot)](x_0)  = \I^{1,\rho}[\eta(t_0,\cdot)](x_0) + \I^{2,\rho}[\eta(t_0,\cdot)](x_0).
\eeqs

\medskip

\begin{lem} \label{estclaims} 
We have,
\beq\label{estclaims1}
(W''(\phi(z_{i_0})) - W''(0))\left ( \frac{1}{\delta}\sum_{\stackrel{i=M_\rho}{i\not=i_0}}^{N_\rho} \tilde{\phi}(z_i) - \frac{ 1}{\alpha}\I^{1,\rho}[\eta(t_0,\cdot)](x_0)\right ) = o_\ep(1)+ o_R(1) + o_\rho(1) + O\left ( \frac{R}{\rho} \right ) ,
\eeq
and
\beq\label{estclaims2}
\frac{1}{\delta}\sum_{i=M_\ep}^{M_\rho-1}\tilde{\phi}(z_i^0) +  \frac{1}{\delta}\sum_{i=N_\rho+1}^{N_\ep} \tilde{\phi}(z_i^0) 
-   \frac{ 1}{\alpha}\I^{2,\rho}[\eta(t_0,\cdot)](x_0) =  o_\ep(1)  + o_\rho(1) + O\left ( \frac{R}{\rho} \right ) .
\eeq
\end{lem}
\begin{proof}
Let us prove \eqref{estclaims1}.
By \eqref{proplippart1}, for $i\not= i_0$, and $\ep$  (thus $\delta$) small enough
\beqs
|z_i| = \left | \frac{\xs-x_i(\ts)}{\ep \delta} \right | \geq \frac{L^{-1}}{2\delta} \geq 1.
\eeqs 
Then, by \eqref{phiinfinity}, for $i\not= i_0$,
\beqs
\left | \tilde{\phi}(z_i) + \frac{\ep \delta}{\alpha \pi (\xs - x_i(\ts))} \right | \leq \frac{K_1 \ep^2 \delta^2}{(\xs-x_i(\ts))^2},
\eeqs
which implies that 
\beqs
\Gamma_1 - \Gamma_2 \leq \sum_{\stackrel{i=M_\rho}{i\not=i_0}}^{N_\rho} \frac{\tilde{\phi}(z_i)}{\delta} - \frac{1}{\alpha}
\I^{1,\rho}[\eta(t_0,\cdot)](x_0) \leq \Gamma_1 + \Gamma_2,
\eeqs
where $\Gamma_1$ and $\Gamma_2$ are respectively defined by
\beqs 
\Gamma_1 := \frac{1}{\alpha}\left (\frac{1}{\pi}\sum_{\stackrel{i=M_\rho}{i\not=i_0}}^{N_\rho} \frac{\ep}{x_i(\ts)-\xs} - \I^{1,\rho}[\eta(t_0,\cdot)](x_0)\right ) \quad \text{and} \quad
\Gamma_2 := K_1\sum_{\stackrel{i=M_\rho}{i\not=i_0}}^{N_\rho} \frac{\ep^2 \delta}{(x_i-\xs(\ts))^2}.
\eeqs
Since $(\ts,\xs)\in Q_{B_0R,R}(t_0,x_0)$, by Lemma \ref{partcilescontrollem} we have that $x_{N_\rho}(\ts) - \xs > x_0+\rho-\xs>\rho-R$ and $\xs-x_{M_\rho}(\ts) >\xs-x_0+\rho> \rho - R.$
Then, 
\beqs \label{changeradiuspoint} \begin{split}
\sum_{\stackrel{i=M_\rho}{i\not=i_0}}^{N_\rho} \frac{\ep}{x_i(\ts)-\xs} &= \sum_{\stackrel{i=M_\rho}{\substack{i\not=i_0\\ |x_i(\ts)-\xs|\leq \rho-R}}}^{N_\rho} \frac{\ep}{x_i(\ts)-\xs} + \sum_{\stackrel{i=M_\rho}{\substack{ |x_i(\ts)-\xs|> \rho-R}}}^{N_\rho} \frac{\ep}{x_i(\ts)-\xs}\\
&=\sum_{\stackrel{i\not=i_0}{|x_i(\ts)-\xs|\leq \rho-R}} \frac{\ep}{x_i(\ts)-\xs} + \sum_{\stackrel{i=M_\rho}{|x_i(\ts)-\xs|> \rho-R}}^{N_\rho} \frac{\ep}{x_i(\ts)-\xs}.
\end{split} \eeqs
Notice that
\beq \begin{split}\label{allultimo}
\I^{1,\rho}[\eta(\ts,\cdot)](\xs) - \I^{1,\rho}[\eta(t_0,\cdot)](x_0) = o_R(1),\\   
\I^{1,\rho}[\eta(\ts,\cdot)](\xs) - \I^{1,\rho-R}[\eta(\ts,\cdot)](\xs) =o_R(1).\end{split} \eeq
By \eqref{allultimo} and Proposition \ref{apprIcorall}, we have 
\beqs \begin{split}
\sum_{\stackrel{i\not=i_0}{|x_i(\ts)-\xs|\leq \rho-R}} \frac{\ep}{x_i(\ts)-\xs} &= \I^{1,\rho-R}[\eta(\ts,\cdot)](\xs) + o_\ep(1) + o_\rho(1) + O(\gamma)\\
&=\I^{1,\rho}[\eta(t_0,\cdot)](x_0)   + o_R(1) + o_\ep(1) + o_\rho(1) + O(\gamma).
\end{split} \eeqs
Next, let $n$ be the number of points $x_i(\ts)$, $i=M_\rho,\ldots,N_\rho$, such that $|x_i(\ts)-\xs|> \rho-R$.
Since $|\xs-x_0|<R$ and by Lemma \ref{partcilescontrollem} $x_i(\ts)\in (x_0-(\rho+3R),x_0+\rho+3R)$, such points must belong 
to the set $\{\rho-2R<|x-x_0|<\rho+3R\}$ whose length is $10R$. Therefore, by  \eqref{proplippart1}, $n\le CR/\ep$.
Then, 
\beqs
\left|\sum_{\stackrel{i=M_\rho}{|x_i(\ts)-\xs|> \rho-R}}^{N_\rho} \frac{\ep}{x_i(\ts)-\xs} \right|\leq \sum_{\stackrel{i=M_\rho}{|x_i(\ts)-\xs|> \rho-R}}^{N_\rho} \frac{\ep}{\rho-R} = \frac{\ep}{\rho-R}n \leq \frac{\ep}{\rho-R}\cdot \frac{CR}{\ep} = O\left (\frac{R}{\rho} \right ).
\eeqs
We conclude  that 
\beqs
\Gamma_1 = o_\ep(1) + o_R(1)  + o_\rho(1) + O(\gamma)  + O\left ( \frac{R}{\rho} \right ).
\eeqs
Since in addition, by \eqref{i/k^2sum},  $\Gamma_2 =O(\delta)$, we have proven that 
\beq\label{lemclaim1estiphii1rho}
 \sum_{\stackrel{i=M_\rho}{i\not=i_0}}^{N_\rho} \frac{\tilde{\phi}(z_i)}{\delta} 
 - \frac{1}{\alpha}\I^{1,\rho}[\eta(t_0,\cdot)](x_0)=o_\ep(1) + o_R(1) + o_\rho(1) + O(\gamma)  + O\left ( \frac{R}{\rho} \right ).
\eeq
Notice that $O(\gamma)  $ is not necessarily small. 
Next, we consider two cases.

{\em Case 1: $|\gamma|<\delta$.} Then, $O(\gamma)=o_\ep(1)$ and 
\beqs\begin{split}
&|W''(\tilde{\phi}(z_{i_0})) - W''(0)| \left ( \sum_{\stackrel{i=M_\rho}{i\not=i_0}}^{N_\rho} \frac{\tilde{\phi}(z_i)}{\delta} -  \frac{1}{\alpha}\I^{1,\rho}[\eta(t_0,\cdot)](x_0)\right )\\& \leq 2\|W''\|_{\infty} \left(o_\ep(1) + o_R(1) + o_\rho(1) +O\left ( \frac{R}{\rho} \right )\right ),
\end{split}\eeqs
and \eqref{estclaims1} is proven.

{\em Case 2: $|\gamma|\geq \delta$.} By  \eqref{phiinfinity}, and using the fact that $z_{i_0}= \gamma/\delta$, we have
$$\left|\tilde{\phi}(z_{i_0})+\frac{\delta}{\alpha\pi\gamma}\right|\leq K_1\frac{\delta^2}{\gamma^2},$$
which implies that 
\beqs
|W''(\tilde{\phi}(z_{i_0})) - W''(0)| \leq |W'''(0)||\tilde{\phi}(z_{i_0})| + O(\tilde{\phi}(z_{i_0}))^2 \leq C\left (\frac{\delta}{|\gamma|} + \frac{\delta^2}{\gamma^2}\right)\leq C\frac{\delta}{|\gamma|}.
\eeqs
Hence, it follows that 
\beqs \begin{split} &|W''(\tilde{\phi}(z_{i_0})) - W''(0)|  \left ( \sum_{\stackrel{i=M_\rho}{i\not=i_0}}^{N_\rho} \frac{\tilde{\phi}(z_i)}{\delta}
 -  \frac{1}{\alpha}\I^{1,\rho}[\eta(t_0,\cdot)](x_0) \right )\\
&\leq  C\frac{\delta}{|\gamma|}\left (o_\ep(1) + o_R(1) + o_\rho(1) + O(\gamma) + O\left ( \frac{R}{\rho} \right ) \right )\\
&\leq o_\ep(1) + o_R(1) + o_\rho(1) + O\left ( \frac{R}{\rho} \right ).
\end{split} \eeqs
This completes the proof of \eqref{estclaims1}.

\medskip
Let us now turn to the proof of \eqref{estclaims2}.
As before, by  \eqref{phiinfinity}, for $i=M_\ep,\ldots, M_\rho-1$ and $i=N_\rho+1,\ldots,N_\ep$,
\beqs
\left | \tilde{\phi}(z_i^0) + \frac{\ep \delta}{\alpha \pi (\xs - x_i^0)} \right | \leq \frac{K_1 \ep^2 \delta^2}{(\xs-x_i^0)^2}.
\eeqs
Hence, we obtain
\beq \begin{split}\label{combinetet1}
\frac{1}{\delta}\sum_{i=M_\ep}^{M_\rho-1}\tilde{\phi}(z_i^0) +  \frac{1}{\delta}\sum_{i=N_\rho+1}^{N_\ep} \tilde{\phi}(z_i^0) &\leq \frac{1}{\alpha \pi} \sum_{i=M_\ep}^{M_\rho-1} \frac{\ep }{x_i^0-\xs} + K_1\sum_{i=M_\ep}^{M_\rho-1} \frac{ \ep^2 \delta}{(x_i^0-\xs)^2}\\
&+\frac{1}{\alpha \pi} \sum_{i=N_\rho+1}^{N_\ep} \frac{\ep }{x_i^0-\xs} + K_1\sum_{i=N_\rho+1}^{N_\ep} \frac{ \ep^2 \delta}{(x_i^0-\xs)^2},
\end{split} \eeq
and
\beq \begin{split}\label{combinetet2}
\frac{1}{\delta}\sum_{i=M_\ep}^{M_\rho-1}\tilde{\phi}(z_i^0) +  \frac{1}{\delta}\sum_{i=N_\rho+1}^{N_\ep} \tilde{\phi}(z_i^0) &\geq \frac{1}{\alpha \pi} \sum_{i=M_\ep}^{M_\rho-1} \frac{\ep }{x_i^0-\xs} -K_1\sum_{i=M_\ep}^{M_\rho-1} \frac{ \ep^2 \delta}{(x_i^0-\xs)^2}\\
&+\frac{1}{\alpha \pi} \sum_{i=N_\rho+1}^{N_\ep} \frac{\ep }{x_i^0-\xs} - K_1\sum_{i=N_\rho+1}^{N_\ep} \frac{ \ep^2 \delta}{(x_i^0-\xs)^2}.
\end{split} \eeq
By \eqref{i/k^2sum}, we have 
\beq \label{term1}
K_1\sum_{i=M_\ep}^{M_\rho-1} \frac{ \ep^2 \delta}{(x_i^0-\xs)^2} + K_1\sum_{i=N_\rho+1}^{N_\ep} \frac{ \ep^2 \delta}{(x_i^0-\xs)^2} = O(\delta).
\eeq
Moreover, since $|\xs-x_0|<R$ and $|x_i^0-x_0|>\rho+R$, for $i=M_\ep,\ldots, M_\rho-1$ and $i=N_\rho+1,\ldots,N_\ep$, it follows that $|x_i^0-\xs|>\rho$, thus,
\beqs \begin{split}
&\frac{1}{\pi}\sum_{i=M_\ep}^{M_\rho-1} \frac{\ep }{x_i^0-\xs} + \frac{1}{\pi}\sum_{i=N_\rho+1}^{N_\ep} \frac{\ep }{x_i^0-\xs}= \frac{1}{\pi}\sum_{\stackrel{i=M_\ep}{|x_i^0-\xs|> \rho}}^{N_\ep} \frac{\ep}{x_i^0 - \xs} - \frac{1}{\pi}\sum_{\stackrel{i=M_\rho}{|x_i^0-\xs|\geq \rho}}^{N_\rho} \frac{\ep}{x_i^0 - \xs}.\\
\end{split} \eeqs
By Lemma \ref{approxIlonglem},
$$ \frac{1}{\pi}\sum_{\stackrel{i=M_\ep}{|x_i^0-\xs|> \rho}}^{N_\ep} \frac{\ep}{x_i^0 - \xs}
=\I^{2,\rho}[\eta(t_0,\cdot)](x_0) + o_\ep(1) + o_\rho(1),
$$
and as before, 
$$\left| \frac{1}{\pi}\sum_{\stackrel{i=M_\rho}{|x_i^0-\xs|\geq \rho}}^{N_\rho} \frac{\ep}{x_i^0 - \xs}\right|\leq C\frac{R}{\rho}.$$
Therefore, 
\beq \label{term2} \begin{split}
&\frac{1}{\pi}\sum_{i=M_\ep}^{M_\rho-1} \frac{\ep }{x_i^0-\xs} + \frac{1}{\pi}\sum_{i=N_\rho+1}^{N_\ep} \frac{\ep }{x_i^0-\xs}=\I^{2,\rho}[\eta(t_0,\cdot)](x_0) + o_\ep(1) + o_\rho(1)+O\left(\frac{R}{\rho}\right).
\end{split} \eeq
Combining \eqref{combinetet1}, \eqref{combinetet2},  \eqref{term1} and \eqref{term2}, yields \eqref{estclaims2}.
This concludes the proof of the lemma.
\end{proof}
Next, we have a control over the remaining errors.
\begin{lem} \label{errorcontrol} For $i\ge 1$, the error $E_i$ satisfies 
\beqs 
E_i= O(\delta).
\eeqs
\end{lem}
We postpone the proof of Lemma \ref{errorcontrol} to Section \ref{lemmatasec}.

\medskip

Let us finally complete the proof of \eqref{mainlem2}. By \eqref{Lambdafinal},  Lemma \ref{estclaims}, Lemma \ref{errorcontrol} and recalling the definition  \eqref{E0} of $E_0$, we  obtain
\beqs
\Lambda(\ts,\xs) \geq L_1 + o_\ep(1)+  o_R(1) + o_\rho(1) + \frac{o_R(1)}{\rho}.
\eeqs
We choose $R<<\rho<<1$ and $\ep_0$ so small that for any $\ep<\ep_0$,
$$ \left|o_\ep(1) +  o_R(1) + o_\rho(1) + \frac{o_R(1)}{\rho}\right|<\frac{L_1}{2}.$$Then, 
$$\Lambda(\ts,\xs) >\frac{L_1}{2}>0.$$
This completes the proof of \eqref{mainlem2}.
\section{Comparison between $u^+$ and $u^-$: proof of \eqref{comparisonu+u-}}\label{comparisonu+u-sec}
Let us consider the approximation of the initial  datum $u_0\in C^{1,1}(\R)$, given by Proposition \ref{approxpropfinal}:
\beq\label{initialdatumappr}\sum_{i=M_\ep}^{N_\ep} \ep \phi\left(\frac{x-x_{0,i}}{\ep\delta}\right)+\ep M_\ep,\eeq
where 
\beq\label{xi0u0}\begin{split} x_{0,i}:=\inf\{x\in\R\,|\, u_0( x)=\ep i\}\qquad i=M_\ep,\ldots,N_\ep,\\
 M_\ep:= \left\lceil\frac{\inf_\R u_0+\ep}{\ep}\right\rceil\quad\text{ and }\quad N_\ep:=\left \lfloor\frac{\sup_\R u_0-\ep}{\ep} \right \rfloor.
 \end{split}\eeq
 Then, for all $x\in\R$,
 \beq\label{approu_0}\left|\sum_{i=M_\ep}^{N_\ep} \ep \phi\left(\frac{x-x_{0,i}}{\ep\delta}\right)+\ep M_\ep-u_0(x)\right|\leq o_\ep(1).\eeq
Let us first show the following asymptotic behavior of $u^+$ and $u^-$.
\begin{lem}\label{u+-atinfinitylem}
For all $t>0$, 
\beq\label{u=limprop-inf}\lim_{x\to -\infty}u^-(t,x)= \lim_{x\to -\infty}u^+(t,x)=\inf_\R u_0,\eeq and 
\beq\label{u=limprop+inf}\lim_{x\to +\infty}u^-(t,x)= \lim_{x\to +\infty}u^+(t,x)=\sup_\R u_0.\eeq
Moreover, for all $x\in\R$,
\beq\label{uplusuminusat0}u^+(0,x)=u^-(0,x)=u_0(x).\eeq
\end{lem}
\begin{proof}
To prove the asymptotic behavior at infinity of $u^+$ and $u^-$, we will construct sub and supersolutions of \eqref{uepeq}.
Let $x_i(t)$, $i=M_\ep,\ldots,N_\ep$ be the solutions of 
\beqs\begin{cases}\dot{x}_i(t)=-c_0L,&t>0\\
x_i(0)=x_{0,i},
\end{cases}
\eeqs
with $L>0$ to be chosen and $x_{0,i},\,M_\ep,\,N_\ep$ defined by \eqref{xi0u0},
that is $x_i(t)=x_{0,i}-c_0Lt$. 
Consider the function
$$h_0^\ep(t,x):=\sum_{i=M_\ep}^{N_\ep}\ep\left(\phi\left(\frac{x-x_i(t)}{\ep \delta}\right)+\delta\psi\left(\frac{x-x_i(t)}{\ep \delta}\right)\right)+\frac{\ep\delta L}{\alpha}+\ep M_\ep+\ep \left\lceil\frac{o_\ep(1)}{\ep}\right\rceil,$$
where  $\phi$ and $\psi$ are respectively  solution of \eqref{phi} and \eqref{psi}. By   the fact that 
\beq\label{psiboundlemm6}\sum_{i=M_\ep}^{N_\ep}\left|\ep\delta\psi\left(\frac{x-x_i(t)}{\ep \delta}\right)\right|\leq \ep(N_\ep-M_\ep+1)\delta\|\psi\|_\infty\le(\sup_\R u_0-
\inf_\R u_0+\ep)\delta\|\psi\|_\infty,\eeq
and \eqref{approu_0}, 
we can choose $o_\ep(1)$ such that 
\beq\label{initialconduepL}u_0(x)\leq h^\ep_0(0,x).\eeq
We are going to show that for $L>0$ large enough, $h^\ep$ is supersolution of \eqref{uepeq}.
Fix $(\ts,\xs)\in(0,+\infty)\times\R$. Let $x_{i_0}(\bar t)$ be   the closest point to $\xs$ and let us  denote 
$z_i:=(\xs-x_i(\ts))/(\ep\delta)$. As in the proof of~Lemma \ref{supersolutiponlemma}, we compute
\beqs \begin{split}
\Lambda(\ts,\xs) &:= \delta\partial_t h_0^\ep(\ts,\xs)-\I[h_0^\ep(\ts,\cdot)](\xs)+\frac{1}{\delta}W'\left(\frac{h_0^\ep(\ts,\xs)}{\ep}\right)\\
&= \sum_{i=M_\ep}^{N_\ep} \left [c_0L\phi'(z_i)+\delta c_0L\psi'(z_i) \right] 
-\sum_{i=M_\ep}^{N_\ep} \frac{1}{\delta}\I[\phi](z_i) - \sum_{i=M_\ep}^{N_\ep}\I[\psi](z_i)\\
&+\frac{1}{\delta}W'\left (\sum_{i=M_\ep}^{N_\ep} [\phi(z_i) +\delta \psi(z_i) ]+ \frac{\delta L}{\alpha} \right ).
\end{split}
\eeqs

By  \eqref{phi} and making a Taylor expansion of $W'$ around $\phi(z_{i_0})$, we get
\beqs \begin{split}
\Lambda(\ts,\xs) &=c_0L\phi'(z_{i_0})-\I[\psi](z_{i_0})+W''(\tilde\phi(z_{i_0}))\psi
(z_{i_0})\\&
- \frac{1}{\delta}W'(\phi(z_{i_0}))- \frac{1}{\delta}\sum_{\stackrel{i=M_\ep}{i\not=i_0}}^{N_\ep}W'(\tilde\phi(z_i))
-\sum_{\stackrel{i=M_\ep}{i\not=i_0}}^{N_\ep}\I[\psi](z_i)\\&
+\frac{1}{\delta}W'(\phi(z_{i_0}))+\frac{1}{\delta}W''(\phi(z_{i_0}))\left (\sum_{\stackrel{i=M_\ep}{i\not=i_0}}^{N_\ep} [\tilde\phi(z_i) +\delta \psi(z_i) ]+ \frac{\delta L}{\alpha} \right )\\&+
\frac{1}{\delta}O\left (\sum_{\stackrel{i=M_\ep}{i\not=i_0}}^{N_\ep} [\tilde\phi(z_i) +\delta \psi(z_i) ]+ \frac{\delta L}{\alpha} \right )^2+
c_0L\sum_{\stackrel{i=M_\ep}{i\not=i_0}}^{N_\ep}\left [\phi'(z_i)+\delta \psi'(z_i) \right] +\delta \psi'(z_{i_0}),
\end{split}
\eeqs
where $\tilde\phi(z)=\phi(z)-H(z)$ with $H$ the Heaviside function. By \eqref{psi} and  making  a Taylor expansion of $W'$ around 0, we obtain
\beqas
\Lambda(\ts,\xs) &=&-\frac{L}{\alpha}(W''(\phi(z_{i_0})) - W''(0)) \\
&+&(W''(\phi(z_{i_0})) - W''(0))\sum_{\stackrel{i=M_\ep}{i\not=i_0}}^{N_\ep} \frac{\tilde\phi(z_i)}{\delta} +W'(\phi(z_{i_0}))\frac{L}{\alpha} \\
&+&\frac{1}{\delta}O\left (\sum_{\stackrel{i=M_\ep}{i\not=i_0}}^{N_\ep} [\tilde\phi(z_i) +\delta \psi(z_i) ]+ \frac{\delta L}{\alpha} \right )^2
+\frac{1}{\delta}\sum_{\stackrel{i=M_\ep}{i\not=i_0}}^{N_\ep}O(\tilde\phi(z_i))^2\\
&+&W''(\phi(z_{i_0}))\sum_{\stackrel{i=M_\ep}{i\not=i_0}}^{N_\ep}\psi(z_i)-\sum_{\stackrel{i=M_\ep}{i\not=i_0}}^{N_\ep}\I[\psi](z_i)
\\&+&c_0L\sum_{\stackrel{i=M_\ep}{i\not=i_0}}^{N_\ep}\left [\phi'(z_i)+\delta \psi'(z_i) \right]+\delta \psi'(z_{i_0}). 
\eeqas
Thus, recalling that $\alpha=W''(0)$,
\beqas
\Lambda(\ts,\xs) &=&(W''(\phi(z_{i_0})) - W''(0))\sum_{\stackrel{i=M_\ep}{i\not=i_0}}^{N_\ep} \frac{\tilde\phi(z_i)}{\delta} +L\\&
+&\frac{1}{\delta}O\left (\sum_{\stackrel{i=M_\ep}{i\not=i_0}}^{N_\ep} [\tilde\phi(z_i) +\delta \psi(z_i) ]+ \frac{\delta L}{\alpha} \right )^2
+\frac{1}{\delta}\sum_{\stackrel{i=M_\ep}{i\not=i_0}}^{N_\ep}O(\tilde\phi(z_i))^2\\
&+&W''(\phi(z_{i_0}))\sum_{\stackrel{i=M_\ep}{i\not=i_0}}^{N_\ep}\psi(z_i)-\sum_{\stackrel{i=M_\ep}{i\not=i_0}}^{N_\ep}\I[\psi](z_i)
\\&+&c_0L\sum_{\stackrel{i=M_\ep}{i\not=i_0}}^{N_\ep}\left [\phi'(z_i)+\delta \psi'(z_i) \right]+\delta \psi'(z_{i_0}). 
\eeqas
Notice that if $x_{i_0}(\ts)$  is the closest point to $\xs$, then $x_{0,i_0}$ is the closest point to $\xs+c_0L\ts$ and 
$\xs-x_{i}(\ts)=(\xs+c_0L\ts)-x_{0,i}$. Then, 
by \eqref{phiinfinity}, Lemma \ref{sumxialwaysfiniteprop} applied to $u_0\in C^{1,1}(\R)$, and \eqref{i/k^2sum},
\beqs  \left|\sum_{\stackrel{i=M_\ep}{i\not=i_0}}^{N_\ep} \frac{\tilde\phi(z_i)}{\delta}\right|
\le \frac{1}{\alpha\pi}\left| \sum_{\stackrel{i=M_\ep}{i\not=i_0}}^{N_\ep}\frac{\ep}{x_{0,i}-(\xs+c_0L\ts)}\right|+K_1\sum_{\stackrel{i=M_\ep}{i\not=i_0}}^{N_\ep}
\frac{\ep^2}{(x_{0,i}-(\xs+c_0L\ts))^2}\leq C.
\eeqs
Moreover, as in the proof of Lemma~\ref{supersolutiponlemma},
\beqas
&&\frac{1}{\delta}O\left (\sum_{\stackrel{i=M_\ep}{i\not=i_0}}^{N_\ep} [\tilde\phi(z_i) +\delta \psi(z_i) ]+ \frac{\delta L}{\alpha} \right )^2+\frac{1}{\delta}\sum_{\stackrel{i=M_\ep}{i\not=i_0}}^{N_\ep}O(\tilde\phi(z_i))^2
\\&&W''(\phi(z_{i_0}))\sum_{\stackrel{i=M_\ep}{i\not=i_0}}^{N_\ep}\psi(z_i)-\sum_{\stackrel{i=M_\ep}{i\not=i_0}}^{N_\ep}\I[\psi](z_i)
\\&&+c_0L\sum_{\stackrel{i=M_\ep}{i\not=i_0}}^{N_\ep}\left [\phi'(z_i)+\delta \psi'(z_i) \right]+\delta \psi'(z_{i_0})\\&&=O(\delta).
\eeqas
We conclude that 
\beqs \Lambda(\ts,\xs)\geq -C+L\ge 0,\eeqs choosing $L>0$  large enough (but independent of $\ep$ and $(\ts,\xs)$).
Since in addition \eqref{initialconduepL} holds true, by the comparison principle,  for all $(t,x)\in(0,+\infty)\times\R$, 
\beq\label{uep-h0epcopm}u^\ep(t,x)\leq h^\ep_0(t,x).\eeq
We will show that the previous inequality implies that for any $\tau>0$ there exists $\tilde K=\tilde K(\tau, T)$ such that for all
 $(t,x)\in [0,T]\times\R$ with $x<\tilde K$,
\beq\label{uniformliminf}u^\ep(t,x)\leq \inf_\R u_0+\tau+o_\ep(1).\eeq 
Fix $\tau>0$.
Since $\lim_{x\to-\infty}u_0(x)=\inf_\R u_0$,  there exists $K\in\R$ such that for all 
$x<K$,
\beqs 
 u_0(x)\leq \inf_\R u_0+\tau.
\eeqs
Given $T>0$, let $\tilde K:=K-c_0 LT$. Then, by \eqref{uep-h0epcopm}, \eqref{approu_0} and \eqref{psiboundlemm6}, for all $(t,x)\in [0,T]\times\R$ such that $x<\tilde K$,
\beqs\begin{split} u^\ep(t,x)&\leq h^\ep_0(t,x)\\&
=
\sum_{i=M_\ep}^{N_\ep}\ep\phi\left(\frac{x+c_0Lt-x_{i,0}}{\ep \delta}\right)+\ep M_\ep+o_\ep(1)
\\&\leq u_0(x+c_0Lt)+o_\ep(1)\\&\le \inf_\R u_0+\tau+o_\ep(1),
\end{split}\eeqs which 
 proves \eqref{uniformliminf}.
On the other hand,  by the comparison principle, 
 $u^\ep\geq \ep\lfloor\inf_\R u_0/\ep\rfloor$.
 Thus, \eqref{u=limprop-inf} follows.
 Similarly one can prove that the limits  \eqref{u=limprop+inf} hold true. 
 
Finally, to prove \eqref{uplusuminusat0}, take a sequence $(t_\ep,x_\ep)\to(0,x)$ as $\ep\to0$.
Then by \eqref{uep-h0epcopm},  \eqref{approu_0} and \eqref{psiboundlemm6},
$$u^\ep(t_\ep,x_\ep)\leq u_0(x_\ep+c_0Lt_\ep)+o_1(\ep)$$
which implies that $u^+(0,x)\leq u_0(x).$ On the other hand, $u^+(0,x)\geq \limsup_{\ep\to0} u^\ep(0,x)=u_0(x).$ We infer that $u^+(0,x)=u_0(x)$. Similarly, $u^-(0,x)=u_0(x)$. This proves  
  \eqref{uplusuminusat0} and concludes the proof of the lemma.
\end{proof}

Now, let $f^\ep$ be the smooth and positive global solution of equation  \eqref{transporteq} with initial datum 
$$f_0^\ep(x)=\frac{1}{\delta}\sum_{i=M_\ep}^{N_\ep} \phi'\left(\frac{x-x_{0,i}}{\ep\delta}\right)>0$$ provided by Theorem \ref{globalesistthmderieq}. 
Notice that by \eqref{phi'infinity},
$$f_0^\ep\in L^p(\R)\quad\text{for all }p\in [1,\infty].$$
Integrating   equation  \eqref{transporteq}  from $a$ to $b$ yields
\beq\label{abintegratedeq}\partial_t\int_{a}^b f^\ep(t,y)\,dy=c_0 f^\ep(t,b)\,\mathcal{H}[f^\ep(t,\cdot)](b)-c_0 f^\ep(t,a)\,\mathcal{H}[f^\ep(t,\cdot)](a).\eeq
Sending $a\to-\infty$ and $b\to+\infty$ and using that $f^\ep>0 $ is vanishing at infinity and $\mathcal{H}[f^\ep(t,\cdot)]\in L^\infty(\R)$, 
we see that $f^\ep(t,\cdot)\in L^1(\R)$ for all $t\ge0$  and 
$$\|f^\ep(t,\cdot)\|_{L^1(\R)}=\|f_0^\ep\|_{L^1(\R)}.$$
Following \cite{bkm}, one can actually show that for all $p\in[1,\infty)$,
$$\|f^\ep(t,\cdot)\|_{L^p(\R)}\leq \|f_0^\ep\|_{L^p(\R)},\quad \|f^\ep(t,\cdot)\|_{L^p(\R)}\leq C_p\|f_0^\ep\|_{L^1(\R)}^\frac{p+1}{2p}t^{-\frac{p-1}{2p}}.$$
By taking $b=x$ and $a=-\infty$ in \eqref{abintegratedeq}, we see that the function 
$$F^\ep(t,x)=\int_{-\infty}^x f^\ep(t,y)\,dy,$$ is a solution of~\eqref{ubareq} with initial datum 
$$\sum_{i=M_\ep}^{N_\ep} \ep \phi\left(\frac{x-x_{0,i}}{\ep\delta}\right).$$
Note that for all $t>0$, 
\beq\label{F-inf}\lim_{x\to-\infty}F^\ep(t,x)=0,\eeq
and by using that $\lim_{x\to+\infty}\phi(x)=1$ and $\lim_{x\to-\infty}\phi(x)=0$,
\beq\label{F+inf}\lim_{x\to+\infty}F^\ep(t,x)=\|f^\ep(t,\cdot)\|_{L^1(\R)}=\|f_0^\ep\|_{L^1(\R)}=\sum_{i=M_\ep}^{N_\ep}\ep=\ep(N_\ep-M_\ep+1).\eeq

Finally, $w^\ep(t,x)=F^\ep(t,x)+\ep M_\ep$ is the unique (and smooth) viscosity   solution of~\eqref{ubareq} with initial datum \eqref{initialdatumappr}. Moreover $\partial_xw^\ep(t,x)=f^\ep(t,x)>0$ for all $(t,x)\in (0,+\infty)\times\R$.
By \eqref{F-inf} and \eqref{F+inf}, we see that 
\beqs\lim_{x\to-\infty}w^\ep(t,x)=\ep M_\ep\quad\text{and}\quad \lim_{x\to+\infty}w^\ep(t,x)=\ep(N_\ep+1).\eeqs
In particular, by Lemma \ref{u+-atinfinitylem}  and the fact that $0\le\ep M_\ep-\inf_\R u_0\leq 2\ep$ and $0\le \sup_\R u_0-\ep N_\ep\le 2\ep$, we have that 
\beq\label{weinfty} \lim_{x\to-\infty}(u^+(t,x)-w^\ep(t,x))\leq 0\quad\text{and} \lim_{x\to+\infty}(u^+(t,x)-w^\ep(t,x))\leq \ep.
\eeq
Moreover, by  \eqref{approu_0} and \eqref{uplusuminusat0},
\beq\label{u+-weptime0}u^+(0,x)-w^\ep(0,x)=u_0(x)-\sum_{i=M_\ep}^{N_\ep} \ep \phi\left(\frac{x-x_{0,i}}{\ep\delta}\right)-\ep M_\ep\leq o_1(\ep).\eeq
We next show that 
\beq\label{comparisonu+wep}u^+(t,x)-w^\ep(t,x)\leq o_\ep(1) \quad\text{for all }(t,x)\in (0,+\infty)\times\R,\eeq 
for  $o_\ep(1)\ge  \ep$ for which such that \eqref{u+-weptime0} holds true.

 Suppose by contradiction that  for some $T>0$, 
\beq\label{comparisonu+wepcontra}\sup_{(t,x)\in(0,T)\times\R}u^+(t,x)-w^\ep(t,x)>o_\ep(1).\eeq
Then, for $\theta>0$  small enough the supremum of the function 
$$u^+(t,x)-w^\ep(t,x)-\frac{\theta}{T-t}-o_\ep(1) $$ is positive and by   \eqref{weinfty} and \eqref{u+-weptime0}
attended at some point $(\ts,\xs)\in (0,T)\times\R$. Then, $\eta(t,x)=w^\ep(t,x)+\frac{\theta}{T-t}+o_\ep(1)$ is a test function for $u^+$ as subsolution with 
 $\partial_x\eta(\ts,\xs)=\partial_x w^\ep(\ts,\xs)>0$, and  by \eqref{goal}, 
\beqs \partial_t w^\ep(\ts,\xs)<\frac{\theta}{(T-\ts)^2}+\partial_t w^\ep(\ts,\xs)\leq c_0\partial_x w^\ep(\ts,\xs)\I[w^\ep(\ts,\cdot)](\xs).\eeqs
On the other hand, since  $w^\ep$ is a smooth solution of~\eqref{ubareq} we have
$$\partial_t w^\ep(\ts,\xs)=c_0\partial_x w^\ep(\ts,\xs)\I[w^\ep(\ts,\cdot)](\xs).$$
We have reached a contradiction. This proves  \eqref{comparisonu+wep}. Moreover, by \eqref{approu_0} and  the comparison principle, 
$ |w^\ep- \us|\leq o_\ep(1)$. Therefore, passing to the limit as $\ep\to0$,
  we finally obtain $u^+\le \us$. Similarly we can prove that $\us\le u^-$. This completes the proof of    \eqref{comparisonu+u-}.

\begin{rem}
Notice that the viscosity solution   $\us=u^+=u^-$ of \eqref{ubareq}
satisfies 
\beqs\lim_{x\to -\infty}\us(t,x)= \inf_\R u_0\quad\text{and}\quad \lim_{x\to +\infty}\us(t,x)=\sup_\R u_0. \eeqs 
\end{rem}

\section{Proofs of Lemmas \ref{partcilescontrollem},  \ref{spproxetalem1},   \ref{spproxetalem2} and
 \ref{errorcontrol}}\label{lemmatasec}
\subsection{Proof of Lemma \ref{partcilescontrollem}}
Recall that by  \eqref{initialpartcoleposN}
\beq\label{xnrho-x0upper}x^0_{N_\rho}-x_0>\rho+R,\eeq
and by \eqref{initialpartcoleposN}, \eqref{etaxposirho} and \eqref{proplippart2},
\beq\label{xnrho-x0upperbis}x^0_{N_\rho}-x_0=(x^0_{N_\rho}- x^0_{N_\rho-1})+(x^0_{N_\rho-1}-x_0) \leq C\ep+\rho+R\le \rho+2R<2\rho,\eeq
for $\ep$ small enough.
Similarly,
$$-(\rho+2R)<x^0_{M_\rho}-x_0<-(\rho+R).$$
In particular, for all $i= M_\rho,...,N_\rho$, $(t_0,x_i^0)\in  Q_{2\rho,2\rho}(t_0,x_0)$.
Then by the regularity of $\eta$ and  \eqref{etaxposirho}, the ODE
\beq\label{odey}
\dot{y}_i(t) = -\frac{\partial_t \eta(t,y_i(t))}{\partial_x \eta(t,y_i(t))}
\eeq
has a unique local solution $y_i(t)$ such that $y_i(t_0)=x_i^0$ which is of class $C^1$ as long as $(t,y_i(t))\in Q_{2\rho,2\rho}(t_0,x_0)$.
Since in addition $\eta(t,y_i(t))=\ep i$ and $\eta$ is strictly increasing in $ Q_{2\rho,2\rho}(t_0,x_0)$, we must have $y_i=x_i$. 
Moreover, as long as $(t,x_i(t))\in Q_{2\rho,2\rho}(t_0,x_0)$, by \eqref{etaxposirho},
\beq\label{xdotlessB0}
|\dot{x}_i(t)| \leq \frac{2\|\partial_t\eta\|_\infty}{\partial_x \eta(t_0,x_0)} = B_0^{-1}.
\eeq
Next, let $-\infty\le t^*\leq +\infty$ be the first time such that 
\beqs  | x_{N_\rho}(t^*)-x^0_{N_\rho}| = R,\eeqs
and 
\beqs \tau:=\min\{2\rho, |t^*-t_0|\}.\eeqs
Then, for $t$ such that $|t-t_0|<\tau$,  
\beq\label{xnt-xnt0bound}|x_{N_\rho}(t) - x^0_{N_\rho}| <  R\eeq
and 
by \eqref{xnrho-x0upperbis},
\beq\label{upperboundxnx0}\begin{split}x_{N_\rho}(t)-x_0&\leq |x_{N_\rho}(t)-x^0_{N_\rho}|+x^0_{N_\rho}-x_0
\leq \rho+3R,
\end{split}\eeq
In particular, $(t,x_{N_\rho}(t))\in Q_{2\rho,2\rho}(t_0,x_0)$ and \eqref{xdotlessB0} holds true. 
Therefore, if $|t^*-t_0|<2\rho$, 
\beqs R = |x_{N_\rho}(t^*) - x_{N_\rho}(t_0)| =\left| \int_{t_0}^{t^*} \dot{x}_{N_\rho}(t)dt\right| \leq \frac{|t^*-t_0|}{B_0},
\eeqs
which implies that $|t^*-t_0| \geq B_0R.$ Hence, for $t$ such that $|t-t_0| < B_0R$, \eqref{upperboundxnx0} holds true which proves the upper bound in \eqref{x_ncontrolenq}.
For the lower bound, for $t$ such that $|t-t_0| < B_0R$, by \eqref{xnt-xnt0bound} and
 \eqref{xnrho-x0upper}, we have
\beqs \begin{split}
x_{N_\rho}(t) - x_0 &\ge  x^0_{N_\rho} - x_0-|x_{N_\rho}(t) - x^0_{N_\rho}|>\rho+R-R=\rho.
\end{split} \eeqs
This completes the proof of \eqref{x_ncontrolenq}. 
Similarly, one can  prove \eqref{x_mcontrolenq}. By the monotonicity of $\eta$, for $i=M_\rho,\ldots,N_\rho$, $x_{M_\rho}(t)<x_i(t)<x_{N_\rho}(t)$
and by \eqref{x_ncontrolenq} and \eqref{x_mcontrolenq}, for $|t-t_0|<B_0 R$, $|x_i(t)-x_0|\leq \rho+3R\leq2\rho$. Therefore, 
$x_i\in C^1(t_0-B_0R,t_0+B_0R)$ and \eqref{xdotlessB0} holds true. This proves \eqref{xdotbound} and concludes the proof of the lemma.

\subsection{ Proof of Lemma \ref{spproxetalem1}} 

We divide the proof of the lemma into three claims.

\medskip

\noindent{\em Claim 1: $\left|\sum_{i=M_\rho}^{N_\rho}\ep\phi\left(\frac{x-x_i(t)}{\ep \delta}\right)+\ep M_\rho-\eta(t,x)\right|\leq o_\ep(1)\left(1+\frac{\delta}{R}\right).$}

{\em Proof of Claim 1.} By Lemma \ref{partcilescontrollem}, if $(t,x)\in Q_{B_0R,\rho-R}(t_0,x_0)$, then 
$x\in (x_{M_\rho}(t)+R, x_{N_\rho}(t)-R)$.  Therefore, Claim 1 immediately follows from Lemma \ref{vapproxphisteps}.

\medskip

\noindent{\em Claim 2: $\left| \sum_{i=M_\ep}^{M_\rho-1}\ep\phi\left(\frac{x-x_i^0}{\ep \delta}\right)+\ep M_\ep-\ep M_\rho\right| \le 
C\ep \left(1+\frac{\delta}{R}\right).$}

{\em Proof of Claim 2.} By \eqref{initialpartcoleposM}, if $(t,x)\in Q_{B_0R,\rho-R}(t_0,x_0)$, then 
$x>x^0_{M_\rho-1}+R$. Claim 2 then follows from \eqref{vapprowhatremainseq1} and the fact that $\ep M_\rho=\eta (t_0, x^0_{M_\rho-1})+\ep$.

\medskip
\noindent{\em Claim 3:   $0\le\sum_{i=N_\rho+1}^{N_\ep}\ep\phi\left(\frac{x-x_i^0}{\ep \delta}\right)\le  C\ep \left(1+\frac{\delta}{R}\right)$.}

{\em Proof of Claim 3.} By  \eqref{initialpartcoleposN}, if $(t,x)\in Q_{B_0R,\rho-R}(t_0,x_0)$, then 
$x<x^0_{N_\rho+1}-R$. Claim 3 then immediately follows from \eqref{vapprowhatremainseq2}.

\medskip

Finally, the lemma is a consequence of Claims 1-3 and Lemma \ref{psismall}, by choosing $\ep$ so small that $\delta/R\leq1$.

\subsection{ Proof of Lemma \ref{spproxetalem2}} 

We first consider the case  $$|x-x_0|>\rho+4R.$$

Let us  assume $x>x_0+\rho+4R$.  Similarly one can prove the lemma for   $x<x_0-(\rho+4R)$.

We divide the proof  into three claims.

\medskip

\noindent{\em Claim 1: $\left|\sum_{i=M_\rho}^{N_\rho}\ep\phi\left(\frac{x-x_i(t)}{\ep \delta}\right)+\ep M_\rho-\ep N_\rho\right|\leq C\ep \left(1+\frac{\delta}{R}\right).$}

{\em Proof of Claim 1.} By Lemma \ref{partcilescontrollem}, if $|t-t_0|<B_0R$ and $x>x_0+\rho+4R$, then 
$x> x_{N_\rho}(t)+R$.  Therefore, Claim 1 immediately follows from \eqref{vapprowhatremainseq1} and the fact that
$\ep N_\rho=\eta(t, x_{N_\rho}(t))$.

\medskip

\noindent{\em Claim 2: $\left| \sum_{i=M_\ep}^{M_\rho-1}\ep\phi\left(\frac{x-x_i^0}{\ep \delta}\right)+\ep M_\ep-\ep M_\rho\right| \le 
C\ep \left(1+\frac{\delta}{R}\right).$}

{\em Proof of Claim 2.}  By \eqref{initialpartcoleposM}, if  $x>x_0+\rho+4R$, then 
$x>x^0_{M_\rho}+R$. Claim 2 then follows from \eqref{vapprowhatremainseq1} and the fact that $\ep M_\rho=\eta(t_0, x^0_{M_\rho-1})+\ep$.

\medskip
\noindent{\em Claim 3:   $\left|\sum_{i=N_\rho+1}^{N_\ep}\ep\phi\left(\frac{x-x_i^0}{\ep \delta}\right)+\ep N_\ep-\eta(t,x)\right|\le  
o_\ep(1)\left(1+\frac{\delta}{R}\right)+O(R).$}

{\em Proof of Claim 3.}  By \eqref{initialpartcoleposN}, if  $x>x_0+\rho+4R$ and in addition $x<x^0_{N_\ep}-R$,  then 
$x\in (x^0_{N_\rho}+R, x^0_{N_\ep}-R)$. Therefore  by Lemma \ref{vapproxphisteps} and the fact that $|t-t_0|<B_0R$, 

\beqs\begin{split}
\left|\sum_{i=N_\rho+1}^{N_\ep}\ep\phi\left(\frac{x-x_i^0}{\ep \delta}\right)+\ep N_\ep-\eta(t,x)\right|&=
\left|\sum_{i=N_\rho}^{N_\ep}\ep\phi\left(\frac{x-x_i^0}{\ep \delta}\right)+\ep N_\ep-\eta(t,x)\right|+o_\ep(1)\\
&\le  
\left|\sum_{i=N_\rho}^{N_\ep}\ep\phi\left(\frac{x-x_i^0}{\ep \delta}\right)+\ep N_\ep-\eta(t_0,x)\right|
\\&+|\eta(t_0,x)-\eta(t,x)|+o_\ep(1)\\&
\leq 
o_\ep(1)\left(1+\frac{\delta}{R}\right)+O(R),
\end{split}
\eeqs
and the claim is proven for  $x_0+\rho+4R<x<x^0_{N_\ep}-R$.

Next, if $x>x^0_{N_\ep}+R$, then by \eqref{vapprowhatremainseq2}, 
\beq\label{claim4lem2part2}
\left|\sum_{i=N_\rho+1}^{N_\ep}\ep\phi\left(\frac{x-x_i^0}{\ep \delta}\right)\right|\leq C\ep \left(1+\frac{\delta}{R}\right).
\eeq
Moreover, since $\ep N_\ep\to\sup_\R\eta(t_0,\cdot)$ as $\ep\to0$,  $|t-t_0|<B_0R$ and $\eta$ is non-decreasing, 
\beq\label{claim4lem2part3}\begin{split}
\eta(t,x)=\eta(t_0,x)+O(R)\leq \sup_\R\eta(t_0,\cdot)+O(R)=\ep N_\ep+o_\ep(1)+O(R),\\
\eta(t,x)\ge \eta(t, x^0_{N_\ep}+R)= \eta(t_0,x^0_{N_\ep})+O(R)= \ep N_\ep+O(R).
\end{split}
\eeq
Estimates  \eqref{claim4lem2part2} and  \eqref{claim4lem2part3}  imply Claim 3 for $x>x^0_{N_\ep}+R$.

Finally, let us assume $x^0_{N_\ep}-R\le x\le x^0_{N_\ep}+R$. Then, by using the monotonicity of $\phi$ and 
that the claim holds true for $x=x^0_{N_\ep}-2R$ and $x= x^0_{N_\ep}+2R$, we get
\beqs\begin{split} &\sum_{i=N_\rho+1}^{N_\ep}\ep\phi\left(\frac{x-x_i^0}{\ep \delta}\right)+\ep N_\ep-\eta(t,x)\\&\le  
\sum_{i=N_\rho+1}^{N_\ep}\ep\phi\left(\frac{x^0_{N_\ep}+2R-x_i^0}{\ep \delta}\right)+\ep N_\ep-\eta(t,x^0_{N_\ep}+2R)
+O(R)\\&\le
o_\ep(1)\left(1+\frac{\delta}{R}\right)+O(R),
\end{split}
\eeqs
and 
\beqs\begin{split} &\sum_{i=N_\rho+1}^{N_\ep}\ep\phi\left(\frac{x-x_i^0}{\ep \delta}\right)+\ep N_\ep-\eta(t,x)\\&\ge  
\sum_{i=N_\rho+1}^{N_\ep}\ep\phi\left(\frac{x^0_{N_\ep}-2R-x_i^0}{\ep \delta}\right)+\ep N_\ep-\eta(t,x^0_{N_\ep}-2R)
+O(R)\\&\ge
o_\ep(1)\left(1+\frac{\delta}{R}\right)+O(R).
\end{split}
\eeqs
This concludes the proof of Claim 3.
\medskip

The lemma for $|x-x_0|>\rho+4R$ is then a consequence of Claims 1-3 and Lemma \ref{psismall},  by choosing $\ep$ so small that $\delta/R\leq1$.

Next, let us consider the case 
$$\rho-R\leq |x-x_0|\leq \rho+4R.$$

Assume without loss of generality that $\rho-R\le x-x_0\le \rho+4R$.
Then,  by using  Lemma \ref{spproxetalem1}  at the point $x_0+\rho-2R$,  Lemma \ref{psismall} and the monotonicity of $\phi$, 
we get
\beqs\begin{split}
h^\ep(t,x)+\ep M_\ep&\geq \sum_{i=M_\rho}^{N_\rho}\ep\phi\left(\frac{x-x_i(t)}{\ep \delta}\right)+
\sum_{i=M_\ep}^{M_\rho-1}\ep\phi\left(\frac{x-x_i^0}{\ep \delta}\right)+\sum_{i=N_\rho+1}^{N_\ep}\ep\phi\left(\frac{x-x_i^0}{\ep \delta}\right)
-C\delta\\&
\geq \sum_{i=M_\rho}^{N_\rho}\ep\phi\left(\frac{x_0+\rho-2R-x_i(t)}{\ep \delta}\right)+
\sum_{i=M_\ep}^{M_\rho-1}\ep\phi\left(\frac{x_0+\rho-2R-x_i^0}{\ep \delta}\right)\\&+\sum_{i=N_\rho+1}^{N_\ep}\ep\phi\left(\frac{x_0+\rho-2R-x_i^0}{\ep \delta}\right)
-C\delta\\&
\ge 
 \eta(t,x_0+\rho-2R)+o_\ep(1)\\&
 \ge  \eta(t,x)+o_\ep(1)+O(R).
\end{split}
\eeqs
Moreover, by using that the lemma holds true  at the point $x_0+\rho+5R$, Lemma \ref{psismall} and  and the monotonicity of $\phi$, we get 
\beqs\begin{split}
h^\ep(t,x)+\ep M_\ep&\leq \sum_{i=M_\rho}^{N_\rho}\ep\phi\left(\frac{x-x_i(t)}{\ep \delta}\right)+
\sum_{i=M_\ep}^{M_\rho-1}\ep\phi\left(\frac{x-x_i^0}{\ep \delta}\right)+\sum_{i=N_\rho+1}^{N_\ep}\ep\phi\left(\frac{x-x_i^0}{\ep \delta}\right)
+C\delta\\&
\leq \sum_{i=M_\rho}^{N_\rho}\ep\phi\left(\frac{x_0+\rho+5R-x_i(t)}{\ep \delta}\right)+
\sum_{i=M_\ep}^{M_\rho-1}\ep\phi\left(\frac{x_0+\rho+5R-x_i^0}{\ep \delta}\right)\\&+\sum_{i=N_\rho+1}^{N_\ep}\ep\phi\left(\frac{x_0+\rho+5R-x_i^0}{\ep \delta}\right)
+C\delta\\&
\le 
 \eta(t,x_0+\rho+5R)+o_\ep(1)+O(R)\\&
 \le  \eta(t,x)+o_\ep(1)+O(R).
\end{split}
\eeqs
This concludes the proof of the lemma.

\subsection{ Proof of Lemma \ref{errorcontrol}.}
By  \eqref{xdotbound}, \eqref{phi'infinity}, \eqref{psi'infinity} and \eqref{i/k^2sum}, we have
\beqs\begin{split}|E_1|&\le B_0^{-1} \left(\sum_{\stackrel{i=M_\rho}{i\not=i_0}}^{N_\rho} \phi'(z_i) + \delta \sum_{\stackrel{i=M_\rho}{i\not=i_0}}^{N_\rho} |\psi'(z_i)| +\delta|\psi'(z_{i_0}) |\right)\\&
\le  B_0^{-1} \left(( K_1 + \delta K_3)\delta^2\sum_{\stackrel{i=M_\rho}{i\not=i_0}}^{N_\rho}\frac{\ep^2}{(x_i-\xs)^2}+\delta\|\psi'\|_\infty\right)\\&
\leq C\delta, 
\end{split}\eeqs
which gives $E_1 = O(\delta)$. 

Now, for $E_2$,  using \eqref{phiinfinity}, \eqref{psiinfinity}, and \eqref{i/k^2sum} we get
\beqs \begin{split}
|E_2|
&\leq \frac{C}{\delta}\left ( \sum_{\stackrel{i=M_\rho}{i\not=i_0}}^{N_\rho} \tilde{\phi}(z_i)^2 + \delta^2 \sum_{\stackrel{i=M_\rho}{i\not=i_0}}^{N_\rho}\psi(z_i)^2 +\delta^2 \psi(z_{i_0})^2+ \sum_{i=M_\ep}^{M_\rho-1}\tilde{\phi}(z_i^0)^2 + \sum_{i=N_\rho+1}^{N_\ep} \tilde{\phi}(z_i^0)^2 + \frac{\delta^2 L_1^2}{\alpha^2} \right )\\
&\leq \frac{C}{\delta}\left ( \sum_{\stackrel{i=M_\ep}{i\not=i_0}}^{N_\ep} \frac{\ep^2 \delta^2}{(x_i-\xs)^2} +\frac{\delta^2 L_1^2}{\alpha^2} 
+\delta^2\|\psi\|^2_\infty\right )\\
&\leq C\delta,
\end{split}
\eeqs
 that is, $E_2 = O(\delta)$.

Similarly, \eqref{phiinfinity} and \eqref{i/k^2sum} imply that $E_3 = O(\delta).$

Finally, consider $E_4$ defined by \eqref{E4}.
From  \eqref{psiinfinity}, \eqref{i/k^2sum}, Proposition \ref{apprIcorall} and  the fact that $|\gamma|\leq 2/\partial_x\eta(t_0,x_0)$,
\beqs 
\left|W''(\tilde{\phi}(z_{i_0})) \sum_{\stackrel{i=M_\rho}{i\not=i_0}}^{N_\rho} \psi(z_i) \right| \leq C\delta\left| \sum_{\stackrel{i=M_\rho}{i\not=i_0}}^{N_\rho} \frac{\ep }{x_i-\xs} \right|+ C \delta^2\left|\sum_{\stackrel{i=M_\rho}{i\not=i_0}}^{N_\rho} \frac{\ep^2 }{(x_i-\xs)^2} \right| \leq C\delta.
\eeqs 
Now, using \eqref{psi} and a Taylor expansion, we get 
\beqs \begin{split}
\I[\psi](z_i)&=W''(\tilde{\phi}(z_i))\psi(z_i)+\frac{L_0+L_1}{\alpha}(W''(\tilde{\phi}(z_i))-W''(0))+c_0(L_0+L_1)\phi'(z_i)\\
&=W''(0)\psi(z_i) + \frac{L_0+L_1}{\alpha}W'''(0)\tilde{\phi}(z_i) +O(\tilde{\phi}(z_i))\psi(z_i) + O(\tilde{\phi})^2 \\&+ c_0(L_0+L_1)\phi'(z_i).
\end{split}
\eeqs
Hence, again from \eqref{phiinfinity}, \eqref{phi'infinity}, \eqref{psiinfinity}, \eqref{i/k^2sum} and Proposition \ref{apprIcorall}, we obtain $$\sum_{\stackrel{i=M_\rho}{i\not=i_0}}^{N_\rho} \I[\psi](z_i) = O(\delta).$$ We infer that 
$E_4 = O(\delta)$. This completes the proof of the lemma.


\begin{thebibliography}{10}


\bibitem{barles}{\sc G. Barles}, Some homogenization results for
non-coercive Hamilton-Jacobi equations. {\em Calculus of
Variations and Partial Differential Equations}, {\bf 30} (2007),
no. 4, 449-466.


\bibitem{bkm}{\sc P. Biler, G. Karch and R. Monneau}, Nonlinear diffusion of dislocation density and self-similar solutions,  {\em Comm. Math. Phys.,} {\bf294} (2010), no. 1, 145-168.


\bibitem{cs}{\sc X. Cabr\'{e} and Y. Sire}, Nonlinear equations for fractional Laplacians II: existence, uniqueness,
and qualitative properties of solutions, {\em Trans. Amer. Math. Soc.},
{\bf 367} (2015), no. 2, 911-941.

\bibitem{csm}{\sc X. Cabr\'{e} and J. Sol\`{a}-Morales}, Layer
solutions in a half-space for boundary reactions, {\em Comm. Pure
Appl. Math.}, {\bf 58} (2005), no. 12, 1678-1732.


\bibitem{caffavaz}{\sc L. Caffarelli and  J.L.  Vazquez,}  Nonlinear Porous Medium Flow with Fractional Potential Pressure, {\em Arch Rational Mech Anal}, {\bf 202} (2011),  537-565.

\bibitem{caffavaz2}{\sc L. Caffarelli and  J.L.  Vazquez,} Regularity of solutions of the fractional porous medium flow with exponent $1/2$,
{\em St. Petersburg Math. J.}, {\bf 27} (2016), 437-460.

\bibitem{caffavazsoria}{\sc L. Caffarelli, F. Soria and  J.L.  Vazquez,} Regularity of solutions of the fractional porous medium flow, {\em Eur. Math. Soc. (JEMS)}, {\bf 15} (2013), no. 5, 1701-1746.


\bibitem{carrillo}{\sc J. A. Carrillo, L.C.F. Ferreira and J. C. Precioso}, 
 A mass-transportation approach to a one dimensional fluid mechanics model with nonlocal velocity, {\em Adv. Math.}, 
 {\bf 231} (2012), no. 1, 306-327.
 
\bibitem{cc}{\sc  A. Castro and D. C\`{o}rdoba}, Global existence, singularities and ill-posedness for a nonlocal flux, {\em  Advances in Math.,} {\bf 219} (2008), 1916-1936.





\bibitem{cddp}{\sc M. Cozzi, J. D\'{a}vila and M. del Pino,} {Long-time asymptotics for evolutionary crystal dislocation models},
{\em Adv. Math.}, {\bf  371} (2020), 107242.


\bibitem{constantin}{\sc P. Constantin, P. D.  Lax and A. Majda}, A simple one-dimensional model for the three dimensional vorticity, {\em Comm. Pure Appl. Math.,} {\bf 38} (1985), 715-724.

\bibitem{Denoual}
{\sc C. Denoual},
Dynamic dislocation modeling by combining Peierls Nabarro and Galerkin methods,
{\em Phys. Rev. B}, {\bf 70} (2004), 024106.


\bibitem{dfv}{\sc S. Dipierro, A. Figalli and E. Valdinoci},
Strongly nonlocal dislocation dynamics in crystals,  
{\em Commun. Partial Differ. Equations}, {\bf 39}
(2014) no. 12, 2351-2387.

\bibitem{dpv}{\sc S. Dipierro, G. Palatucci and E. Valdinoci}, Dislocation 
dynamics in crystals: a macroscopic
theory in a fractional Laplace setting, 
{\em Comm. Math. Phys.}, 
{\bf 333} (2015) no. 2, 1061-1105.

\bibitem{dnpv}{\sc E. Di Nezza, G. Palatucci and E. Valdinoci}, Hitchhiker's guide to fractional Sobolev spaces,
{\em Bull. Sci. Math.}, {\bf 136} (2012), no. 5, 521-573. 

\bibitem{e1} {\sc L. C. Evans}, The perturbed test function
 method for viscosity solutions of nonlinear PDE, {\em Proc. Roy.
 Soc. Edinburgh Sect. A}, {\bf 111} (1989), no. 3-4, 359-375.


\bibitem{fim} {\sc N. Forcadel, C. Imbert and  R.  Monneau}, Homogenization of some particle systems with two-body interactions and of the dislocation dynamics, {\em  Discrete Contin. Dyn. Syst.}, {\bf 23} (2009), no. 3, 785-826.

\bibitem{gm}{\sc A. Garroni and S. M\"{u}ller}, $\Gamma$-limit of
a phase-field model of dislocations {\em SIAM J. Math. Anal.},
{\bf 36} (2005), no. 6, 1943-1964.

\bibitem{gm2}{\sc A. Garroni and S. M\"{u}ller}, A variational
model for dislocations in the line tension limit, {\em Arch.
Ration. Mech. Anal.}, {\bf 181} (2006), 535-578.

\bibitem{gmps}{\sc A. Garroni, P. van Meurs, M. Peletier and L.  Scardia}, Convergence and non-convergence of many-particle evolutions with multiple signs., {\em  Arch. Ration. Mech. Anal.},  {\bf 235} (2020), no. 1, 3-49.


\bibitem{glp}{\sc  A. Garroni, G. Leoni and M. Ponsiglione},  Gradient theory for plasticity via homogenization of discrete dislocations, 
{\em J. Eur. Math. Soc. (JEMS)}, {\bf 12} (2010), no. 5, 1231-1266.


\bibitem{gonzalezmonneau}{\sc M. Gonz\'{a}lez and R. Monneau},
Slow motion of particle systems as a limit of a reaction-diffusion
equation with half-Laplacian in dimension one, {\em Discrete Contin. Dyn. Syst.}, {\bf  32} (2012),
no. 4, 1255-1286.

\bibitem{H} {\sc A. K. Head},
Dislocation group dynamics III. 
Similarity solutions of the continuum approximation,
{\em Phil. Magazine}, {\bf 26}, (1972), 65-72.




\bibitem{hl}{\sc J. R. Hirth and L. Lothe}, Theory of
dislocations, Second Edition. Malabar, Florida: Krieger, 1992.

\bibitem{im}{\sc C. Imbert and R. Monneau}, Homogenization
of first order equations with $u/\epsilon$-periodic Hamiltonians.
Part I: local equations, {\em Archive for Rational Mechanics and
Analysis}, {\bf 187} (2008), no. 1, 49-89.



\bibitem{imr}{\sc C. Imbert, R. Monneau and E. Rouy}, Homogenization
of first order equations with $u/\epsilon$-periodic Hamiltonians.
Part II: application to dislocations dynamics, {\em Communications
in Partial Differential Equations}, {\bf 33} (2008), no. 1-3,
479-516.

\bibitem{jk}{\sc E. R. Jakobsen and K. H. Karlsen}, Continuous
dependence estimates for viscosity solutions of integro-PDEs. {\em
J. Differential Equations}, {\bf 212} (2005), 278-318.




\bibitem{mm}{\sc P. J. P. van Meurs and A. Muntean}, Upscaling of the dynamics of dislocation walls,
{\em Adv. Math. Sci. Appl.}, {\bf 24} (2014), no. 2, 401-414.

\bibitem{mmp}{\sc P. J. P. van Meurs, A. Muntean and M. A. Peletier},  Upscaling of dislocation walls in finite domains,
{\em European J. Appl. Math.}, {\bf 25} (2014), 749-781.

\bibitem{MeMo}{\sc P. J. P.  van Meurs and M.  Morandotti},  Discrete-to-continuum limits of particles with an annihilation rule, {\em SIAM J. Appl. Math.}, {\bf 79} (2019), no. 5, 1940-1966.

\bibitem{mp}{\sc R. Monneau and S. Patrizi}, Homogenization of the Peierls-Nabarro model for dislocation dynamics, {\em 
J. Differential Equations}, {\bf 253} (2012), no. 7, 2064-2105.

\bibitem{mp2}{\sc R. Monneau and S. Patrizi}, Derivation of the Orowan's law from the Peierls-Nabarro model,  {\em Comm.  Partial Differential Equations}, {\bf 37} (2012), no. 10, 1887-1911. 

\bibitem{mora}{\sc M. G. Mora, M. A. Peletier and L. Scardia}, Convergence of Interaction-Driven Evolutions of Dislocations with Wasserstein Dissipation and Slip-Plane Confinement, {\em SIAM J. Math. Anal.}, {\bf 49} (2017), no. 5, 4149-4205.



\bibitem{MBW}
{\sc A.B. Movchan, R. Bullough and J.R. Willis}, Stability of a
dislocation: discrete model, {\em Eur. J. Appl. Math.} {\bf 9}
(1998), 373-396.




\bibitem{n} {\sc F.R.N. Nabarro}, Dislocations in a simple cubic
lattice, {\em Proc. Phys. Soc.}, {\bf 59} (1947), 256-272.



\bibitem{p} {\sc R. Peierls}, The size of a dislocation, {\em Proc. Phys. Soc.}, {\bf
52} (1940), 34-37.






\bibitem{nabarro} {\sc F. R. N. Nabarro}, 
Fifty-year study of the Peierls-Nabarro stress. {\em Mat. Sci. Eng. A},
{\bf 234-236} (1997), 67-76.

\bibitem{psv}{\sc G. Palatucci, O. Savin and  E. Valdinoci}, Local and global minimizers for a variational energy
involving a fractional norm,
{\em  Ann. Mat. Pura Appl.},  {\bf 192} (2013), no. 4, 673-718.


\bibitem{pv2}{\sc S. Patrizi and  E. Valdinoci}, Crystal dislocations with different orientations 
and collisions, {\em Arch. Rational Mech. Anal.}, \textbf{217} (2015), 231-261.

\bibitem{pv}{\sc S. Patrizi and  E. Valdinoci}, Homogenization and Orowan's law for anisotropic fractional operators of any order, 
{\em Nonlinear Analysis: Theory, Methods and Applications}, \textbf{119} (2015), 3-36.

\bibitem{pv4}{\sc S. Patrizi and  E. Valdinoci}, Long-time behavior for crystal dislocation dynamics, {\em Math. Models Methods Appl. Sci.}, {\bf 27} (2017), no. 12, 2185-2228.


\bibitem{pv3}{\sc S. Patrizi and  E. Valdinoci}, Relaxation times for atom dislocations in crystals, 
{\em Calc. Var. Partial Differential Equations}, {\bf 55} (2016), no. 3, 1-44.

\bibitem{sppg}{\sc L. Scardia, R. H. J. Peerlings, M. A. Peletier and M. G. D. Geers,} Mechanics of dislocation pile-ups:
A unification of scaling regimes, {\em J. Mech. Phys. Solids}, {\bf 70}  (2014), 42-61.


\bibitem{sed}{\sc  R. Sedlacek, J. Kratochvi  and E. Werner},  The importance of being curved: bowing dislocations in a continuum description, {\em Philosophical Magazine}, {\bf 83} (2003),  no. 31-34, 3735-3752.

\bibitem{s}{\sc L. Silvestre}, {\em Regularity of the obstacle problem for a fractional power of the Laplace operator}, PhD thesis, University of Texas at Austin (2005).

\bibitem{stein}{\sc E. M. Stein}, {\em Singular integrals and differentiability properties of functions}, Princeton University Press, Princeton, N.J., 1970, Princeton Mathematical Series, No. 30.



\end{thebibliography}
\end{document}